\documentclass[a4paper,10pt]{article}

\usepackage{graphicx}
\usepackage{amsmath}
\usepackage{amsfonts}
\usepackage{amssymb}
\usepackage{bm}

\usepackage{fancyhdr}
\fancypagestyle{plain}{%
  \fancyhead{}
  
  \fancyfoot[C]{- \thepage\ -}}

\newtheorem{proposition}{Proposition}[section]

\newtheorem{lemma}[proposition]{Lemma}
\newtheorem{theorem}[proposition]{Theorem}
\newtheorem{hypothesis}[proposition]{Hypothesis}
\newtheorem{notita}[proposition]{Remark}
\newenvironment{remark}{\begin{notita}\rm}{\hfill$\Box$\\[0.5ex]\end{notita}}

\newenvironment{proof}{\noindent{\em Proof.}}{\hfill$\Box$}

\newcommand{\R}{{\mathbb R}}
\newcommand{\N}{{\mathbb N}}

\newcommand\qin{\quad\hbox{in}\quad}
\newcommand\qon{\quad\hbox{on}\quad}
\newcommand\qan{\quad\hbox{and}\quad}
\newcommand{\ds}{\displaystyle }
\newcommand{\disp}{\displaystyle }

\renewcommand\div{{\mathrm{div}}}

\newcommand{\jump}[1]{[\![ #1 ]\!]}
\newcommand{\triple}[1]{|\!|\!| #1 |\!|\!|_h}

\newcommand{\bnu}{\bfm\nu}
\newcommand{\bfm}[1]{\mbox{\boldmath $#1$}}

\newcommand{\bbe}{\bfm\beta}

\def\curl{\mathop{\mathrm{curl}}\nolimits}

\newcommand{\e}{\bm{e}}

\title{On the discontinuous Galerkin method for solving
boundary value problems for the Helmholtz equation: A priori and a posteriori error
analyses\thanks{This research was partially supported by Spanish Project 
MTM2010-21037, the Direcci\'on de Investigaci\'on of the Universidad Cat\'olica
de la Sant{\'\i}sima Concepci\'on,  FONDECYT project
No. 1130158, BASAL project CMM, Universidad de Chile; Centro de
Investigaci\'on en
Ingenier\'{\i}a Matem\'atica $($CI$^2$MA$)$, Universidad de Concepci\'on; and CONICYT through project Anillo ACT1118 (ANANUM).}}
\author{
{\sc Tom\'as P. Barrios}\thanks{Departamento de Matem\'atica y
F{\'\i}sica Aplicadas, Universidad Cat\'olica de la Sant\'{\i}sima
Concepci\'on, Casilla 297, Concepci\'on, Chile, e-mail: {\tt tomas}@{\tt
ucsc.cl}} \quad
 {\sc Rommel Bustinza}\thanks{CI$^2$MA and Departamento de Ingenier\'\i a
Matem\' atica, Universidad de Concepci\'on, Casilla 160-C, Concepci\' on,
Chile, e-mail: {\tt rbustinz}@{\tt ing-mat.udec.cl}} \quad
and \quad
 {\sc V\'{\i}ctor Dom\'{\i}nguez}\thanks{Departamento de Ingenier\'{\i}a
Matem\'atica e Inform\'atica, Universidad P\'ublica de Navarra, Campus de
Tudela, 31500 - Tudela, Spain, e-mail: {\tt victor.dominguez}@{\tt
unavarra.es}}}

\begin{document}

\date{ } %\today}
\maketitle

\begin{abstract}
We apply the local discontinuous Galerkin (LDG for short)
method to solve a mixed boundary value problems for the  Helmholtz equation in bounded
polygonal domain in 2D. Under some assumptions
on regularity of the solution of an adjoint problem, we prove that: (a) the
corresponding indefinite discrete scheme is well posed; (b) there is 
convergence with the expected convergence rates as long as the
meshsize $h$ is small enough. We give precise information on how small $h$ has
to be in terms of the size of the wavenumber and its distance to the
set of eigenvalues for the same boundary value problem for the Laplacian. We
also present an a posteriori error estimator showing both  the reliability and
efficiency of the estimator complemented with detailed information on the
dependence of the constants on the wavenumber. We finish presenting extensive
numerical experiments which illustrate the theoretical results proven in this paper and
suggest that stability and convergence may occur under less restrictive assumptions
than those taken in the present work. 
\end{abstract}


\vspace{5pt}

%
%%%%%%%%%%%%%%%%%%%%%%%%%%%%%%%%%%%%%%%%%%%%%%%%%%%%%%%%%%%%%%%%%%%%%%%%%%%%%%%%
%%%%%%%%%%
\section{Introduction}\label{section1}

In this paper we are interested in the numerical analysis of discontinuous Galerkin (dG) schemes for the
following model boundary value problem for the Helmholtz equation:
\begin{equation}\label{eq1}
-\Delta u-\omega^2u=f\qin\Omega\,,\quad u=g_D\qon\Gamma_D\,,\qan
\partial_{\bm{\nu}}u=g_N\qon\Gamma_N\,.
\end{equation}
Here, $\Omega$ is a bounded polygonal Lipschitz domain in $\R^{2}$, whose boundary $\Gamma:=\partial\Omega$ is decomposed into two
disjoint sets $\Gamma_D$ and $\Gamma_N$ (that is
$\Gamma=\bar{\Gamma}_D\cup\bar{\Gamma}_N$), with $|\Gamma_D|>0$, and $\omega>0$ denotes a fixed wave number (of corresponding wave length $\lambda=2\pi/\omega$). The right hand side
$f$ is a source term in $L^2(\Omega)$, while the boundary data $g_D\in
H^{1/2}(\Gamma_D)$ and $g_N\in L^2(\Gamma_N)$, and $\bm{\nu}$ represents the
outward
normal unit vector to $\Gamma$.

Now, introducing the auxiliary unknown $\bm{\sigma}:=\nabla u$, problem
$(\ref{eq1})$
is rewritten as: {\it Find $(\bm{\sigma},u)$ in appropriate spaces, such that}
\begin{equation}\label{eq2}
\begin{array}{c}
 \bm{\sigma}=\nabla u\qin\Omega\,,\quad -\div\bm{\sigma} -\omega^2u = f
\qin\Omega\,,\cr\cr
 u=g_D\qon\Gamma_D\,,\qan \bm{\sigma}\cdot\bm{\nu} = g_N\qon\Gamma_N\,.
\end{array}
\end{equation}

The use of classical continuous Galerkin finite element
method provides good phase and amplitude accuracy as long as the mesh is fine
enough with respect to the wave number in the propagation region. This condition
could be too expensive even for moderate wave numbers. Thus, one of the main
concerns in acoustic finite element analysis is the adequacy of the finite
element mesh. Acousticians often use the so-called {\it rule of the thumb} (see
\cite{Ihlenburg-1998}) which prescribes a relation between the minimal number of
elements and the wave number. Indeed, using linear finite elements, an estimate
of the relative error is derived, which is of the form: $e_h\leq C_1\omega
h+C_2\omega^3h^2$, $\omega h<1$, and where the constants $C_1,C_2>0$ are
independent of the wave number $\omega$ and the mesh size $h$. The first term on
the right hand side of the previous estimate denotes the interpolation error
while the second one represents the {\it pollution effect}. It is clear that the
interpolation error would be constant if $\omega h$ remains constant. %, which yields to the {\it rule of the thumb}. 
However, this choice does not guarantee the control of the pollution effect, that increases with $\omega$. This behavior
has became a real challenge for numerical analysts and therefore several
different approaches to deal with the pollution effect have been developed (see
for e.g. \cite{fhf-CMAME2001,fhh-CMAME2003,Ihlenburg-1998} and the references
therein). In \cite{g-APNUM1986,cd-SINUM1998,mw-CMAME1999}, the authors introduce
discontinuous methods for the acoustic problem considering non polygonal basis
functions. They use wave like functions in order to better capture the unique
features corresponding to medium and high frequency regime.

In the context of discontinuous Galerkin methods (see \cite{SIAM-a-2001} for an
overview) there have been several approaches to deal with this problem, too. In
\cite{ghp-M2AN2009, hp-2009} the authors develop a discontinuous Galerkin method for
Helmholtz equation in two dimensions, using local plane waves as trial and test
functions. Our aim here is similar to them, but using piecewise polynomial finite element
basis as starting point to keep the simplicity of coding. There are previous
works in this direction, such as \cite{bdga-CMAME2006,dbga-CMAME2007}, where the
authors propose two very similar DG schemes  for solving \eqref{eq1}, and include several numerical
examples, showing a good behavior of the discrete solution for different values
of the wave number $\omega$, by choosing appropriately the parameters that
define the method as well as taking into account the {\it rule of the thumb} $\omega h<1$.
However, up to the authors' knowledge, they do not establish (from a theoretical point of view) neither the
well-posedness of the scheme nor its corresponding a priori error estimate. On the other hand, in \cite{fw-siam2009} the authors applied a stabilized Interior Penalty Discontinuous Galerkin (IPDG) method to the Helmholtz equation with the first order absorbing boundary conditions (cf. \cite{em-cpam1979}) instead of the mixed boundary conditions considered here. Later, in \cite{fw-mcom2011}, they extend their analysis to the corresponding $hp-$ version, proving optimal convergence with respect to $h$ in the high frequency regime,  when a mesh condition is satisfied.
% but having to deal with the pollution effect. 

Therefore, we are interested in deriving the LDG formulation of $(\ref{eq2})$ and
establish that it has a unique solution as well as the optimal rates of
convergence (under suitable additional regularity on the exact solution), for moderated values of wave number. This
is done following the ideas given in \cite{hpss-M2AN2005}, for solving an
indefinite time-harmonic Maxwell problem, and are based on an earlier work of
Schatz (\cite{schatz-MCOM74}). Then, we develop (in essence) the a priori
error analysis applying duality arguments, so the existence and uniqueness of
the solution of the discrete scheme is proven for $h$ {\it small enough}. The
minimum size of $h$ to enter in the convergence region is shown to depend on
the size of $\omega$, the distance to the closest eigenvalue for the Laplacian
and on the regularity of the adjoint problem (see Hypothesis \ref{hypothesis} below). 

Another important aspect to take into account is the development of a technique that improves the quality of numerical approximation, without performing uniform refinement. One tool is the so called a posteriori error analysis, which gives us a full computable indicator that behaves as the exact error and thus is used in the subsequent adaptive algorithm. For this reason, this indicator is known as a posteriori error estimator. In the context of DG methods, we can refer to  \cite{ains-2007,bhl-2003,bcg-2004,hsw-2007,kp-2003,rw-2003}, where several a posteriori error analyses are developed for standard second order elliptic boundary value problems. Up to the authors' knowledge, there are few work on a posteriori error estimation using DG methods for Helmholtz problem. We can mention \cite{hs-2012}, which includes an a posteriori error estimation and a convergence analysis of the adaptive mechanism when applying an IPDG approach to deal with Helmholtz boundary value problems with the first order absorbing boundary conditions.  Then, as complementary part of the current analysis, we derive a reliable and quasi-efficient  a posteriori error estimate, in order to improve the quality of the numerical approximation, showing the dependence of the constants on the frequency $\omega$. %taking into account also the pollution effect. 

The outline of the paper is as follows: in Section \ref{section2} we introduce
the main elements that let us to derive the LDG formulation associated to
$(\ref{eq2})$, reviewing some basic properties of the discrete scheme. The
well-posedness of the scheme as well as the corresponding a priori error bounds
are stated at the end of this section; the proofs of these results are carried
out in Section \ref{section3}. Next, in Section \ref{section4} we derive a residual a posteriori error estimate, which results to be reliable and locally efficient, up to high order terms.
 Finally, we show several numerical examples in Section \ref{section5}, validating our theoretical
results, and summarize the work presented in this paper, drawing some
conclusions.

%Throughout this paper, $c$ and $C$, with or without subscripts, denote positive constants, independent of the parameters and functions involved, which may take different values at different occurrences. 

%.............................

\begin{hypothesis}\label{hypothesis}
%We first develop here the tools required to prove Theorem \ref{theo:mainLDG}
%and give the proof at the end of this section. Since we follow the ideas given
%in \cite{schatz-MCOM74}, 
 
Given $z\in L^2(\Omega)$ we will assume that
the problem
\begin{equation}
 \label{eq:adjointProblem}
 -\Delta \varphi-\omega^2\varphi = z,\quad\text{on $\Omega$}\, \qquad
\varphi|_{\Gamma_D}=0,\quad
\partial_{\bm{\nu}}\varphi|_{\Gamma_N}=0 .
\end{equation}
admits a unique solution, i.e.,  $-\omega^2$ is not an
eigenvalue for the Laplacian. 

We demand also an extra smoothness properties for ${\cal
L}_{\omega}$:  There exists 
$\varepsilon\in (1/2,1]$ such that the mapping
\begin{equation}
\label{eq:defLw}
\begin{array}{rcl}
{\cal L}_{\omega}:L^2(\Omega)&\longrightarrow&
H^{1+\varepsilon}(\Omega)\\
z&\longmapsto&\varphi
\end{array}
\end{equation}
is continuous $(H^t(\Omega)$ denotes the classical Sobolev space or order $t)$.

%Namely, if two edges $\Gamma_1\subset \Gamma_N,$ $
%\Gamma_2\subset \Gamma_D$ are adjoints of each other,
%then the inner angle between $\Gamma_1$ and $\Gamma_2$ has to be strictly less
%than $\pi$. 

\end{hypothesis} 
We point out that this hypothesis implies in particular that problem \eqref{eq1}
is well posed. Regarding \eqref{eq:defLw},  it  is well known  that such
result
holds for the pure Dirichlet and Neumann
problem. With mixed boundary conditions this hypothesis is satisfied  if one
assumes  some geometric restrictions
on the angles between the sides of $\Gamma_D$ and $\Gamma_N$. 
We refer to \cite[Chapter 4]{Grisvard} or \cite{dauge-1988} for more
references on this topic. 
%opening
%%%%%%%%%%%%%%%%%%%%%%%%%%%%%%%%%%%%%%%%%%%%%%%%%%%%%%%%%%%%%%%%%%%%%%%%%%%%%%%%
%%%%%%%%%%

\begin{remark}
 Although the boundary problem \eqref{eq1} admits complex-valued data
functions $f$, $g_D$ and $g_N$, actually this is what one can expect in many
practical applications, we assume for lighten the analysis that
all these functions, and so the solution,  are real valued. The results can be
straightforwardly adapted to
the complex case. 

For similar reasons, we will assume that $\Gamma_D\ne \emptyset$. The pure
Neumann problem can be studied with a slight modification  of the
arguments developed here (see Remark \ref{remarkNeumann} for more information
on this topic). 
\end{remark}

\section{The LDG formulation}\label{section2}
In this section, we partially follow \cite{CMAME-hmw-2006} (see also
\cite{bb-2006} and \cite{bb-2007})
to derive a discrete formulation for the linear model $(\ref{eq2})$, applying a
consistent and conservative
discontinuous Galerkin method in gradient  form.
 
\subsection{Meshes, averages, and jumps}
We let $\{{\cal T}_h\}_{h>0}$ be a family of shape-regular triangulations of
$\bar{\Omega}$ (with possible hanging nodes) made up of straight-side triangles
$T$ with diameter $h_T$ and unit outward normal vector to $\partial T$ given by
$\bm{\nu}_T$. As usual, the index $h$ also denotes 
\[
\disp h:=\max_{T\in {\cal T}_h}h_T\,, 
\]
which without loss of generality we can assume to be less than $1$. 

Given $ {\cal T}_h$, its edges are defined as follows. An {\it interior edge}
of $ {\cal T}_h$ is the (nonempty) interior of $\partial T\cap\partial T'$,
where $T$ and
$T'$ are two adjacent elements of ${\cal T}_h$, not necessarily matching.
Similarly, a {\it boundary edge} of ${\cal T}_h$ is the (nonempty) interior of
$\partial T\cap\partial\Omega$, where $T$ is a boundary element of ${\cal T}_h$.
We
denote by ${\cal E}_I$ the list of all interior edges of
(counted only once) on $\Omega$, and  by ${\cal E}_D$ and ${\cal E}_N$  the
lists of all  edges lying on $\Gamma_D$ and $\Gamma_N$. Hence,  ${\cal
E}:={\cal E}_I\cup{\cal E}_D\cup{\cal E}_N$ is the set of all edges, or
skeleton,  of the triangulation ${\cal T}_h$. Further, for each $e\in{\cal
E}$, $h_e$ represents its length. Also, in what follows we assume that ${\cal
T}_h$ is of
{\it bounded variation}, which means that there exists a constant $c > 1$,
independent of the meshsize $h$, such that
\[
c^{-1}\,\,\leq\,\,\frac{h_T}{h_{T'}}\,\,\leq\,\,c
\]
 for each pair $T,\,T'\in{\cal T}_h$ sharing an interior edge.

Next, to define average and jump operators, let  $T$ and $T'$ be two adjacent
elements of ${\cal T}_h$ and $\bm{x}$ be an
arbitrary point on the interior edge $e=\partial T\cap\partial T'\in{\cal E}_I$.
In addition, let $v$ and $\bm{\tau}$ be scalar- and vector-valued functions,
respectively, that are smooth inside each element $T\in{\cal T}_h$. We denote
by
$(v_{T,e}, \bm{\tau}_{T,e} )$ the restriction of $(v_T, \bm{\tau}_T )$ to $e$.
Then, we define the averages at $\bm{x}\in e$ by:
\[
\{v\}:=\frac{1}{2}\big(v_{T,e}+v_{T',e}\big)\,,\quad
\{\bm{\tau}\}:=\frac{1}{2}\big(\bm{\tau}_{T,e}+\bm{\tau}_{T',e}\big)\,.
\]
Similarly, the jumps at $\bm{x}\in e$ are given by
\[
\jump{v}:=v_{T,e}\,\bm{\nu}_T + v_{T',e}\,\bm{\nu}_{T'}\,,\quad
\jump{\bm{\tau}}:=\bm{\tau}_{T,e}\cdot\bm{\nu}_T +
\bm{\tau}_{T',e}\cdot\bm{\nu}_{T'}\,.
\]
On boundary edges $e$, we set  
$\{v\}:=v$, $\{\bm{\tau}\}:=\bm{\tau}$, as well as $\jump{v}:=v\,\bm{\nu}$, 
and $\jump{\bm{\tau}}:=\bm{\tau}\cdot\bm{\nu}$. Hereafter, as usual, $\nabla_h$
denotes the piecewise gradient operator.

\subsection{LDG method}
Our purpose is to approximate the exact solution
$(\bm{\sigma},u)$ of $(\ref{eq2})$ by discrete functions $(\bm{\sigma}_h,u_h)$
in appropriate finite element space $\bm{\Sigma}_h\times V_h$,  defined as
follows 
\begin{eqnarray*}
V_h&:=&\big\{v_h\in L^2(\Omega)\ : \ v|_{T}\in\mathbb{P}_m(T),\quad \forall
T\in{\cal
T}_h \big\}\,,\\
\bm{\Sigma}_h&:=&\big\{\bm{\tau}_h\in [L^2(\Omega)]^2 \ : \
\bm{\tau}_h|_{T}\in [\mathbb{P}_{m'}(T)]^2,\quad \forall T\in{\cal
T}_h
\big\}. 
\end{eqnarray*}
In the expression above $\mathbb{P}_m(T)$ denotes the space of polynomials on $T$ of
degree $m$. We restrict ourselves to consider $m \le m'+1$ so that
$\nabla_h V_h\subset
\bm{\Sigma}_h$, which is required for guaranteeing the solvability of the discrete variational formulation. The usual choices in practical situations is letting  $m'=m$ or $m'=m-1$.

We are ready to introduce  the DG method: find
$(\bm{\sigma}_h,u_h)\in \bm{\Sigma}_h\times V_h$ so that for all $T\in {\cal
T}_h$ it satisfies
\begin{equation}\label{eq4}
\begin{array}{ll}
\disp \int_T \bm{\sigma}_h\cdot \bm{\tau} + \int_T u_h\div\bm{\tau} -
\int_{\partial
T}\widehat{u}\,\bm{\tau}\cdot\bm{\nu}_T= 0 & \forall\,\bm{\tau}\in
\bm{\Sigma}_h\,,\cr\cr
\disp \int_T \bm{\sigma}_h\cdot\nabla v - \int_{\partial
T}v\,\widehat{\bm{\sigma}}\cdot\bm{\nu}_T-\omega^2\int_{T}u_hv =
\int_T f\,v & \forall\,v\in V_h\,.
\end{array}
\end{equation}
The  functions $\widehat{u}$ and $\widehat{\bm{\sigma}}$ are the
so called {\em numerical fluxes} and 
depend on $u_h$, $\bm{\sigma}_h$, the boundary
data, and are set so that some compatibility conditions are
satisfied (see~\cite{SIAM-a-2001}).

Indeed, taking into account the approach from
\cite{ps-2002} and
\cite{ccps-2000}, the LDG is defined by taking  
$\widehat{u}:=\widehat{u}(u_h,g_D)$
and $\widehat{\bm{\sigma}}:=\widehat{\bm{\sigma}}(\bm{\sigma}_h,u_h,g_D,g_N)$
for each $T\in{\cal T}_h$ as
follows:
\begin{equation}\label{ldg1}
\widehat{\bm{u}}_{T,e}:=\left\{\begin{array}{ll}
\{u_h\}+\jump{u_h}\cdot\bbe & \textrm{if }e\in{\cal E}_I,\\[.5em]
g_D    & \textrm{if }e\in{\cal E}_D,\\[.5em]
u_h & \textrm{if }e\in{\cal E}_N,
\end{array}\right.
\end{equation}
and
\begin{equation}\label{ldg2}
\hspace{32pt}\widehat{\bm{\sigma}}_{T,e}:=\left\{\begin{array}{ll}
\{\bm{\sigma}_h\}-\jump{\bm{\sigma}_h}\bbe-\alpha\jump{u_h}
& \textrm{if }e\in{\cal E}_I,\\[.5em]
\bm{\sigma}_h-\alpha(u_h-g_D)\bm{\nu} & \textrm{if } e\in{\cal E}_D,\\[.5em]
g_N\bm{\nu} & \textrm{if } e\in{\cal E}_N,
\end{array}\right.
\end{equation}
where the auxiliary functions $\alpha$ (scalar) and $\bbe$
(vector), to be chosen appropriately, are single valued on each
edge $e\in{\cal E}$ and such that they allow us to prove the optimal rates of
convergence of our approximation. To this aim, we set
$\alpha:=\frac{\widehat\alpha}{\tt h}$, and $\bbe$ as an arbitrary vector in
$\R^2$. Hereafter, $\widehat{\alpha}>0$ is fixed, while $\tt h$ is defined
on
the skeleton of ${\cal T}_h$ by
\[
{\tt h}_e:=\left\{\begin{array}{ll}
               \max\{h_T,h_{T'}\} & \textrm{if }e\in{\cal E}_I\,, \\[.5em]
               h_T & \textrm{if }e\in{\cal E}_\Gamma\,.
                \end{array}
\right.
\]
Then, integrating by parts in the first equation in $(\ref{eq4})$ and summing
up
over all $T\in{\cal T}_h$, we arrive to the problem: {\it Find
$(\bm{\sigma}_h,u_h)\in\bm{\Sigma}_h\times V_h$ such that:}
\begin{equation}\label{LDG-form1}
\begin{array}{c}
\disp \int_{\Omega} \bm{\sigma}_h\cdot\bm{\tau} -
\int_{\Omega}\nabla_hu_h\cdot\bm{\tau} + S_h(u_h,\bm{\tau}) =
\int_{\Gamma_D}g_D\bm{\tau}\cdot\bm{\nu}\,,  \cr\cr
\disp \int_{\Omega}\nabla_hv\cdot\bm{\sigma}_h - S_h(v,\bm{\sigma}_h) +
\bm{\alpha}(u_h,v)\,-\,\omega^2\int_{\Omega}u_h\,v  =
\int_{\Omega} f\,v + \int_{\Gamma_D}\alpha\, g_D\,v + \int_{\Gamma_N}g_N\,v\,,
\end{array}
\end{equation}
for all $(\bm{\tau},v)\in\bm{\Sigma}_h\times V_h$, where the bilinear forms
appearing
above are given by 
\begin{eqnarray}
S_h(v,\bm{\tau})&:=&\int_{{\cal
E}_I}\big(\{\bm{\tau}\}-\jump{\bm{\tau}}\bbe\big)\cdot\jump{v} + \int_{{\cal
E}_D}v\bm{\tau}\cdot\bm{\nu} \,,
%%\quad\forall\,(v,\bm{\tau})\in H^1({\cal T}_h)\times[H^s({\cal T}_h)]^2\,,
 \label{eq:defSbilinear}\\
\bm{\alpha}(v,w)&:=&\int_{{\cal
E}_I}\alpha\,\jump{v}\cdot\jump{w}\,+\,\int_{{\cal
E}_D} \alpha\,v\,w\,.
%,\quad\forall\,v,w\in H^1({\cal T}_h) 
\label{eq:int-form}
\end{eqnarray}

\subsection{Sobolev and discrete norms}

We denote by $\|\cdot\|_{0,\Omega}$ the standard $L^2(\Omega)$ norm. Sobolev
spaces $H^r(\Omega)$ will appear in what follows, equipped with the norm 
\[
\|v\|_{r,\Omega}^2:=\|v\|_{0,\Omega}^2+\sum_{|\bm{\beta}|=r}\left\|\partial^{\bm{\beta
} }
v\right\|_{0,\Omega}^2 \,,
\]
for positive integer $r$ (we follow the usual multi-index notation).  For
fractional values of $r=n+\gamma$, with $n\in\mathbb{N}\cup\{0\}$ and
$\gamma\in(0,1)$, we have instead the norm
\[
\|v\|_{r,\Omega}^2:=\|v\|_{0,\Omega}^2+\sum_{|\bm{\beta}|=n}|\partial^{\bm{\beta
} }
v|_{\gamma, \Omega}^2 \,,
\] 
where
\[
 |f|_{\gamma,\Omega}^2:=\int_{\Omega}\!\int_\Omega\frac{|f(\bm{x})-f(\bm{y})|^2}
{ |\bm{x}-\bm{y}|^ {
2+2\gamma}}\,{\rm d}\bm{x}\,{\rm d}\bm{y}
\]
is the Slobodecki seminorm. Finally, the tensor Sobolev spaces 
$H^r(\Omega)\times H^r(\Omega)$ will be equipped with the usual norm and denoted with the same symbol $||\cdot||_{r,\Omega}$, to avoid any confusion in the context. 

The space $L^2(e)$, with $e$ being a edge or a finite union of edges, is defined
accordingly and we will use the same notation for the norm, namely,
$\|\cdot\|_{0,e}$.

Finally,   $H^1({\cal T}_h)$ denotes the space whose the elements
$v|_T\in H^{1}(T)$, for all $T\in{\cal T}_h$. 
We endow this space with the discrete seminorm and norm
(see
\cite{b-SIAM2003})  
\begin{equation}\label{eq:seminorm}
|v|_h\,:=\,\Big(\|\alpha^{1/2}\jump{v}\|_{0,{\cal
E}_I}^2\,+\,\|\alpha^{1/2}v\|_{0, {\cal E}_D}^2\Big)^{1/2}\quad\forall\,v\in
H^1({\cal T}_h)\,,
\end{equation}
and
\begin{equation}\label{eq:deftripleh}
\triple{v}^2:=\|\nabla_hv\|_{0,\Omega}^2+|v|_h^2\quad\forall\,v\in H^1({\cal
T}_h)\,.
\end{equation}
We point out that a  Poincar\'{e} type inequality
(see \cite{bg-2004-SSC} for a proof) holds: there exists $C_{\rm P}>0$, 
independent of ${\cal T}_h$, such that
\begin{equation}
 \label{eq:poincare}
 \|v\|_{0,\Omega}\le C_{\rm P}\triple{v}\quad\forall\,v\in H^1({\cal T}_h)\,.
\end{equation}

\section{Convergence and stability of LDG method}\label{section3}

This section is devoted to proving the a priori error estimate for the method. 

\begin{theorem}\label{theo:mainLDG}
There exists $h_0=h_0(\varepsilon,\omega)>0$ such that for all $h<h_0$ the 
numerical method
\eqref{eq:primalform} admits a unique solution $u_h\in V_h$. Moreover if $u\in
H^{l+1}(\Omega)$ with $1/2<l\le m$, there holds
\[
 \triple{u_h-u}+\|\bm{\sigma}_h-\nabla u\|_{0,\Omega}\le C(\varepsilon,\omega) 
\bigg[
\sum_{T\in{\cal T}_h}  h_T^{2l}
 \|u\|_{l+1,T}^2
\bigg]^{1/2},
\]\
with $C(\varepsilon,\omega)>0$ independent of $h$ and $u$. 
\end{theorem}

The proof is presented in  the next subsections.  We stress that
$C(\varepsilon,\omega)$ and $h_0(\varepsilon,\omega)$ is shown to be dependent  
(see \eqref{eq:c:01}-\eqref{eq:c:02})
on $\omega$, via how large and how close it is  from the closest 
eigenvalue for
the Laplacian,  and the regularity of the adjoint problem
\eqref{eq:adjointProblem}, represented by the parameter $\varepsilon$. 

Hence, we
start recalling some
well-known results which we present for the sake of completeness and give, in the
last part, the proof itself.

\subsection{Approximation properties of the discrete spaces}

We start recalling the  local approximation properties of piecewise
polynomials. Denote by $\Pi_T^m: L^2(T)\to \mathbb{P}_m$ the $L^2-$orthogonal
projection. Then there exists $C>0$,
independent of the meshsize, such that for each $s,t$ satisfying $0\leq s\leq 
m+1$ and $0\leq s<t$, there holds (cf. \cite{c-1978} and \cite{gs-2004})
\begin{equation}\label{interp1}
|w-\Pi_T^m w|_{s,T} \leq
C\,h_T^{\min\{t,m+1\}-s}\|w\|_{t,T}\quad\forall\,w\in H^t(T)\,,
\end{equation}
and
\begin{equation}\label{interp2}
|w-\Pi_T^m w|_{0,\partial T} \leq
C\,h_T^{\min\{t,m+1\}-1/2}\|w\|_{t,T}\quad\forall\,w\in H^t(T)\,.
\end{equation}
Therefore, if 
\[
\Pi_{V_h}:L^2(\Omega)\to V_h, \qquad
\bm{\Pi_{\Sigma_h}}:[L^2(\Omega)]^2\to\bm{\Sigma}_h \,,
\]
are the $L^2-$orthogonal projections on the discrete spaces $V_h$ and
$\bm{\Sigma}_h$, respectively, we have, from \eqref{interp1}-\eqref{interp2}: 
\begin{equation}
\label{eq:Pi-I}
 \|v-\Pi_{V_h} v\|_{0,\Omega}\le C_l\bigg[\sum_{T\in{\cal T}_h}
h_T^{2l}\|v\|_{l,T}^2\bigg]^{1/2},\quad \text{and}\quad 
 \triple{v-\Pi_{V_h} v}\le C_l\bigg[\sum_{T\in{\cal T}_h}
h_T^{2l-2}\|v\|_{l,T}^2\bigg]^{1/2}\,,
\end{equation}
for $v\in H^l({\mathcal T}_h)$, $1\le l\le m+1$, with  $C_l>0$ independent of $v$, $\bm{\sigma}$ and ${\cal
T}_h$.  Similarly, we obtain
\[
 \|\bm{\Pi_{\Sigma_h}}\bm{\sigma}-\bm{\sigma}\|_{0,\Omega}\le
C_l\bigg[\sum_{T\in{\cal T}_h}
h_T^{2l}\|\bm{\sigma}\|_{l,T}^2\bigg]^{1/2},
\] 
for $\bm{\sigma}\in [H^l({\mathcal T}_h)]^2$, with $1\le l\le m'+1$.

\subsection{The primal formulation of the LDG method}
\label{sec:LDG}

Associated to $S_h$ (cf. \eqref{eq:defSbilinear}), we introduce the discrete
lifting operators
  ${\bf S}_h:H^1({\cal
T}_h)\to \bm{\Sigma}_h$ and ${\bf G}_h:L^2(\Gamma_D)\to \bm{\Sigma}_h$ defined,
respectively, as the solutions of the problems
\begin{eqnarray}
\int_\Omega {\bf S}_h(v)\cdot\bm{\tau} =S_h(v,\bm{\tau})\qquad \forall\,
\bm{\tau}\in\bm{\Sigma}_h\,,\label{eq:defSh}\\
\int_\Omega {\bf G}_h(g_D)\cdot\bm{\tau} =\int_{\Gamma_D} g_D\:
\bm{\tau}\cdot\bm{\nu}\qquad\forall\,\bm{\tau}\in\bm{\Sigma}_h\,,\label{eq:defGh}
\end{eqnarray}
whose existence and uniqueness are guaranteed by Riesz representation theorem. 
We notice in passing that if $v\in H^1(\Omega)$ with
$v|_{\Gamma_D}=g_D$, then ${\bf G}_h(g_D)={\bf S}_h(v)$.

Thus, the first equation of \eqref{LDG-form1} can be read as 
\[
 \int_{\Omega} \bm{\sigma}_h\cdot\bm{\tau}=
 \int_{\Omega} (\nabla_h u_h-{\bf S}_h(u_h)+{\bf G}_h(g_D))\cdot\bm{\tau}\qquad
\forall\,
\bm{\tau}\in \bm{\Sigma}_h.
\]
Since $\nabla_h V_h\subset \bm{\Sigma}_h$,  we conclude
\begin{equation}
\label{eq:sigma}
 \bm{\sigma}_h=\nabla_h u_h-{\bf S}_h(u_h)+{\bf G}_h(g_D).
\end{equation}
In other words, we have expressed $\bm{\sigma}_h$ in terms of $u_h$ and the
Dirichlet data. Besides, from the second equation of \eqref{LDG-form1}, taking into account again ${\bf S}_h$, we obtain
\begin{eqnarray*}
 \int_\Omega \bm{\sigma}_h\cdot\nabla_h v
+\bm{\alpha}(u_h,v) -\int_{\Omega} {\bf S}_h(v)\cdot
\bm{\sigma}_h-\omega^2\int_{\Omega} u_h v=\int_{\Omega} f v +\int_{\Gamma_N}g_N
v+\int_{\Gamma_D}\alpha g_D
v.
\end{eqnarray*}
Using \eqref{eq:sigma} to substitute $\bm{\sigma}_h$ we arrive to the reduced (an equivalent)
primal form: {\it Find
$u_h\in V_h$ such that}
\begin{equation}\label{eq:primalform}
a_h(u_h,v)-\omega^2\int_{\Omega} u_hv={\cal{F}}_h(v)\qquad \forall\, v\in V_h\,,
\end{equation}
where 
\begin{eqnarray}
a_h(t,v)&:=&\int_{\Omega}(\nabla_h t -{\bf S}_h(t))\cdot(\nabla_h v -{\bf S}_h
(v))+\bm{\alpha}(t,v)\label{eq:defa_h}\,,\\
{\mathcal F}_h(v)&:=& \int_{\Omega} f v+\int_{\Gamma_N} g_N v+\int_{D}\alpha g_D
v-\int_{\Omega}  (\nabla_h v-{\bf S}_h(v)) \cdot {\bf G}_h(g_D)\,.
\label{eq:defl}
\end{eqnarray} 
This way we have established the next result.

\begin{theorem}
If $(\bm{\sigma}_h,u_h)\in\bm{\Sigma}_h\times V_h$ is a solution of 
\eqref{LDG-form1}, then 
%$\nabla_h u_h-{\bf S}_h(u_h)+{\bf G}_h(g_D)$ and 
$u_h\in V_h$ is a
solution of \eqref{eq:primalform}. Reciprocally, if $u_h\in V_h$ is a solution of
\eqref{eq:primalform} then $(\bm{\sigma}_h,u_h)\in\bm{\Sigma}_h\times V_h$, with $\bm{\sigma}_h:=\nabla_h u_h-{\bf S}_h(u_h)+{\bf
G}_h(g_D)$, is a solution of  \eqref{LDG-form1}.
\end{theorem}

The boundedness and ellipticity of bilinear form $a_h$ is established next.  
\begin{theorem}\label{theo:prop:ah}
 There exist  $C_{\rm cont}, c_{\rm coer} >0$ such that for all $t,v\in H^1({\cal
T}_h)$ there hold
\[
 |a_h(t,v)| \le   C_{\rm cont} \triple{t} \triple{v},\quad{\rm and}\quad
  a_h(v,v)  \ge c_{\rm coer} \triple{v}^2.
\]
\end{theorem}
\begin{proof}
 We refer to \cite{ps-2002, bg-2004-SSC} for a proof of this result. We note
that $C_{\rm cont}:=\max\{2,2\|{\bf S}_h\|,\|{\bf S}_h\|^2\}$.
\end{proof}

We point out that this theorem is the key result for proving stability and
convergence of the method for the Laplace equation. However, the $L^2-$term
spoils the coercivity of the bilinear form and forces us to consider a
different approach for proving the convergence of the method.  This is what we
will describe in next subsection. 

%%%%%%%%%%%%%%%%%%%%%%%%%%%%%%%%%%%%%%%%%%%%%%%%%%%%%%%%%%%%%%%%%%%%%%%%%%%%%%%%
%%%%%%%%%%
\subsection{Proof of Theorem \ref{theo:mainLDG}}%\label{section3}

We start assuming that the exact solution $u\in
H^{l+1}(\Omega)$ with $l>1/2$ and that there exists a numerical solution $u_h$
for the reduced  primal scheme $(\ref{eq:primalform})$. 

Take an arbitrary element  $v\in V_h$. Now, thanks to the
coercivity of $a_h$ (cf. Theorem \ref{theo:prop:ah}), there holds
\begin{eqnarray}
c_{\rm coer}\triple{u_h-v}^2&\le&
a_h(u_h-v,u_h-v)=a_h(u_h,u_h-v)-a_h(u,u_h-v)+a_h(u-v,u_h-v)\nonumber\\
&=&R_h(u,u_h-v)+\omega^2\int_{\Omega}(u_h-u)(u_h-v)
+a_h(u-v,u_h-v)\,,
\label{eq:uh-v}
\end{eqnarray}
where
\[
 R_h(u,q):={\mathcal F}_h(q)-a_h(u,q)+\omega^2\int_{\Omega} u q \qquad\forall\, q\in V_h\,,
\]
is the so-called {\it consistency term}. We point out in pass that, unlike the
original formulation of the method \eqref{eq4},  the consistency term does not
vanish for the exact solution $u$, i.e. the primal formulation is not consistent.

By using the continuity of the bilinear form $a_h$,  
we derive  from \eqref{eq:uh-v} that
\begin{eqnarray}
 c_{\rm coer} \triple{u_h-v}&\le& \sup_{0\ne z\in V_h}
\frac{1}{\triple{z}}  R_h(u,z)+\omega^2 \sup_{0\ne z\in V_h}
\frac{1}{\triple{z}}   \int_{\Omega} (u-u_h)
z  % \nonumber
%\\
%&&+
\,+\,C_{\rm
cont}\triple{u-v}\,.
 \label{eq:toBound:LDG}
\end{eqnarray}

The first term is bounded straightforwardly by using
the following relations \cite{bg-2004-SSC,gs-2004}
\begin{equation}
\label{eq:boundFirstTerm}
  R_h(u,z)  = {S}_h(z,\bm{\Pi}_{\bm{\Sigma}_h}\nabla
u-\nabla u) 
 \le  C_{l,S}   \bigg[\sum_{T\in{\cal T}_h}
 h_T^{2l}|\nabla u|^2_{l,T}
\bigg]^{1/2}{\triple{z}}\qquad\forall\,z\in V_h\,,
\end{equation}
which holds for all $1/2<l\le m$.

Next we show a simple presentation of the boundedness of second order
 term in \eqref{eq:toBound:LDG}. To this aim, we require the following lemma.
Recall first the mapping ${\cal L}_\omega$ introduced in \eqref{eq:defLw}.

\begin{lemma}\label{lemma:term2}
Let $z\in L^2(\Omega)$ and $u$ the exact solution of \eqref{eq1}. If $u_h\in V_h$
is a solution of \eqref{eq:primalform} then for 
any $\psi\in
V_h$ it holds
\begin{eqnarray*}
\int_{\Omega}(u-u_h)z&=&   a_h(u-u_h,{\cal
L}_{\omega}z-\psi)+\omega^2\int_{\Omega}
(u-u_h)({\cal L}_{\omega}z-\psi)+
S_h(u_h-u, \nabla{\cal
L}_{\omega}z- \bm{\Pi}_{\bm{\Sigma}_h}\nabla{\cal L}_{\omega}z)\\
&&+
S_h(\psi,\nabla u-\bm{\Pi}_{\bm{\Sigma}_h}\nabla u).
\end{eqnarray*}
\end{lemma}
\begin{proof}
Fix $z\in L^2(\Omega)$ and set $\varphi:={\cal L}_{\omega}z\in
H^{1+\varepsilon}(\Omega)$ (recall that $\varepsilon>1/2$). Take then $v\in
H^1({\cal T}_h)$.
By
integrating by parts on each element of the grid ${\cal T}_h$ and using that
$\jump{\nabla \varphi}=0$ on any $e\in{\cal E}_N\cup{\cal E}_I$, we deduce
\begin{eqnarray*}
 -\int_{\Omega}v z&=&\int_{\Omega} v(\Delta
\varphi+\omega^2\varphi)=-\int_{\Omega} \nabla_h v\cdot\nabla\varphi
+\omega^2\int_{\Omega} v \varphi+
\int_{{\cal E}_D\cup {\cal E}_I } \jump{v}\cdot\{\nabla \varphi\}.
\end{eqnarray*}
Besides, since  $\varphi|_{\Gamma_D}=0$, $\jump{\varphi}=0$ for all $e\in
{\cal E}_D\cup{\cal E}_I$,
\[
  {\bf S}_h(\varphi)={\bf 0},\qquad \text{and}\qquad
\bm{\alpha}(\varphi,v)=0,\quad \forall\,
v\in H^1({\cal T}_h).
\]
Thus
\begin{eqnarray*}
 -\int_{\Omega} v z&=& -a_h(v,\varphi) +\omega^2\int_{\Omega} v\varphi
-\int_{\Omega} {\bf S}_h(v)\cdot \nabla\varphi+\int_{{\cal E}_D\cup {\cal E}_I }
\jump{v}\cdot\{\nabla \varphi\}\\
&=& -a_h(v,\varphi) +\omega^2\int_{\Omega} v\varphi
-\int_{\Omega} {\bf S}_h(v)\cdot
\bm{\Pi}_{\bm{\Sigma}_h}\nabla\varphi+\int_{{\cal
E}_D\cup {\cal E}_I }
\jump{v}\cdot\{\nabla \varphi\}\\
&=&-a_h(v,\varphi) +\omega^2\int_{\Omega} v\varphi+
 \int_{{\cal
E}_D\cup {\cal E}_I }
\jump{v}\cdot\{\nabla \varphi-\bm{\Pi}_{\bm{\Sigma}_h}\nabla \varphi\}
+
 \int_{{\cal E}_I }
\jump{ \bm{\Pi}_{\bm{\Sigma}_h}\nabla  \varphi}\ \bm{\beta}\cdot\jump{v}\\
&=&-a_h(v,\varphi) +\omega^2\int_{\Omega} v\varphi+
 \int_{{\cal
E}_D\cup {\cal E}_I }
\jump{v}\cdot\{\nabla \varphi-\bm{\Pi}_{\bm{\Sigma}_h}\nabla \varphi\}
-
 \int_{{\cal E}_I }
\jump{\nabla  \varphi- \bm{\Pi}_{\bm{\Sigma}_h}\nabla  \varphi}\
\bm{\beta}\cdot\jump{v}\\
&=&-a_h(v,\varphi) +\omega^2\int_{\Omega} v\varphi+S_h(v,  \nabla
\varphi-\bm{\Pi}_{\bm{\Sigma}_h}\nabla \varphi)\,,
\end{eqnarray*}
where we have applied sequentially the definition of $a_h$ cf. \eqref{eq:defa_h}, the
fact that ${\bf S}_h(v)\in {\bm \Sigma}_h$, the definitions  of ${\bf S}_h$
and its associated bilinear
form ${S}_h$ cf. \eqref{eq:defSbilinear} and \eqref{eq:defSh} respectively, and that
$\jump{\nabla \varphi}=0$ on any $e\in {\cal E}_{I}$.

Take  now $v=u-u_h$ above. We can then check that for any $\psi\in V_h$,
\begin{eqnarray*}
 -\int_{\Omega} (u-u_h)
z&=&-a_h(u-u_h,\varphi)+\omega^2\int_{\Omega}(u-u_h)\varphi+S_h(u-u_h,
\nabla
\varphi-\bm{\Pi}_{\bm{\Sigma}_h}\nabla \varphi)\\
&=&-a_h(u-u_h,\varphi-\psi)+\omega^2\int_{\Omega}(u-u_h)(\varphi-\psi)\\
&&+a_h(u_h,\psi)-\omega^2\int_{\Omega} u_h\psi
-a_h(u,\psi)+\omega^2\int_{\Omega} u \psi
+S_h(u-u_h, \nabla
\varphi-\bm{\Pi}_{\bm{\Sigma}_h}\nabla \varphi)\\
&=&-a_h(u-u_h,\varphi-\psi)+\omega^2\int_{\Omega}(u-u_h)(\varphi-\psi)\\
&&+{\mathcal F}_h(\psi) - a_h(u,\psi)+\omega^2\int_{\Omega} u \psi+
S_h(u-u_h, \nabla
\varphi-\bm{\Pi}_{\bm{\Sigma}_h}\nabla \varphi)\\
&=&-a_h(u-u_h,\varphi-\psi)+\omega^2\int_{\Omega}
(u-u_h)(\varphi-\psi)+R_h(u,\psi) \\
&&+
S_h(u-u_h, \nabla
\varphi-\bm{\Pi}_{\bm{\Sigma}_h}\nabla \varphi)\,.
\end{eqnarray*}
The result follows now readily by applying  \eqref{eq:boundFirstTerm}.
\end{proof}

Now we are in position to establish a bound of the second term in \eqref{eq:toBound:LDG}. Hereafter, we denote by
 $\|A\|_{X\to Y}$ the operator norm of a linear mapping
$A:X\to Y$ between two normed spaces $X$ and $Y$.

\begin{proposition}
\label{prop:term2:v2}
 Under the same notations defined above, and for all $u\in H^{l+1}({\cal T}_h)$
with $1/2<l\le m$,
\begin{eqnarray}
\sup_{0\ne z\in V_h}
\frac{1}{\triple{z}}
\bigg|\int_{\Omega}(u-u_h)z\bigg| 
& \le&  C_{\rm P}\|{\cal
L}_{\omega}\|_{L^2(\Omega)\to
H^{1+\varepsilon}(\Omega)}   \bigg[\big((C_{\rm
cont}C_\varepsilon+C_{\varepsilon,S}) h^{\varepsilon}+
C_{\rm P}C_\varepsilon\omega^2
h^{1+\varepsilon}\big)\triple{u-u_h}\qquad \nonumber\\
&&+ C_{l ,S}'  \bigg[ \sum_{T\in{\cal
T}_h} h_T^{2l}
 |\nabla u|^2_{l,T}
\bigg]^{1/2}\bigg]\,, \label{eq:bound(u-uh)z}
\end{eqnarray}
where $ C'_{l,S},C_\varepsilon, C_{\varepsilon,S}>0$, with $\varepsilon>1/2$
being as in \eqref{eq:defLw} and $C_{\rm P}>0$ the Poincar\'e-type inequality constant given in
\eqref{eq:poincare}, all of them independent of ${\cal T}_h$, $\omega$ and
$u$.
\end{proposition}
\begin{proof}
Taking $\psi=\Pi_{V_h}{\cal L}_\omega z$  in  Lemma \ref{lemma:term2} and
applying Theorem \ref{theo:prop:ah} and the Cauchy-Schwarz inequality we derive
\begin{eqnarray}
\bigg|\int_{\Omega}(u-u_h)z\bigg|\!\! &\le &\!\!C_{\rm
cont}\triple{u-u_h}\triple{{\cal
L}_{\omega}z-\Pi_{V_h}{\cal L}_\omega z}+
\omega^2 \|u-u_h\|_{0,\Omega}\|{\cal L}_{\omega}z-\Pi_{V_h}{\cal L}_\omega
z\|_{0,\Omega}\nonumber\\
&&\!\!+
\Big|S_h(u_h-u, \bm{\Pi}_{\bm{\Sigma}_h} \nabla{\cal L}_{\omega}z- \nabla{\cal
L}_{\omega}z)\Big|+
\Big|S_h(\Pi_{V_h}{\cal L}_\omega z , \bm{\Pi}_{\bm{\Sigma}_h}\nabla u- \nabla
u)\Big|\,.\label{eq:01:lemma:term2:v2}
\end{eqnarray}
Estimate \eqref{eq:Pi-I} implies 
\begin{eqnarray*}
 \|{\cal L}_{\omega}z-\Pi_{V_h}{\cal L}_\omega
z \|_{0,\Omega}+
h \triple{{\cal L}_{\omega}z-\Pi_{V_h}{\cal L}_\omega
z}&\le& C_\varepsilon h^{1+\varepsilon}\  \|{\cal
L}_{\omega}z\|_{1+\varepsilon,\Omega}
\\
&\le&
C_\varepsilon h^{1+\varepsilon}\|{\cal
L}_\omega \|_{L^2(\Omega)\to
H^{1+\varepsilon}(\Omega)}  \|z\|_{0,\Omega}\,,
\end{eqnarray*}
which together with Poincar\'{e} inequality \eqref{eq:poincare} let us to
bound the first two terms of \eqref{eq:01:lemma:term2:v2}.

Regarding the third term,  using \eqref{eq:boundFirstTerm} we deduce
\begin{eqnarray*}
|
S_h(u-u_h, \nabla {\cal L}_\omega z-\bm{\Pi}_{\bm{\Sigma}_h} \nabla {\cal
L}_\omega z)|
&\le&
C_{\varepsilon, S}\triple{u-u_h}
\Bigg[ \sum_{T\in{\cal T}_h}  h_T^{2\varepsilon}
 |\nabla{\cal L}_\omega z|^2_{\varepsilon,T}
\Bigg]^{1/2}\\
&\le& C_{\varepsilon,S}
h^{\varepsilon}\triple{u-u_h}\|{\cal L}_\omega
z\|_{1+\varepsilon,\Omega}\\
&\le&
C_{\varepsilon,S}h^{\varepsilon}\|{\cal
L}_\omega \|_{L^2(\Omega)\to
H^{1+\varepsilon}(\Omega)}  \triple{u-u_h}\|z\|_{0,\Omega}.
\end{eqnarray*}
Finally, for the last term in \eqref{eq:01:lemma:term2:v2} we proceed
analogously as before to obtain:
\begin{eqnarray*}
|(S_h( \Pi_{V_h} {\cal L}_\omega z, \bm{\Pi}_{\bm{\Sigma}_h}\nabla
u-\nabla
u)|\!\!&& \\
&& \hspace{-2cm}\le C_{l,S}
\big(\triple{\Pi_{V_h}{\cal L}_\omega z  -{\cal L}_\omega z  }+
\triple{{\cal L}_\omega z}\big)
\Bigg[ \sum_{T\in{\cal T}_h}  h_T^{2l}
 |\nabla u|_{l,T}^2
\Bigg]^{1/2}\\
&& \hspace{-2cm}\le C_{l,S} (C_\varepsilon h^{\varepsilon}+1)\|{\cal
L}_\omega z \|_{1+\varepsilon,\Omega}
\Bigg[ \sum_{T\in{\cal T}_h}  h_T^{2l}
 |\nabla u|_{l,T}^2
\Bigg]^{1/2}\\
&& \hspace{-2cm}\le C_{l,S}'\|{\cal
L}_\omega   \|_{L^2(\Omega)\to
H^{1+\varepsilon}(\Omega)}
\Bigg[ \sum_{T\in{\cal T}_h}  h_T^{2l}
 |\nabla u|_{l,T}^2
\Bigg]^{1/2}\|z\|_{0,\Omega},
\end{eqnarray*}
with $C_{l,S}':=C_{l,S}(C_{\varepsilon}{\rm diam}(\Omega)^{\varepsilon}+1)$.
The proof is now finished, once we apply Poincar\'e inequality \eqref{eq:poincare}. % by using the Poincar\'e inequality \eqref{eq:poincare}
\end{proof}

We are ready to prove the main result of this paper. 

\noindent{\bf\em Proof of Theorem \ref{theo:mainLDG}}. Let $u$
be the exact solution of \eqref{eq1}
 and suppose $u_h$ is a solution of \eqref{eq:primalform}.
%and $u_h$ be the
%exact and a numerical solution of \eqref{eq1} and \eqref{eq:primalform},
%respectively. 
Next, we take an arbitrary $v\in V_h$ and write
\[
 \triple{u_h-u}\le  \triple{u_h-v} +\triple{v-u}\,.
\]
By applying \eqref{eq:boundFirstTerm} and
\eqref{eq:bound(u-uh)z} of Proposition \ref{prop:term2:v2} in
\eqref{eq:toBound:LDG}, we derive the following bound for the
first term:
\begin{eqnarray*}
 (1- c(\omega,h,\varepsilon))\triple{u_h-v} &\le&
C_2(\omega,\varepsilon)\left[ \sum_{T\in{\cal T}_h}  h_T^{2l}
 |\nabla u|_{l,T}^2
\right]^{1/2}+c_{\rm coerc}^{-1} C_{\rm cont} \triple{u-v}\,,
\end{eqnarray*}
where
\begin{eqnarray}
c (\omega,h,\varepsilon)&:=& c_{\rm coer}^{-1}\omega^2 C_{\rm P}
\|{\cal L}_{\omega}\|_{L^2(\Omega)\to
H^{1+\varepsilon}(\Omega)} h^\varepsilon
\big(\big(C_{\rm cont}C_{\varepsilon} +C_{\varepsilon,S}\big)
+C_{\rm P} C_{\varepsilon} \omega^2
h\big),\label{eq:c:01}\\
C_2(\omega,\varepsilon)&:=&
c_{\rm
coer}^{-1} \big(1+ \omega^2 C_{\rm P}\|{\cal
L}_{\omega}\|_{L^2(\Omega)\to
H^{1+\varepsilon}(\Omega)}  \big)C'_{l,S}\,.\label{eq:c:02}
\end{eqnarray}
By taking $h_0$ small, say for instance,
\[
 c (\omega,h_0,\varepsilon)\le 1/2\,,
\]
and setting $v= \Pi_{V_h} u$, we have that for $h<h_0$,
\begin{equation}
 \label{eq:estimate:for_u-u_h}
  \triple{u_h-u}\le C(\omega,\varepsilon)\left[ \sum_{T\in{\cal T}_h}
h_T^{2l}
 |\nabla u|_{l,T}^2
\right]^{1/2}\,,
\end{equation}
where
\[
 C(\omega,\varepsilon):=
2 C_2(\omega,\varepsilon)  +
(2  c_{\rm coerc}^{-1} C_{\rm cont}   +1)C_l\,.
%\max\Big\{2 C_2(\omega,\varepsilon)  ,
%(2  c_{\rm coerc}^{-1} C_{\rm cont}   +1)C_l
%\Big\}\,.
\]

Now, the proof of uniqueness solvability of the numerical scheme relies in
the fact that the associated homogeneous discrete linear system has only the
trivial solution. Indeed, if the exact
solution is $u=0$ and $u_h$ is a  
solution of the discrete method,  then \eqref{eq:estimate:for_u-u_h} yields
\[
 \triple{u_h}\le 0\,,
\]
and therefore $u_h=0$ is the only solution of the homogeneous scheme. Thus, we
conclude that the LDG scheme always has only one solution, for $h$ small enough. 

Finally, the convergence for $\bm{\sigma}_h$ follows from standard arguments.
Hence,  using
\eqref{eq:sigma},
\begin{eqnarray}
\|\bm{\sigma}_h-\nabla u\|_{0,\Omega}&\le& \|\nabla_h u_h-\nabla u\|_{0,\Omega}+
\|{\bf
S}_h(u_h)-{\bf G}_h(g_D)\|_{0,\Omega}\nonumber\\
&\le& \triple{u-u_h}+ \|{\bf
S}_h(u_h-u)\|_{0,\Omega}\nonumber\\
&\le& C\, \triple{u-u_h}\,,\label{eq:sigma-nablau}
\end{eqnarray}
%The result follows now readily.
and we end the proof.

\hfill $\Box$

\begin{remark}\label{remarkNeumann}
For the
pure Neumann problem,  the function $\triple{u_h}$ (see \eqref{eq:deftripleh})
is not longer a norm, but a seminorm. This problem can be covered by working
instead with $\triple{u_h}+\|u_h\|_{0,\Omega}$, 
which becomes  again  a norm, satisfying in addition the
Poincar\'{e} inequality. The bilinear form  
$a_h$ \eqref{eq:defa_h} has to be also slightly modified, by adding a $L^2$
term, to  make it elliptic in this norm. The rest of the analysis is
essentially the same. We leave this case as a simple exercise for the reader. 
\end{remark}

\subsection{The case of complex $\omega$}
Let us finish this section analyzing the case of $\omega$ being a complex
number with positive imaginary part.  Problem \eqref{eq1} admits now a unique
solution and it is easy to see that the LDG method is stable and convergent via
an inf-sup condition.

Let us prove that. First, we write
\[
 \omega=|\omega| e^{i{\theta}}\,,
\]
with $\theta\in(0,\pi)$. Then for all $0\ne v\in V_h$ and $\phi\in(0,\pi)$
\[
\sup_{0\ne t\in V_h}\frac{1}{\triple{t}}\bigg|
a_h(v,t)-\omega^2\int_\Omega {v t}\bigg|\ge
\frac{1}{\triple{v}}{\rm
Re}\bigg(
e^{i\phi} a_h(v,\overline{v})+ e^{i(2\theta+\phi-\pi)}|\omega|^2\int |v|^2
\bigg)\,,
\]
simply by taking $t=e^{i\phi} \overline{v}$.  Setting  
\[
\phi:=\bigg\{\begin{array}{ll}
\pi/2-2\theta,\quad& \theta\in(0,\pi/4),\\
              0,\quad&  \theta\in[\pi/4,3\pi/4],\\
3\pi/2-2\theta,\quad&\theta\in(3\pi/4,\pi),
             \end{array}
\]
and using the coercivity of the bilinear form $a_h$ (cf. Theorem
\ref{theo:prop:ah}) and that $\cos(2\theta+\phi-\pi)\ge 0$,  we can prove 
\[
 \sup_{0\ne t\in V_h}\frac{1}{\triple{t}}\bigg|
a_h(v,t)-\omega^2\int_\Omega v t\bigg|\ge  \frac{\cos\phi}{\triple{v_h}}\,
a_h(v,\overline{v})\ge
\hat{c}_{\rm coer}(\theta) \triple{v}\,,
\]
where
\[
\hat{c}_{\rm coer}(\theta):=
\bigg\{\begin{array}{ll}
              c_{\rm coer},\quad& \theta\in[\pi/4,3\pi/4]\,,\\[1.25ex]
|\sin 2\theta|c_{\rm coer},\quad& \theta\in(0,\pi/4)\cup  (3\pi/4,\pi)\,,
             \end{array}
\]
with the positive constant $c_{\rm coer}$ given in Theorem \ref{theo:prop:ah}.

Let $u_h$ be a numerical solution of \eqref{eq:primalform}. Then, for a given $v\in V_h$ we have
\begin{eqnarray*}
 \hat{c}_{\rm coer}(\theta)\triple{u_h-v}&\le&  \sup_{0\ne t\in
V_h}\frac{1}{\triple{t}}\bigg|
a_h(u_h-v,t)-\omega^2\int_\Omega (u_h-v)t\bigg|\\
&=&\sup_{0\ne t\in V_h}\frac{1}{\triple{t}}\bigg|
  R_h(u,t)  + a_h(u -v,t)-\omega^2\int_{\Omega}(u-v)t \bigg|\,.
\end{eqnarray*}
Hence taking $v=\Pi_{V_h}u$, and using \eqref{eq:boundFirstTerm}, Theorem \ref{theo:prop:ah} and \eqref{eq:Pi-I}, we easily derive
\begin{eqnarray}
 \triple{u-u_h}&\le& \triple{u-\Pi_{V_h} u}+\triple{u_h-\Pi_{V_h} u}\nonumber\\
&\le& (\hat{c}_{\rm coer}^{-1}(\theta)C_{\rm cont} +1)\triple{u -\Pi_{V_h}  
u} + \hat{c}_{\rm coer}^{-1}(\theta) C_{l,S}\bigg[ \sum_{T\in{\cal T}_h}  h_T^{2l}
 |\nabla u|_{l,T}^2\bigg]^{1/2}\nonumber\\ 
&&+ |\omega|^2 \hat{c}_{\rm coer}^{-1}(\theta)C_P C_l\bigg[ \sum_{T\in{\cal
T}_h} h_T^{2l+2}
 |\nabla u|_{l,T}^2\bigg]^{1/2}.\label{eq:converg_complex_omega}
\end{eqnarray}

Hence, the stability of the LDG method occurs without assuming any regularity
for the adjoint problem and it is not affected by $\omega$ itself, but by the
argument of $\omega$, i.e., by $\theta$. Notice, however, that    $\hat{c}_{\rm
coer}(\theta)\to 0$ as $\theta\to 0,\pi$. On the other hand, $|\omega|^2$ does
 penalize the convergence such
as it can be clearly seen in the last term in \eqref{eq:converg_complex_omega}.

%%%%%%%%%%%%%%%%%%%%%%%%%%%%%%%%%%%%%%%%%%%%%%%%%%%%%%%%%%%%%%%%%%%%%%%%%%%%%%%%
%%%%%%%%%%
\section{A posteriori error analysis}\label{section4}
%%%%%%%%%%%%%%%%%%%%%%%%%%%%%%%%%%%%%%%%%%%%%%%%%%%%%%%%%%%%%%%%%%%%%%%%%%%%%%%%
%%%%%%%%%%
 
The aim of this section is to develop an a posteriori error estimator for the
LDG scheme (\ref{LDG-form1}).  To this end,  we first use the auxiliary dual
problem \eqref{eq:adjointProblem} to bound the $\|u-u_h\|_{0,\Omega}$. Next, a
Helmholtz decomposition is introduced to derive a reliable and efficient  a
posteriori error estimate. Hereafter, we introduce $\curl v :=(-\frac{\partial
v}{\partial y},\frac{\partial v}{\partial x})$ for any $v\in H^1(\Omega)$, and
the Sobolev space $H^1_{\Gamma_D}:=\{ v\in H^1(\Omega): v=0 \qon \Gamma_D\}$.
 
 The main result of the present section is summarized in the next theorem.
 \begin{theorem}\label{theo:apost}
 Let $(\bm{\sigma},u)\in H(\div,\Omega)\times H^1(\Omega)$ and 
$(\bm{\sigma}_h,u_h)\in \bm{\Sigma}_h\times V_h$ the unique solution of Problems
\eqref{eq2} and \eqref{LDG-form1}, respectively.
Then there exist $C_{\rm rel},\, C_{\rm eff}>0$, independent of the meshsize and the wave number, such
that
\begin{equation}\label{eq:01:teo-apost}
\triple{u-u_h}^2 + \|{\bm\sigma}-{\bm\sigma}_h\|^2_{0,\Omega}\le
C^2_{\rm rel}(1+\omega^2)^2\,\eta^2:=C^2_{\rm rel}(1+\omega^2)^2\sum_{T\in{\cal T}_h}\eta^2_T\,,
\end{equation}
where for any $T\in{\mathcal T}_h$ we define
\begin{eqnarray}
 \eta^2_T&:=&\,h^2_T\,\|f+\omega^2 u_h+\Delta u_h\|^2_{0,T}+\|\nabla
u_h-\bm{\sigma}_h\|^2_{0,T}+h_T\,\left\|\widehat{\bm\sigma}
\cdot\bnu_T-\nabla u_h\cdot\bnu_T \right\|^2_{0, \partial
T\setminus\Gamma_D}\nonumber\\
&&+\|\alpha^{1/2}\jump{u_h}\|^2_{0,\partial T\cap {\cal
E}_I}+\|\alpha^{1/2}(g_D-u_h)\|^2_{0,\partial T\cap {\cal E}_D}\,.
\label{eq:02:teo-apost}
\end{eqnarray}
Moreover, for each $T\in{\mathcal T}_h$:
\begin{eqnarray}
\eta^2_T&\le& C^2_{\rm eff}\big(
\|\bm{\sigma}-\bm{\sigma}_h\|_{0,\mathcal{N}(T)}^2+\|\nabla u-\nabla_h
u_h\|_{0,\mathcal{N}(T)}^2+\omega^4\,h_T^2\,
%\max\{h_T,h_{T'}\}^2
\|u-u_h\|^2_{0,\mathcal{N}(T)}\nonumber\\
&& +\|\alpha^{1/2}\jump{u-u_h}\|_{0,
\partial T\cap{\cal E}_i }+\|\alpha^{1/2}(g_D-u_h)\|^2_{0,\partial T\cap {\cal
E}_D} \big) +\text{h.o.t.}\label{eq:03:teo-apost}
\end{eqnarray}
where  
\[
  \mathcal{N}(T):=\bigcup_{\partial T'\cap \partial T\in {\mathcal E} } T', 
\]
and $h.o.t.$ stands for higher order terms.

\end{theorem}
 
The proof is presented in the two following subsections. 

\subsection{Reliability of the estimator}

Our first aim is to estimate $\|e_u\|_{0,\Omega}$, where $e_u:=u-u_h$ in $\Omega$. For this purpose, 
we take into account Hypothesis \ref{hypothesis} and introduce the function $\phi:={\cal L}_\omega e_u$, that is,
 \begin{equation}
 \label{Aux-Pro}
 -\Delta \phi-\omega^2\phi = e_u,\quad\text{in $\Omega$}\, \qquad
  \phi|_{\Gamma_D}=0,\quad
\partial_{\bm{\nu}} \phi|_{\Gamma_N}=0\,.
\end{equation}

Let $\Pi_0$ be  the piecewise constant projection from $H^1(\Omega)$ onto
$L^2(\Omega)$ defined by
\[
(\Pi_0z)|_T:=\left\{\begin{array}{ll}
\frac{1}{|T|}\int_Tz\,,\quad& \text{if $T\cap \Gamma_D=\emptyset$},\\
0,&\text{otherwise}.
\end{array}\right.
 \]
By combining the classical results on convergence of the $L^2$-orthogonal
projection and the local  Poincar\'e inequality on the triangles with a side
lying on $\Gamma_D$, we can prove
\begin{equation}\label{eq:bound:Pi0}
\|\psi-\Pi_0\psi\|_{0,T}\le C\,h_T\,
\|\nabla\psi \|_{0,T},\qquad
\|\psi-\Pi_0\psi\|_{0,\partial T}\le C\,h^{1/2}_T\,\|\nabla\psi \|_{0,T}\,, 
\end{equation}
for all $T\in  {\cal T}_h$, with $C>0$ independent of $T$ and $\psi$. (See 
\cite[Lemma 5.1]{bg-2004-SSC} for a detailed analysis of this
projection).

%The next lemmas will be useful to simplify the presentation of the estimator
% for $\|e_u\|_{0,\Omega}$.

\begin{lemma}\label{lemma0-eu}
Let $z\in H^1_{\Gamma_D}(\Omega)$ and $\Pi_0$ the above projection. Then there
holds
%\begin{eqnarray*}
$$
\sum_{T\in {\cal T}_h}\int_{\partial
T}\nabla e_u \cdot \bnu_T (z-\Pi_0 z) \,=\,
\sum_{T\in {\cal T}_h}\int_{\partial
T\setminus\Gamma_D}(\widehat{{\bm\sigma}}\cdot\bnu_T-\nabla u_h
\cdot\bnu_T)\,(z-\Pi_0 z) %\\
%&& 
\,+\, \omega^2\displaystyle\int_{\Omega} e_u\Pi_0 z \,.
$$
%\end{eqnarray*}
\end{lemma}
\begin{proof}
Since $\jump{\nabla u}=\jump{\widehat{\bm{\sigma}}}=0$ on ${\cal E}_I$,  and  $z\in
H_{\Gamma_D}^1(\Omega)$, it holds
\begin{eqnarray*}
\sum_{T\in{\cal T}_h}\int_{\partial T}z\, \nabla
u\cdot\bnu_T\, = \int_{\Gamma_N} z\, g_N=\sum_{T\in{\cal T}_h} \int_{\partial
T\setminus\Gamma_D} z\,\widehat{\bm\sigma}\cdot\bnu_T\,.
\end{eqnarray*}
Observe that,  by taking $v=1$ in \eqref{eq4}, the following identity holds
\[
\disp\int_T f\,+\omega^2\int_T
u_h\,+\,\int_{\partial T} \widehat{\bm{\sigma}}\cdot\bnu_T\,=\,0\,.
\]
Therefore,  integrating by parts and taking into account that ${\Pi}_0 u$ is
constant on each triangle, 
\begin{eqnarray*}
\int_{\partial T}\nabla u\cdot\bnu_T\Pi_0 z&=& \int_T\Delta  u  \:\Pi_0z
=\int_T(-f-\omega^2 u)\Pi_0 z \\
&=& \int_{\partial T\setminus\Gamma_D}
\widehat{\bm{\sigma}}\cdot\bnu_T\, \Pi_0z  -\,\omega^2\int_T(u-u_h)\Pi_0z.
\end{eqnarray*}
(We have used also that $\Pi_0z|_{\Gamma_D}=0$). 
The result follows now readily. 
\end{proof}

\begin{lemma}\label{lemma43}
Let $z\in H^1_{\Gamma_D}(\Omega)$. Then
there holds
%\begin{equation}\label{eqlema42}
\begin{eqnarray*}
\disp\sum_{T\in {\cal T}_h}\int_T\nabla e_u \cdot\nabla z &=& \sum_{T\in
{\cal T}_h}\int_T (f+\omega^2u_h+\Delta u_h)(z-\Pi_0z)\\
&&
+\, \disp\sum_{T\in {\cal T}_h}\int_{\partial T\setminus \Gamma_D}
(\widehat{\bm{\sigma}}\cdot\bnu_T
-\nabla u_h\cdot\bnu_T)(z
-\Pi_0z)\,+\, \disp\omega^2\int_{\Omega}  e_u z\,.
\end{eqnarray*}
%\end{equation}
\end{lemma}
\begin{proof}
We proceed as in Lemma 3.2 in \cite{bcg-2004} (see also \cite{bb-2007} and
\cite{bb-2012}).  Note again that $(\Pi_0z)|_T$ is constant, and 
that $z|_{\Gamma_D}=\Pi_{0}z|_{\Gamma_D}=0$. Thus  
\begin{eqnarray*}
 \disp\sum_{T\in {\cal T}_h}\int_T\nabla e_u\cdot\nabla z &=&
\disp\sum_{T\in {\cal T}_h}\int_T\nabla (u-u_h)\cdot\nabla( z-\Pi_0 z)\\
&=&\,\disp\sum_{T\in {\cal T}_h}\bigg\{\int_T (f+\omega^2 u_h+\Delta u_h)(
 z-\Pi_0 z)\,+\,\omega^2\int_T e_u ( z-\Pi_0 z)\\
&&+\int_{\partial
T} \nabla e_u\cdot\bnu_T(  z-\Pi_0 z)\bigg\}\,.  
\end{eqnarray*}
The proof is finished once  Lemma \ref{lemma0-eu} is applied to bound  the
last
term in equation above.
\end{proof}

\begin{proposition}\label{prop:u}
There exists $ C>0$, independent of the meshsize $h$ and the wave number $\omega$, such that
\[
\|e_u\|^2_{0,\Omega}\le  C^2\,\hat\eta^2:= C^2 \sum_{T\in
{\cal T}_h}\hat\eta^2_T  \,, 
\]
where, for each $T\in {\cal T}_h$, we define
\begin{eqnarray*}
\hat\eta^2_T&:=&\,h^2_T\,\|f+\omega^2u_h+\Delta
u_h\|^2_{0,T}+h_T\,\left\|\widehat{\bm\sigma}\cdot\bnu_T-\nabla u_h\cdot\bnu_T
\right\|^2_{0,\partial T\setminus\Gamma_D}\\
&&
\,+\,\|\alpha^{1/2}\jump{u_h}\|^2_{0,\partial T\cap {\cal
E}_I}+\,\|\alpha^{1/2}(g_D-u_h)\|^2_{0,\partial T\cap {\cal E}_D}\,. 
\label{etahat}
\end{eqnarray*}
\end{proposition}
\begin{proof} 
Take $\phi={\cal L}_\omega e_u\in H^1_{\Gamma_D}(\Omega)\cap H^{1+\varepsilon}(\Omega)$. Then, integrating
by parts and making use of Lemma \ref{lemma43} we can obtain
\begin{eqnarray}
\|e_u\|_{0,\Omega}^2&=&\sum_{T\in {\cal T}_h}\int_T 
e_u (-\Delta \phi-\omega^2 \phi) = \disp\sum_{T\in {\cal
T}_h}\bigg\{\int_T\nabla e_u\cdot
\nabla \phi  - \omega^2\int_T\phi\,e_u-\int_{\partial T\setminus
\Gamma_N}e_u\nabla\phi\cdot\bnu_T\bigg\}\nonumber\\
&=&  \sum_{T\in {\cal T}_h}\bigg\{\int_T(f+\omega^2u_h+\Delta
u_h)(\phi-\Pi_0\phi)  +   \int_{\partial
T\setminus\Gamma_D}(\widehat{{\bm\sigma}}\cdot\bnu_T-\nabla u_h \cdot\bnu_T)(\phi-\Pi_0
\phi) \nonumber \\
&&  
\qquad -\, \int_{\partial T\setminus\Gamma}   \nabla\phi \cdot
\jump{u_h}\bigg\}-\, \int_{{\cal E}_D} 
\nabla\phi \cdot \jump{g_D- u_h}\,. \label{eq:01:teou}
\end{eqnarray}
where in the last step we have applied Lemma \ref{lemma0-eu}, with $z=\phi$,
used the relation $-\Delta u -\omega^2 u=f$ in $\Omega$ and  that  
$\jump{\nabla
\phi}=0$,  on any ${e}\in{\cal E}_I\cup{\cal E}_N$ and $\jump{u}=0$, on any
${e}\in{\cal E}_I$. 

%Thanks to \cite{gs-2004}, it is not difficult to derive that 
Note (cf. \cite{gs-2004}) that
\[
\|w\|_{0,\partial T}\le C_{\tt s} h_T^{s-1/2}\|w\|_{s,T}\qquad\forall\,w\in H^{s}(T)\,,
\]
with  $C_{\tt s}>0$ depending only on $s>1/2$.
%(This result is proved by moving to the reference element by the usual affine
%mapping, applied there the trace inequality
%$\|\widehat{w}\|_{0,\widehat{\partial T}}\le C_s
%\|\widehat{w}\|_{s,\widehat{T}}$ and going back to the original element $T$).
Then, using this bound, as well as the definition of $\alpha$, we  can easily check
\[
 \bigg|\int_{\partial T}  \nabla\phi \cdot
\jump{u_h}\bigg|\le 
\|\alpha^{-1/2}\,\nabla\phi\|_{0,\partial T}\|\alpha^{1/2}\,\jump{u}_h\|_{0,\partial T}\le C'\,
%h^s_T\,
\|\phi\|_{1+\varepsilon,T}\|\alpha^{1/2}\,\jump{u}_h\|_{0,\partial T}.
\]
Applying first this inequality and \eqref{eq:bound:Pi0} to the first
two terms in \eqref{eq:01:teou} and next the Cauchy-Schwarz inequality, we get
%\begin{eqnarray*}
$$
\|e_u\|_{0,\Omega}^2 \le C\sum_{T\in
{\cal T}_h} \hat\eta_T \|\phi\|_{1+\varepsilon,T} \le
C\left[\sum_{T\in
{\cal T}_h} \hat\eta_T^2\right]^{1/2} \|\phi\|_{1+\varepsilon,\Omega}\le 
C\,||{\mathcal L}_{\omega}||_{L^2(\Omega)\to H^{1+\varepsilon}(\Omega)}\,\left[\sum_{T\in
{\cal T}_h} \hat\eta_T^2\right]^{1/2} \|e_u\|_{0,\Omega}\,,
$$
%\end{eqnarray*}
where we have applied in the last step Hypothesis \ref{hypothesis}. The proof
is now completed. 
\end{proof}

To derive  a a posteriori error estimator for  $\triple{u-u_h} $ we will make
use of  the following Helmholtz decomposition.
\begin{lemma}\label{lemma41}
There exist $\psi\in H^1_{\Gamma_D}(\Omega)$  and $\chi\in H^1(\Omega)$ with
$\curl\chi\cdot\bnu=0$ on $\Gamma_N$, such that 
\[
\nabla_h e_u= \nabla \psi +\curl\chi\,.
\]
Furthermore, there holds
\[
\|\nabla\psi\|^2_{0, \Omega}+\|\curl\chi\|^2_{0, \Omega}=
\|\nabla_h(u-u_h)\|^2_{0, \Omega}\,.
\]
\end{lemma}
\begin{proof}
It is consequence of Theorem I.3.1 in \cite{gr-1986}. We refer also to Lemma 3.1
in
\cite{bcg-2004} for more details.
\end{proof}

\begin{lemma}\label{lemma44}
Let $\chi\in H^1(\Omega)$ the function from Lemma {\rm \ref{lemma41}}. Then
there exists $c>0$, independent of $h$ and $\omega$, such that
\begin{eqnarray*}
\sum_{T\in {\cal T}_h}\int_T\nabla e_u\cdot\curl\chi\,&\le&\,c\,\|\curl\chi
\|_{0,\Omega}\,|e_u|_h \\
&=&c\,\Big(\|\alpha^{1/2}\jump{u_h}\|_{0,{\cal
E}_I}^2\,+\,\|\alpha^{1/2}(g_D-u_h)\|_{0, {\cal
E}_D}^2\Big)^{1/2}\|\curl\chi\|_{0, \Omega}\quad\forall\,u\in
H^1({\cal T}_h).
\end{eqnarray*}
\end{lemma}
\begin{proof}
See Lemma 4.4 in \cite{bb-2007}.
\end{proof}

We are ready to derive the a posteriori estimate for $\triple{u-u_h} $ and
$\|\bm{\sigma}-\bm{\sigma}_h\|_{0,\Omega}$.

% \begin{theorem}\label{theo:gradu}
% There exists $\bar C>0$, independent of the meshsize and $\omega$, such that
% \[
% \triple{u-u_h}^2+ \|\bm{\sigma}-\bm{\sigma}_h\|_{0,\Omega}^2\le \bar
% C(1+\omega^2)^2\,\bar\eta^2:=\bar C\bigg\{\sum_{T\in {\cal T}_h}\bar\eta^2_T
% \bigg\}\,, 
% \]
% where, for each $T\in {\cal T}_h$, 
% \begin{eqnarray}\bar\eta^2_T&:=&\,h^2_T\,\|f+\omega^2u_h+\Delta
% u_h\|^2_{0,T}+h_T\,\left\|\widehat{\bm\sigma}\cdot\bnu_T-\nabla u_h\cdot\bnu_T
% \right\|^2_{0,\partial T\setminus\Gamma_D}\nonumber\\
% &=&\
% \,+\,\|\alpha^{1/2}\jump{u_h}\|^2_{[0,\partial T\cap {\cal
% E}_I}+\,\|\alpha^{1/2}(g_D-u_h)\|^2_{0,\partial T\cap {\cal E}_D}\,.
% \label{eqaph1}
% \end{eqnarray}
% \end{theorem}
% \begin{proof}
\noindent{\bf \em Proof of \eqref{eq:01:teo-apost} of Theorem \ref{theo:apost}.}
Notice that since
\[
 \|\bm{\sigma}-\bm{\sigma}_h\|_{0,\Omega}\le C\triple{u-u_h}\,,
\]
(see \eqref{eq:sigma-nablau}), it suffices to bound $\triple{u-u_h}$.

We then proceed as in Theorem 3.2 in \cite{bcg-2004}  (see also \cite{bb-2012}).
Since
$\jump{u}=0$ in ${\cal E}_I$ and $g_D$ in ${\cal E}_D$, we deduce 
\begin{equation}\label{eq:00:theo:gradu}
\triple{u-u_h}^2=\triple{e_u}^2\,=\,\|\nabla_h e_u \|^2_{0,
\Omega}+\|\alpha^{1/2}\jump{u_h}\|^2_{0,{\cal E}_I}\,+\, 
\|\alpha^{1/2}(g_D-u_h)\|^2_{0, {\cal E}_D}\,.
\end{equation}
Besides,
\begin{equation}
 \label{eq:01:theo:gradu}
 \|\nabla_h e_u \|^2_{0,
\Omega}= \sum_{T\in{\cal T}_h}\left\{\int_T\nabla e_u\cdot\nabla\psi\, +\,
\int_T\nabla e_u\cdot\curl\chi\right\}\,.
\end{equation}

The first term can be bounded as follows:   
\begin{eqnarray*}
\sum_{T\in{\cal T}_h}\int_T\nabla
e_u\cdot\nabla\psi &=&  \sum_{T\in{\cal T}_h}\bigg\{\int_T
(f+\omega^2 u_h+\Delta
u_h)(\psi-\Pi_0\psi)+\int_{\partial T\setminus \Gamma_D}
\hspace{-4pt}(\widehat{\bm\sigma}\cdot\bnu_T-\nabla
u_h\cdot\bnu_T)  (\psi
-\Pi_0\psi) \bigg\}\\
&& \, +\omega^2
\|e_u\|_{0,\Omega}\|\psi\|_{0,\Omega}\\
&&\hspace{-48pt}\le\: C \bigg( \bigg[\sum_{T\in{\cal T}_h}\Big\{ \Big(h_T
\|f+\omega^2 u_h+\Delta
u_h\|_{0,T}+h_T^{1/2 }\|(\widehat{\bm\sigma}\cdot\bnu_T-\nabla
u_h\cdot\bnu_T)\|_{0,\partial T}\Big)\|\nabla \psi\|_{0,T}\Big\}\bigg]^2  \\
&&  +C^2_{\rm
P}\omega^4
\|e_u\|_{0,\Omega}^2\|\nabla \psi\|^2_{0,\Omega}\bigg)^{1/2}\\
& &\hspace{-48pt}\le\: C'(1+\omega^2) \Bigg[  \sum_{T\in{\cal T}_h}\eta_T^2
\Bigg]^{1/2}\|\nabla \psi\|_{0,\Omega}\,,
\end{eqnarray*}
where we have applied sequentially Lemmas \ref{lemma43} and \ref{lemma41},
estimate \eqref{eq:bound:Pi0}, Poincar\'e inequality \eqref{eq:poincare}, the
Cauchy-Schwarz inequality and, finally, to bound the
term $\|e_u\|_{0,\Omega}$, Proposition \ref{prop:u}.

Using this result,  and Lemma \ref{lemma44} in
\eqref{eq:01:theo:gradu} we derive
\begin{eqnarray}
\|\nabla_h e_u\|_{0,\Omega}^2&\le&  C(1+\omega^2)\ \left[\sum_{T\in
{\cal T}_h} \eta^2_T \right]^{1/2}\big(\|\nabla
\psi\|^2_{0,\Omega}+\|\curl\chi\|^2_{0, \Omega}\big)^{1/2}\nonumber \\
&=&  C (1+\omega^2) \Bigg[\sum_{T\in
{\cal T}_h} \eta^2_T \Bigg]^{1/2}\|\nabla_h
e_u\|_{0,\Omega}\label{eq:02:theo:gradu}.
\end{eqnarray}
Inserting \eqref{eq:02:theo:gradu} in \eqref{eq:00:theo:gradu}, the result is
proven. \hfill $\Box$.
%\end{proof}

\subsection{Quasi-efficiency of the estimator}
In this subsection we prove the quasi-efficiency of the estimator (cf.
$(\ref{eq:03:teo-apost})$).
For simplicity, we assume that each element of $\{{\cal T}_h\}_{h>0}$ has not
hanging node. 

We begin with some notations and preliminary results. For each
$T\in {\cal T}_h$ and $e$ an edge of $T$, we will 
denote in this section only by $\mathbb{P}_m(T)$ and
$\mathbb{P}_m(e)$ the spaces of polynomials on $T$ and $e$ respectively of
degree $m$. On the other hand, we
denote by $\psi_T$ and $\psi_e$ the standard
triangle-bubble Êand edge-bubble functions, respectively. In particular, $\psi_T$
satisfies 
\[
\psi_T\in\mathbb{P}_3(T),\quad \mathrm{supp}(\psi_T)\subseteq T,\quad
\text{with }0\leq \psi_T\leq 1 \quad\text{and }\quad \psi_T|_{\partial T}=0\,.
\]
Similarly,
\[
\psi_e|_T\in\mathbb{P}_2(T),\quad
\mathrm{supp}(\psi_e)\subseteq \cup\{T'\in{\cal T}_h\ :\
e\subset\partial T'\},\quad \text{with $0\leq\psi_e\leq1$ and $
\psi_e|_{\partial T\setminus e}=0$.}
\]
We also recall from \cite{v-1996} that, given $k\in\N\cup\{0\}$, there exists an
extension operator $L:C(e)\to C(T)$ that satisfies 
\[
Lp\in\mathbb{P}_m(T),\qquad  Lp\big|_e=p, \quad\forall\,p\in\mathbb{P}_m(e).
\]
Additional properties of $\psi_T$, $\psi_e$, and $L$ are listed here cf.
\cite[Lemma 1.3]{v-1996}: There exist $c_1,c_2,c_3,c_4>0$, independent of the mesh size, so that
\begin{eqnarray}
  \|\psi_T\,q\|_{0,T}^2 &\le& \|q\|_{0,T}^2 \qquad\quad \leq
c_1\,\|\psi_T^{1/2}\,q\|_{0,T}^2\qquad \forall\,q\in\mathbb{P}_m(T)\,, \label{prop1} \\
  \|\psi_e\,p\|_{0,e}^2 &\le& \|p\|_{0,e}^2  \qquad\quad \le          
c_2\,\|\psi_e^{1/2}\,p\|_{0,e}^2 \qquad \forall\,p\in\mathbb{P}_m(e)\,, \label{prop2} \\
  c_4\,h_e\,\|p\|_{0,e}^2 & \le & \|\psi_e^{1/2}\,Lp\|_{0,T}^2 \le
c_3\,h_e\,\|p\|_{0,e}^2 \qquad\quad\forall\,p\in\mathbb{P}_m(e)\,. \label{prop3}
\end{eqnarray}
 
% We also recall the following inverse inequalities  
% \begin{equation}\label{inv-est}
% |q|_{H^m(T)} \leq c\,h_T^{l-m}|q|_{H^l(T)}\quad\forall\,q\in\mathbb{P}_k(T)\,.
% \end{equation}
% \end{lemma}
% \begin{proof}
% See Theorem 3.2.6 in \cite{c-1978}.
% \end{proof}

%%%%%%%%%%%%%%%%%%%%%%%%%%%%%%%%%%%%%%%%%%%%%
Our aim now is to estimate these five terms which define the error indicator $\eta_T^2$
cf. \eqref{eq:02:teo-apost}.  Observe 
that we can bound three of them straightforwardly, namely
\begin{eqnarray}\label{est1}
\|\nabla u_h-\bm{\sigma}_h\|_{0,T}& \leq& \|\bm{\sigma}-\bm{\sigma}_h\|_{0,T} +
\|\nabla e_u\|_{0,T}\qquad\forall\,T\in{\mathcal T}_h \\
\label{est2}
\|\alpha^{1/2}\jump{u_h}\|_{0,e} &=&
\|\alpha^{1/2}\jump{e_u}\|_{0,e}\qquad\qquad\qquad\quad\forall\,e\in{\cal E}_I\,,\\
\label{est3}
\|\alpha^{1/2}(g_D-u_h)\|_{0,e} &=&
\|\alpha^{1/2}(u-u_h)\|_{0,e}\qquad\qquad\quad\forall\,e\in{\cal E}_D\,,
\end{eqnarray}
where as usual $e_u=u-u_h$.

From  here on, we introduce $f_h=\Pi_{V_h} f$ so that $\|f-f_h\|_{0,T}$ goes to
zero with, at least, the same rate as $\triple{e_u}$ as the mesh is refined.

\begin{lemma}\label{lemma:est4}
There exists $C>0$, independent of the mesh size and  $\omega$, such that for any $T\in{\cal
T}_h$
\begin{equation}\label{est4}
\begin{array}{c}
\ds h_T^2\|f+\omega^2 u_h+\Delta u_h\|_{0,T}^2 \leq \ds C\left(\|\nabla
e_u \|_{0,T}^2+\omega^4h^2_T\|e_u\|^2_{0,T}+h_T^2\|f-f_h\|_{0,T}^2\right)\,.
\end{array}
\end{equation}
\end{lemma}
\begin{proof}
Let $v_h:=f_h+\omega^2u_h+\Delta u_h$, and $v_b:=\psi_T\,v_h$. Then 
\begin{eqnarray*}
c_1^{-1}\,\|v_h\|_{0,T}^2 &\le& \|\psi_T^{1/2}v_h\|_{0,T}^2
=\int_T(v_h+f)v_b-\int_T fv_b =-\int_T\Delta
e_u\:v_b -\omega^2\int_T e_u v_b+\int_T(f_h-f) v_b\\
 &=& \int_T \nabla e_u \cdot\nabla v_b - \omega^2\int_T e_u v_b+
\int_T(f_h-f)v_b\,.
\end{eqnarray*}
Noting that $v_b|_{\partial T}=0$, the inverse inequality 
\begin{equation}\label{eq:inv}
\|\nabla v_b\|_{0,T}\le C h_{T}^{-1}\|v_b\|_{0,T}\le C h_T^{-1}||v_h||_{0,T}^2
\end{equation}
yields now
\begin{eqnarray*}
 \|v_h\|_{0,T}^2&\le& C\big(
 \omega^2\|e_u\|_{0,T}+h_T^{-1}\|\nabla e_u\|_{0,T}+\|f-f_h\|_{0,T}
\big)\|v_h\|_{0,T}. 
\end{eqnarray*}
The proof is finished by noting now 
\[
 \|f+\omega^2 u_h+\Delta u\|_{0,T}^2 \le  2\|f-f_h\|^2_{0, T}+
2\|v_h\|_{0,T}^2\le C\big[
 \omega^4\|e_u\|_{0,T}^2+h_T^{-2}\|\nabla e_u\|_{0,T}^2+\|f-f_h\|^2_{0,T}
\big].
\]
\end{proof}

\begin{lemma}\label{lemma:est6}
Let $e\in{\cal E}_I$, and let $T',T\in{\cal T}_h$ so that $T\cap T'=e$. Then,
there exists $C>0$, independent of the mesh size and  $\omega$, such that 
\begin{eqnarray*}
\ds h_e\,\|\widehat{\bm{\sigma}}\cdot\bnu_T-\nabla u_h\cdot\bnu_T\|_{0,e}^2
\ds &\leq& C_3\Big(\|\alpha^{1/2}\jump{u-u_h}\|_{0,e}^2
+ \|\nabla u-\nabla_hu_h\|_{0,\mathcal{N}(T)}^2+\|\bm{\sigma}-\bm{\sigma}_h\|_{0,\mathcal{N}(T)}^2\\
&& \,+\, 
\omega^4\max\{h_T,h_{T'}\}^2\|u-u_h\|_{0,\mathcal{N}(T)}^2 + \max\{h_T,h_{T'}\}^2\|f-f_h\|_{0,\mathcal{N}(T)}^2\Big)\,.
\end{eqnarray*}
\end{lemma}
\begin{proof}
%We adapt the proof of Lemma 4.11 in \cite{bb-2007}.
%using the definition of the numerical flux $\widehat{\bm\sigma}$ (cf.
%$(\ref{ldg2})$) we notice that
%\[
%\widehat{\bm{\sigma}}\cdot\bnu_T-\nabla u_h\cdot\bnu_T \,=\, 
%-\big({\textstyle\frac{1}{2}}+\bm{\beta}\cdot\bnu_T\big)\jump{\bm{\sigma}
%_h-\bm { \sigma } } 
%- \alpha  \jump{u_h}\cdot\bnu_T\,+\,(\bm{\sigma}_h\cdot\bnu_T-\nabla
%u_h\cdot\bnu_T)\,,
%\]
%and then, applying Cauchy-Schwarz inequality, we have
It can be easily checked that
\begin{equation}\label{eq:01:est6}
\|\widehat{\bm\sigma}\cdot\bnu_T-\nabla u_h\cdot\bnu_T\|_{0,e} \,\leq\,C\,
\big\{\|\jump{\bm{\sigma}_h}\|_{0,e}+\|\alpha^{1/2}\jump{u_h}\|_{0,e}
+ \|\bm { \sigma}_h\cdot\bnu_T-\nabla u_h\cdot\bnu_T\|_{0,e}\big\}\,. 
\end{equation}
Clearly, it only remains to bound  the first and  third  term in the inequality
above.
For the third term, we denote for the sake of a simpler notation 
\[
\lambda_h:=\bm{
\sigma}_h\cdot\bnu_T-\nabla u_h\cdot\bnu_T.
\]
Then, using the property \eqref{prop2} and integrating by parts, we have
\begin{eqnarray*}
 c_2^{-1}\,\|\lambda_h \|_{0,e}^2 &\le& 
\|\psi_e^{1/2}\lambda_h\|_{0,e}^2= \int_T\div(\bm{\sigma}_h-\nabla u_h)
(L \lambda_h)\,\psi_e  +
\int_T(\bm{\sigma}_h-\nabla u_h)\cdot\nabla\big(\psi_eL \lambda_h\,\big)\,. 
\end{eqnarray*}
The Cauchy-Schwarz inequality,  inverse inequality \eqref{eq:inv} and
\eqref{prop3}  yield  
\begin{eqnarray}
\|\lambda_h\|_{0,e}^2 &\le&\,C\,\big\{h_T^{1/2}\|\div(\bm{\sigma}_h-\nabla
u_h)\|_{0,T}\,+\,h_T^{-1/2} \|\bm{\sigma}_h-\nabla u_h\|_{0,T}
\big\}\,\|\lambda_h\|_{0,e}\nonumber\\
&\le&\,C\,(C_{\tt ineq}+1)\,  h_T^{-1/2}\|\bm{\sigma}_h-\nabla
u_h\|_{0,T} \|\lambda_h\|_{0,e}\,,\nonumber
\end{eqnarray}
where in the last step we have also used  cf. \cite[Corollary 1]{lm-2009}
\begin{equation}
\label{eq:02:est6}
\|{\rm div}({\bm \tau})\|_{0,T}\le  C_{\tt ineq} \,h_T^{-1}\|\bm{\tau}\|_{0,T},\quad \forall\,T\in{T}_h\quad\forall\,
\bm{\tau}\in \bm{\Sigma}_h\,,
\end{equation}
with $C_{\tt ineq}>0$ is independent of the mesh size. Regarding the first term $\|\jump{\bm{\sigma}_h}\|_{0,e}$ in
\eqref{eq:01:est6} and denoting $w_h:=\jump{\bm{\sigma}_h}\in \mathbb{P}_m(e)$, we
deduce
\begin{eqnarray*}
c_2^{-1}\|w_h\|_{L^2(e)}^2&\leq& \|\psi_e^{1/2}w_h\|_{L^2(e)}^2\,=\,
\int_e\psi_e Lw_h\,\jump{\bm{\sigma}_h-\bm{\sigma}}
\ds\\
&=& \int_{T\cup T'}\div(\bm{\sigma}_h-\bm{\sigma})\psi_e Lw_h\,
\,+\,\int_{T\cup T'}(\bm{\sigma}_h-\bm{\sigma}
)\cdot\nabla\big(\psi_e Lw_h\,\big)\\
&=& \int_{T\cup T'}\div(\bm{\sigma}_h-\nabla_h u_h)\psi_e Lw_h\,+
\int_{T\cup T'}(\Delta_h u_h+\omega^2 u_h+f)\psi_e Lw_h\,\\
&&+\,
\omega^2\int_{T\cup T'}(u-u_h)\psi_e Lw_h\,
+ \int_{T\cup T'}(\bm{\sigma}_h-\bm{\sigma}
)\cdot\nabla\big(\psi_e Lw_h\,\big)\,. %\\
\end{eqnarray*}
Next, we bound each of the four integrals per element, applying Cauchy-Schwarz and some properties such as inverse inequality, the ones given in \eqref{prop1}-\eqref{prop3}, and/or \eqref{eq:02:est6}. For example, for the element $T$, we derive
$$
\int_T\div(\bm{\sigma}_h-\nabla u_h)\psi_e Lw_h\,\leq\,c\,\frac{h_e^{1/2}}{h_T}\,||\bm{\sigma}_h-\nabla u_h||_{0,T} \,||w_h||_{0,e}\,,
$$
$$
\int_T(\Delta u_h+\omega^2 u_h+f)\psi_e Lw_h \,\leq\,c\, h_e^{1/2}||\Delta u_h+\omega^2 u_h+f||_{0,T}\,||w_h||_{0,e}\,,
$$
$$
\int_T(u-u_h)\psi_e Lw_h \,\leq\,c\, h_e^{1/2}||u-u_h||_{0,T}\,||w_h||_{0,e}\,,
$$
and
$$
\begin{array}{rcl}
\displaystyle \int_T(\bm{\sigma}_h-\bm{\sigma}
)\cdot\nabla\big(\psi_e Lw_h\,\big) &\leq& \displaystyle c\,\frac{h_e^{1/2}}{h_T}\Big(||\bm{\sigma}_h-\nabla u_h||_{0,T}\,+\,||\nabla u-\nabla u_h||_{0,T} \\
&& \displaystyle \,+\,\omega^2\,h_T||u-u_h||_{0,T} \,+\,h_T||f-f_h||_{0,T}\Big)||w_h||_{0,e}\,.
\end{array}
$$

%Using again the inverse inequality \eqref{eq:02:est6} and the identity we deduce 
%\begin{eqnarray}
% \|\jump{\bm{\sigma}_h}\|_{0,e}^2&\le& C
%h_T^{1/2}\big\{\|\div(\bm{\sigma}_h-\nabla u_h)\|_{0,T\cup T'}+
%\|\Delta u_h+\omega^2 u_h+f\|_{0,T\cup T'}\nonumber
% \\
%&& +\omega^2\|e_u\|_{0,T\cup T'}+
%\|\bm{\sigma}_h-\bm{\sigma}\|_{0,T\cup T'}
% \big\}\|\jump{\bm{\sigma}_h}\|_{0,e}\nonumber\\
% &\le&  C' h_T^{-1/2}\big\{
%\|\bm{\sigma}_h-\bm{\sigma}\|_{0,T\cup T'}+
%\|{\bm \sigma}-\nabla u_h\|_{0,T\cup T'}\big\}
% \|\jump{\bm{\sigma}_h}\|_{0,e}\nonumber\\
%&&+ C' h_T^{1/2}\big\{
%\|\Delta u_h+\omega^2 u_h+f\|_{0,T\cup T'} +\omega^2\|e_u\|_{0,T\cup T'}+
%\|\bm{\sigma}_h-\bm{\sigma}\|_{0,T\cup T'} \big\}
% \|\jump{\bm{\sigma}_h}\|_{0,e}\qquad
%\label{eq:03:est6}
%\end{eqnarray}
The result is obtained after we summing up the corresponding estimates for $T$ and $T'$, and applying Lemma \ref{lemma:est4}, to bound $
\|\Delta u_h+\omega^2 u_h+f\|_{0,T}$. We omit further details.
%, % in \eqref{eq:03:est6},  and combining with \eqref{eq:02:est6} in \eqref{eq:01:est6}, the result is proven.  

\end{proof}

%\noindent{\bf \em Proof of \eqref{eq:02:teo-apost} of Theorem
%\ref{theo:apost}.} The result follows readily from
%\eqref{est1}, \eqref{est2}, \eqref{est3} and Lemmas
%\ref{lemma:est4} and \ref{lemma:est6} for $e\in{\cal E}_i$ and that 
%
%.\hfill
%$\Box$
%%%%%%%%%%%%%%%%%%%%%%%%%%%%%%%%%%%%%%%%%%%%%%%%%%%%%%%%%%%%%%%%%%%%%%%%%%%%%%%%
%%%%%%%%%%
\section{Numerical examples}\label{section5}
%%%%%%%%%%%%%%%%%%%%%%%%%%%%%%%%%%%%%%%%%%%%%%%%%%%%%%%%%%%%%%%%%%%%%%%%%%%%%%%%
%%%%%%%%%%
In this Section we show the performance of the method with 
the  $\mathbb P_1-[\mathbb P_1]^2$
 approximation. The code has been written in {\sc Matlab} and run in a {\it Pentium Xeon computer with dual processor}. In what follows,
$N$ stands for the total number of degrees of freedom (unknowns) of
$(\ref{LDG-form1})$. Hereafter, the individual and total errors are denoted  as
follows
%We measure the errors
\[
\bm{e}_h(u):=\triple{u-u_h}, \quad
\bm{e}(\bm{\sigma}):=\|\bm{\sigma}-\bm{\sigma}_h\|_{0,\Omega}, 
\]%\quad
\[
\bm{e}_0(u):=\|u-u_h\|_{0,\Omega}, \quad
\bm{e}:=\Big(\bm{e}_h(u)^2+\bm{e}_0(\bm{\sigma})^2\Big)^{1/2}\,,
\]
where $(\bm{\sigma},u)\in[L^2(\Omega)]^2\times H^1(\Omega)$ and
$(\bm{\sigma}_h,u_h)\in \bm{\Sigma}_h\times V_h$ are the unique solutions of the
continuous and discrete formulations, \eqref{eq2} and \eqref{LDG-form1}, respectively. In addition, if $\e$ and
$\tilde\e$ stand for the errors at two consecutive triangulations with $N$ and
$\tilde N$ degrees of freedom, respectively, then the experimental rate of
convergence is given by $\disp r:=-2\frac{\log(\e/\tilde\e)}{\log(N/\tilde N)}$.
The definitions of $r_h(u)$, $r_0(\bm{\sigma})$, and $r_0(u)$ are given in
analogous way.  Finally, by $\e/\eta$ we measure the effectivity index.

We now specify the data of the three examples to be presented here. 
We take $\Omega$ as either the square $]0,1[^2 $ (for Example 1) or the
L-shaped domains $ ]-1,1[^2\,\,\backslash\,\,[0,1]^2$ for Example 2 and  $
]-1,1[^2\,\,\backslash\,\,[0,1]\times[-1,0]$ for Example 3. 
For  Example 2  we define
$\overline{\Gamma}_D:=\{-1\}\times[-1,1]\cup\{1\}\times[-1,0]\cup\{0\}\times[0,1
]$, and we consider $\overline{\Gamma}_D:=\{0\}\times[-1,0]\cup[0,1]\times\{0\}$
for Example 3. In all these examples, the data $f$, $g_D$ and/or $g_N$ are
chosen so that the exact solution $u$ is the one shown in Table 5.1.  We
emphasize that the solution $u$ of Example 1 is smooth, while the one of
Example 3 (given in polar coordinates) lives in $H^{1+2/3}(\Omega)$, since their derivatives are singular at $(0,0)$. 
This implies that $\div(\bm{\sigma})\in H^{2/3}(\Omega)$ only, which, according
to  Theorem $\ref{theo:mainLDG}$, yields $2/3$ as the expected rate of
convergence for the uniform refinement.  
%Finally $r_h(u)$, $r_0(\bm{\sigma})$, $r_0(u)$ and $r$ denote
%the corresponding experimental rates of convergence for the given
%error measures.

\begin{center}

{\bf Table 5.1}. Summary of data for the three examples.\\

\begin{tabular}{|c|c|c|c|c|}
\hline

{\sc Example}  & {\sc Domain $\Omega$}  &  {\sc B.C.} & $\omega$  &{\sc Solution $u$}   \\
\hline
 %& &  & \\
&    & & 1.0    &  \\[0.3ex] 
  &   &    & 4.0    & \\[0.3ex]
  &   &   & 4.44    &  \\[0.3ex]
1   & {\sc Square} &   {\sc Dirichlet}   & 4.5  & $\sin(\pi x_1)\sin(\pi x_2)$ \\[0.3ex]
  &    &  & 5.0    &  \\  [0.3ex]
 & &  & 10.0 & \\[0.3ex]
 & & & 15.0 & \\ 
\hline
%& & & \\
% $\{(2,x_2)\,:\,x_2\in[0,1]\} \cup \{(1,x_2)\,:\,x_2\in[1,2]\} \cup \{(0,x_2)\,:\,x_2\in[0,2]\}$
& & & 1.0 &  \\[0.3ex] 
 2  & {\sc L-shaped}  &   {\sc Mixed}  & 10.0 &  $\displaystyle\frac{1}{1.1-x_1}$\\
  &  & & 15.0 & \\
\hline

%& & & \\
%  $\{(0,x_2)\,:\,x_2\in[-1,0]\} \cup \{(x_1,0)\,:\,x_1\in[0,1]\}$
&    & &    1.0       &   \\[0.3ex] 
3   & {\sc L-shaped}  & {\sc Mixed}   &     10.0  & $\displaystyle r^{2/3}\sin\left(\frac{2}{3}\theta\right)$  \\[0.3ex]
  &  &   &    15.0  &  \\ %[0.3ex]
  %&   &    0.5  &  \\[0.3ex]
  %&   &    0.5  &  \\  
% &    &   &  \\

\hline
\end{tabular}

\end{center}

The aim of  the numerical experiments for Example 1 is to show the robustness of the  scheme for different values of the wave number $\omega$ when using uniform refinements. Specifically, Tables 5.2 and 5.6 contain 
the results obtained for $\omega\in\{1,5\}$, for which $\omega^2$ is far from
the first eigenvalue of the problem: $2\pi^2$, while in Tables 5.3-5.5 are shown
the results for $\omega\in\{4,4.44,4.5\}$, such that $\omega^2$ is closer to
$2\pi^2$. In all the cases we observe that the $\bm{e}_h(u)$ and
$\bm{e}(\bm{\sigma})$ behave as ${\cal O}(h)$, as expected, while the $L^2-$
error norm of $u$ $(\bm{e}_0(u))$ converges at a rate of order ${\cal O}(h^2)$,
which is expected too, but have not been proved here. 
We remark, since $\omega\in\{4.44,4.5\}$ is very close to $\sqrt{2}\pi$, the
method requires smaller mesh size to behave well, in agreement with the theory
(see Tables 5.4 and 5.5). In addition, with the purpose of showing the robustness of the scheme for moderately large value of wave number, we summarize in Tables 5.7 and 5.8 the individual and global errors (including the $L^2-$ error norm $\bm{e}_0(u)$) and the effectivity index $\bm{e}/\eta$ obtained for  $\omega\in\{10,15\}$,
considering Example 1. In both two cases, the choice for $\omega$
requires a smaller mesh size to get an appropriate approximation of the exact
solution, which are in agreement with the theory. Moreover, the rate of convergence of
$\bm{e}_h(u)$ and $\bm{e}(\bm{\sigma})$ is the expected: ${\cal O}(h)$.  We remark that the ${\mathcal O}(h^2)$ behavior of $\bm{e}_0(u)$ using uniform refinement for all examples and choices of $\omega$ considered in this work has not been proved here, but we think that it should be derived by a standard duality argument.  Furthermore, in order to describe the behavior of the estimator for different values of the wave number $\omega$, in Tables 5.2-5.8 we have included a column with the effectivity indexes. We note that in all cases these indexes remain constants, which is in accordance with the reliability and local efficiency proved here.  

Since we have developed an a posteriori error estimator  $\eta$ in Theorem \ref{theo:apost}, we use it to develop an adaptive procedure in order to improve the quality of the initial approximation by refining the zones of $\bar\Omega$ where the (local) estimator dominates over (part of) the rest. On the other hand, it is well known that in order to avoid the {\it pollution effect} we should start with a {\em coarse} mesh satisfying $\omega\,h<1$. To the aim of comparing the adaptive and uniform strategies, we consider the following algorithm, which we called  {\it hybrid adaptive} algorithm:
\begin{enumerate}
\item[1.] Start with a coarse mesh ${\cal T}_h$.
%\item[2.] Solve the Galerkin scheme \eqref{LDG-form1} for the current mesh $\mathcal T_h$. 
\item[2.] Perform uniform refinements to ${\mathcal T}_h$ until the resulting mesh satisfies $\omega\,h<1$. Define this mesh as the new $\mathcal T_h$ and go to next step.
%Check thumb of the rule. If $\omega h > 1$  use uniform refinement to define the resulting mesh a new  $\mathcal T_h$ and then come back step 2. Else go to next step.
\item[3.] Solve the Galerkin scheme \eqref{LDG-form1} for the current mesh $\mathcal T_h$.
\item[4.] Compute $\eta_T$ for each triangle $T\in\mathcal T_h$.
\item[5.] Consider stopping criterion and decide to finish or go to the next step.
\item[6.] Use {\it newest vertex bisection} procedure to refine each element $T'\in{\mathcal T}_h$ such that 
$$ %\disp
\eta_{T'}\,\ge\, \frac{1}{4}\,\max\{\eta_{T}\,:\,T\in\mathcal T_h\}\,.
$$
%\item[5.] Use {\it D\"orfler} marking criteria: given $\gamma\in(0,1]$, find $\mathcal{M}\subseteq\mathcal{T}_h$ such that
%$$
%\gamma\eta \,\leq\,\left(\sum_{T\in\mathcal{M}}\eta_T^2\right)^{1/2} \,,
%$$
%This procedure is known as guaranteed error reduction strategy (GERS). We refer to \cite{dorfler-1996} for more details.
%and apply then bisection technique to refine the elements of $\mathcal{M}$, with the propagation refinement, if corresponds.
\item[7.] Define the resulting mesh as the new $\mathcal T_h$ and go to step 3.
\end{enumerate}

In practice, we should start with a {\it suitable} coarse mesh ${\mathcal T}_h$, satisfying $\omega h<1$, and go to step 3 in the proposed adaptive algorithm. However, we perform the described algorithm in order to compute the errors and effectivity indexes from a coarsest mesh not verifying the condition on $\omega h$ a priori and compared them with results obtained by the adaptive refinement.

We apply this algorithm to Examples 2 and 3. Respect to Example 2, it is not difficult to check that its exact solution has a singularity on the line $x_1=1.1$, which at the discrete level results in a numerical singularity on the edge $\{1\}\times[-1,0]$. In Tables 5.9-5.11 we resume the behavior of the individual and total errors, as well as the index of efficiency, after performing uniform refinement for different values of the wave number $\omega$. In all the cases we observe that the method converges with the optimal rate, and the effectivity index remains bounded.  Tables 5.15-5.17 contain the respective output when the proposed adaptive algorithm is applied. In this case, we observe that the method also converges at the same optimal rate, but it is able to detect the numerical singularity in the neighborhood of $\{1\}\times[-1,0]$, which yields to a boundary layer close to the referred part of $\partial\Omega$.  This is better described in Figures 5.1-5.3, where a comparison between the total error obtained performing uniform and adaptive refinements is shown, for $\omega\in\{1,10,15\}$. Thanks to the numerical singularity, the adaptive procedure is focused on the large error region and then it improves the quality of the approximation. Some intermediate adapted meshes for each value of $\omega$ considered here, are included in Figures 5.7-5.9, where the corresponding boundary layer is recognized and localized.

Now, concerning Example 3, we point out that the gradient of the exact solution is singular at the origin $(0,0)$,  so the expected rate of convergence of the method is of order $\mathcal O(h^{2/3})$, for moderate values of the wave number. This is confirmed when performing uniform refinements and can be noticed in Tables 5.12-5.14, where, in addition, the index of efficiency is bounded in all cases. As in the previous example, the adaptive refinement is able to detect the singularity region, and  we observe that the optimal rate of convergence is achieved, for different values of $\omega$ (cf. Tables 5.18-5.20).  Again, in all these cases, the effective index remains bounded. Figures 5.4-5.6 shows the improvement of the quality of the approximation when using adaptivity, for all the values of $\omega$ we considered. Some adapted meshes for the wave numbers we set here, are displayed in Figures 5.10-5.11, which exhibit the localization of the singularity. In addition, it is important to mention that this example behaves as predicted by our theory, despite the fact that it does not rely in it: the geometric condition on the mixed boundary $\partial\Omega$ is not satisfied in the present case  so we can not ensure that the adjoint problem has a smooth enough solution (cf. Grisvard \cite{Grisvard}). This example gives us, therefore, numerical evidence to conjecture that our results could be proved without using that geometric assumption.

\section*{Conclusions and final comments}
Summarizing, the numerical results presented here underline the reliability and efficiency of the a posteriori error estimator $\eta$, and strongly show that the associated hybrid adaptive algorithms are much more suita\-ble than a uniform discretization procedure when solving problems with non-smooth solutions. We notice that in all the examples we considered, the effectivity index does not behave as the current analysis predicts: ${\cal O}(\omega^{2})$. This gives us some numerical evidence that the behavior of this index could be overestimated and could be the subject of future research. 

%due to the behavior of the reliability and efficiency constants (see Section 4), with respect to the wave number, is given by ${\cal O}(\omega^{2})$, then we expect a lack of convergence when $\omega$ increases. Furthermore, we note that this effect is not local. 
%\newpage

\begin{center}
{\bf Table 5.2}. Example 1: uniform refinement with $\omega=1.0$
\begin{tabular}{|c||c|c||c|c||c|c||c|c||c|} 
\hline  
$N$ 	 & 	 $\e(u)$ 	 & 	 $r(u)$ 	 & 	 $\e(\bm{\sigma})$ 	 & 	 $r(\bm{\sigma})$ 	 & 	 $\e$ 	 & 	 $r$ 	 & 	$ \e_0(u)$ 	 & 	 $r_0$ 	 & 	 $\e/\eta$ 	 \\
\hline  
    36 	 & 	 1.667e+00 	 & 	 ----- 	 & 	 1.119e+00 	 & 	 ----- 	 & 	 2.008e+00 	 & 	 ----- 	 & 	 1.629e-01 	 & 	 ----- 	 & 	     0.1721 	 \\ 
   144 	 & 	 9.355e-01 	 & 	     0.8337 	 & 	 4.822e-01 	 & 	     1.2146 	 & 	 1.052e+00 	 & 	     0.9320 	 & 	 5.229e-02 	 & 	     1.6394 	 & 	     0.1721 	 \\ 
   576 	 & 	 4.863e-01 	 & 	     0.9438 	 & 	 2.803e-01 	 & 	     0.7829 	 & 	 5.613e-01 	 & 	     0.9069 	 & 	 1.552e-02 	 & 	     1.7524 	 & 	     0.1813 	 \\ 
  2304 	 & 	 2.441e-01 	 & 	     0.9944 	 & 	 1.585e-01 	 & 	     0.8225 	 & 	 2.910e-01 	 & 	     0.9476 	 & 	 4.304e-03 	 & 	     1.8505 	 & 	     0.1912 	 \\ 
  9216 	 & 	 1.222e-01 	 & 	     0.9983 	 & 	 8.318e-02 	 & 	     0.9298 	 & 	 1.478e-01 	 & 	     0.9773 	 & 	 1.121e-03 	 & 	     1.9414 	 & 	     0.1960 	 \\ 
 36864 	 & 	 6.117e-02 	 & 	     0.9982 	 & 	 4.237e-02 	 & 	     0.9734 	 & 	 7.441e-02 	 & 	     0.9903 	 & 	 2.846e-04 	 & 	     1.9772 	 & 	     0.1981 	 \\ 
147456 	 & 	 3.061e-02 	 & 	     0.9988 	 & 	 2.134e-02 	 & 	     0.9891 	 & 	 3.732e-02 	 & 	     0.9956 	 & 	 7.162e-05 	 & 	     1.9906 	 & 	     0.1990 	 \\ 
\hline  
\end{tabular} 

\end{center}
%
%%%%%%%%%%%%%%%%%%%%%%%%%%%%%%%%%%%%%%%%%%%%%%%%%%%%%%%%%
%
%%%%%%%%%%%%%%%%%%%%%%%%%%%%%%%%%%%%%%%%%%%%%%%%%%%%%%%%%
%
%%%%%%%%%%%%%%%%%%%%%%%%%%%%%%%%%%%%%%%%%%%%%%%%%%%%%%%%%
%%%%%%%%%%%%%%%%%%%%%%%%%%%%%%%%%%%%%%%%%%%%%%%%%%%%%%%%%
%

\begin{center}
{\bf Table 5.3}. Example 1: uniform refinement with $\omega=4.0$
\begin{tabular}{|c||c|c||c|c||c|c||c|c||c|} 
\hline  
$N$ 	 & 	 $\e(u)$ 	 & 	 $r(u)$ 	 & 	 $\e(\bm{\sigma})$ 	 & 	 $r(\bm{\sigma})$ 	 & 	 $\e$ 	 & 	 $r$ 	 & 	 $\e_0(u)$ 	 & 	 $r_0$ 	 & 	 $\e/\eta$ 	 \\
\hline  
    36 	 & 	 2.731e+01 	 & 	 ----- 	 & 	 1.797e+01 	 & 	 ----- 	 & 	 3.269e+01 	 & 	 ----- 	 & 	 4.623e+00 	 & 	 ----- 	 & 	     0.3164 	 \\ 
   144 	 & 	 9.569e-01 	 & 	     4.8348 	 & 	 6.021e-01 	 & 	     4.8993 	 & 	 1.131e+00 	 & 	     4.8537 	 & 	 1.085e-01 	 & 	     5.4135 	 & 	     0.3164 	 \\ 
   576 	 & 	 5.073e-01 	 & 	     0.9155 	 & 	 3.236e-01 	 & 	     0.8958 	 & 	 6.017e-01 	 & 	     0.9099 	 & 	 4.569e-02 	 & 	     1.2475 	 & 	     0.2304 	 \\ 
  2304 	 & 	 2.496e-01 	 & 	     1.0230 	 & 	 1.668e-01 	 & 	     0.9566 	 & 	 3.002e-01 	 & 	     1.0032 	 & 	 1.458e-02 	 & 	     1.6484 	 & 	     0.2205 	 \\ 
  9216 	 & 	 1.231e-01 	 & 	     1.0203 	 & 	 8.438e-02 	 & 	     0.9828 	 & 	 1.492e-01 	 & 	     1.0085 	 & 	 3.975e-03 	 & 	     1.8743 	 & 	     0.2069 	 \\ 
 36864 	 & 	 6.129e-02 	 & 	     1.0057 	 & 	 4.252e-02 	 & 	     0.9887 	 & 	 7.460e-02 	 & 	     1.0002 	 & 	 1.025e-03 	 & 	     1.9554 	 & 	     0.2012 	 \\ 
147456 	 & 	 3.063e-02 	 & 	     1.0009 	 & 	 2.136e-02 	 & 	     0.9931 	 & 	 3.734e-02 	 & 	     0.9983 	 & 	 2.593e-04 	 & 	     1.9829 	 & 	     0.1998 	 \\ 
\hline  
\end{tabular} 

\end{center}
%
%%%%%%%%%%%%%%%%%%%%%%%%%%%%%%%%%%%%%%%%%%%%%%%%%%%%%%%%%
%
%%%%%%%%%%%%%%%%%%%%%%%%%%%%%%%%%%%%%%%%%%%%%%%%%%%%%%%%%
%
\begin{center}
{\bf Table 5.4}. Example 1: uniform refinement with $\omega=4.44$
\begin{tabular}{|c||c|c||c|c||c|c||c|c||c|} 
\hline  
$N$ 	 & 	 $\e(u)$ 	 & 	 $r(u)$ 	 & 	 $\e(\bm{\sigma})$ 	 & 	 $r(\bm{\sigma})$ 	 & 	 $\e$ 	 & 	 $r$ 	 & 	 $\e_0(u)$ 	 & 	 $r_0$ 	 & 	 $\e/\eta$ 	 \\
\hline  
    36 	 & 	 2.250e+00 	 & 	 ----- 	 & 	 2.244e+00 	 & 	 ----- 	 & 	 3.177e+00 	 & 	 ----- 	 & 	 5.023e-01 	 & 	 ----- 	 & 	    43.1922 	 \\ 
   144 	 & 	 2.164e+00 	 & 	     0.0562 	 & 	 2.161e+00 	 & 	     0.0545 	 & 	 3.058e+00 	 & 	     0.0553 	 & 	 4.863e-01 	 & 	     0.0469 	 & 	    43.1922 	 \\ 
   576 	 & 	 2.077e+00 	 & 	     0.0589 	 & 	 2.074e+00 	 & 	     0.0591 	 & 	 2.935e+00 	 & 	     0.0590 	 & 	 4.668e-01 	 & 	     0.0590 	 & 	    18.0626 	 \\ 
  2304 	 & 	 1.803e+00 	 & 	     0.2045 	 & 	 1.800e+00 	 & 	     0.2046 	 & 	 2.547e+00 	 & 	     0.2046 	 & 	 4.050e-01 	 & 	     0.2048 	 & 	    14.7562 	 \\ 
  9216 	 & 	 1.188e+00 	 & 	     0.6017 	 & 	 1.185e+00 	 & 	     0.6023 	 & 	 1.678e+00 	 & 	     0.6020 	 & 	 2.667e-01 	 & 	     0.6029 	 & 	     8.9778 	 \\ 
 36864 	 & 	 5.068e-01 	 & 	     1.2290 	 & 	 5.049e-01 	 & 	     1.2313 	 & 	 7.154e-01 	 & 	     1.2301 	 & 	 1.134e-01 	 & 	     1.2336 	 & 	     4.8120 	 \\ 
147456 	 & 	 1.562e-01 	 & 	     1.6983 	 & 	 1.546e-01 	 & 	     1.7075 	 & 	 2.197e-01 	 & 	     1.7029 	 & 	 3.451e-02 	 & 	     1.7164 	 & 	     2.4737 	 \\ 
\hline  
\end{tabular} 

\end{center}
%
%%%%%%%%%%%%%%%%%%%%%%%%%%%%%%%%%%%%%%%%%%%%%%%%%%%%%%%%%
%
%%%%%%%%%%%%%%%%%%%%%%%%%%%%%%%%%%%%%%%%%%%%%%%%%%%%%%%%%
%
\begin{center}
{\bf Table 5.5}. Example 1: uniform refinement with $\omega=4.5$
\begin{tabular}{|c||c|c||c|c||c|c||c|c||c|} 
\hline  
$N$ 	 & 	 $\e(u)$ 	 & 	 $r(u)$ 	 & 	 $\e(\bm{\sigma})$ 	 & 	 $r(\bm{\sigma})$ 	 & 	 $\e$ 	 & 	 $r$ 	 & 	 $\e_0(u)$ 	 & 	 $r_0$ 	 & 	 $\e/\eta$ 	 \\
\hline  
    36 	 & 	 1.944e+00 	 & 	 ----- 	 & 	 2.046e+00 	 & 	 ----- 	 & 	 2.822e+00 	 & 	 ----- 	 & 	 4.429e-01 	 & 	 ----- 	 & 	     2.2210 	 \\ 
   144 	 & 	 5.152e+00 	 & 	    ----- 	 & 	 5.175e+00 	 & 	    ----- 	 & 	 7.302e+00 	 & 	    ----- 	 & 	 1.159e+00 	 & 	    ----- 	 & 	     2.2210 	 \\ 
   576 	 & 	 5.368e+00 	 & 	    ----- 	 & 	 5.356e+00 	 & 	    ----- 	 & 	 7.583e+00 	 & 	    ----- 	 & 	 1.187e+00 	 & 	    ----- 	 & 	     0.9030 	 \\ 
  2304 	 & 	 6.627e-01 	 & 	     3.0180 	 & 	 6.366e-01 	 & 	     3.0727 	 & 	 9.190e-01 	 & 	     3.0447 	 & 	 1.359e-01 	 & 	     3.1262 	 & 	     0.7514 	 \\ 
  9216 	 & 	 1.843e-01 	 & 	     1.8462 	 & 	 1.613e-01 	 & 	     1.9805 	 & 	 2.449e-01 	 & 	     1.9075 	 & 	 3.040e-02 	 & 	     2.1607 	 & 	     0.4839 	 \\ 
 36864 	 & 	 6.994e-02 	 & 	     1.3981 	 & 	 5.430e-02 	 & 	     1.5709 	 & 	 8.854e-02 	 & 	     1.4680 	 & 	 7.466e-03 	 & 	     2.0257 	 & 	     0.3090 	 \\ 
147456 	 & 	 3.176e-02 	 & 	     1.1387 	 & 	 2.297e-02 	 & 	     1.2412 	 & 	 3.920e-02 	 & 	     1.1756 	 & 	 1.866e-03 	 & 	     2.0004 	 & 	     0.2333 	 \\ 
\hline  
\end{tabular} 

\end{center}
%
%%%%%%%%%%%%%%%%%%%%%%%%%%%%%%%%%%%%%%%%%%%%%%%%%%%%%%%%
%
%%%%%%%%%%%%%%%%%%%%%%%%%%%%%%%%%%%%%%%%%%%%%%%%%%%%%%%%%
%
%%%%%%%%%%%%%%%%%%%%%%%%%%%%%%%%%%%%%%%%%%%%%%%%%%%%%%%%%
%
\begin{center}
{\bf Table 5.6}. Example 1: uniform refinement with $\omega=5.0$
\begin{tabular}{|c||c|c||c|c||c|c||c|c||c|} 
\hline  
$N$ 	 & 	 $\e(u)$ 	 & 	 $r(u)$ 	 & 	 $\e(\bm{\sigma})$ 	 & 	 $r(\bm{\sigma})$ 	 & 	 $\e$ 	 & 	 $r$ 	 & 	 $\e_0(u)$ 	 & 	 $r_0$ 	 & 	 $\e/\eta$ 	 \\
\hline  
    36 	 & 	 1.281e+00 	 & 	 ----- 	 & 	 1.557e+00 	 & 	 ----- 	 & 	 2.016e+00 	 & 	 ----- 	 & 	 2.695e-01 	 & 	 ----- 	 & 	     0.3241 	 \\ 
   144 	 & 	 1.243e+00 	 & 	     0.0432 	 & 	 7.348e-01 	 & 	     1.0837 	 & 	 1.444e+00 	 & 	     0.4818 	 & 	 1.152e-01 	 & 	     1.2261 	 & 	     0.3241 	 \\ 
   576 	 & 	 5.392e-01 	 & 	     1.2048 	 & 	 3.438e-01 	 & 	     1.0958 	 & 	 6.395e-01 	 & 	     1.1749 	 & 	 3.902e-02 	 & 	     1.5618 	 & 	     0.1962 	 \\ 
  2304 	 & 	 2.519e-01 	 & 	     1.0980 	 & 	 1.690e-01 	 & 	     1.0248 	 & 	 3.033e-01 	 & 	     1.0761 	 & 	 1.116e-02 	 & 	     1.8054 	 & 	     0.1992 	 \\ 
  9216 	 & 	 1.232e-01 	 & 	     1.0314 	 & 	 8.463e-02 	 & 	     0.9975 	 & 	 1.495e-01 	 & 	     1.0207 	 & 	 2.937e-03 	 & 	     1.9266 	 & 	     0.1989 	 \\ 
 36864 	 & 	 6.131e-02 	 & 	     1.0073 	 & 	 4.255e-02 	 & 	     0.9919 	 & 	 7.463e-02 	 & 	     1.0024 	 & 	 7.496e-04 	 & 	     1.9702 	 & 	     0.1989 	 \\ 
147456 	 & 	 3.063e-02 	 & 	     1.0012 	 & 	 2.137e-02 	 & 	     0.9939 	 & 	 3.735e-02 	 & 	     0.9988 	 & 	 1.891e-04 	 & 	     1.9869 	 & 	     0.1992 	 \\ 
\hline  
\end{tabular} 

\end{center}
%
%%%%%%%%%%%%%%%%%%%%%%%%%%%%%%%%%%%%%%%%%%%%%%%%%%%%%%%%
%
%%%%%%%%%%%%%%%%%%%%%%%%%%%%%%%%%%%%%%%%%%%%%%%%%%%%%%%%%
%\newpage
%%%%%%%%%%%%%%%%%%%%%%%%%%%%%%%%%%%%%%%%%%%%%%%%%%%%%%%%%
%
\begin{center}
{\bf Table 5.7}. Example 1: uniform refinement with $\omega=10.0$
\begin{tabular}{|c||c|c||c|c||c|c||c|c||c|} 
\hline  
$N$ 	 & 	 $\e(u)$ 	 & 	 $r(u)$ 	 & 	 $\e(\bm{\sigma})$ 	 & 	 $r(\bm{\sigma})$ 	 & 	 $\e$ 	 & 	 $r$ 	 & 	 $\e_0(u)$ 	 & 	 $r_0$ 	 & 	 $\e/\eta$ 	 \\
\hline  
    36 	 & 	 1.271e+00 	 & 	 ----- 	 & 	 1.723e+00 	 & 	 ----- 	 & 	 2.142e+00 	 & 	 ----- 	 & 	 1.658e-01 	 & 	 ----- 	 & 	     0.1011 	 \\ 
   144 	 & 	 2.157e+00 	 & 	    ----- 	 & 	 1.060e+00 	 & 	     0.7011 	 & 	 2.403e+00 	 & 	    ----- 	 & 	 1.316e-01 	 & 	     0.3341 	 & 	     0.1011 	 \\ 
   576 	 & 	 1.087e+00 	 & 	     0.9887 	 & 	 9.812e-01 	 & 	     0.1115 	 & 	 1.464e+00 	 & 	     0.7148 	 & 	 9.640e-02 	 & 	     0.4486 	 & 	     0.2300 	 \\ 
  2304 	 & 	 1.635e+00 	 & 	    ----- 	 & 	 1.624e+00 	 & 	    ----- 	 & 	 2.305e+00 	 & 	    ----- 	 & 	 1.623e-01 	 & 	    ----- 	 & 	     0.3659 	 \\ 
  9216 	 & 	 3.455e-01 	 & 	     2.2430 	 & 	 3.343e-01 	 & 	     2.2804 	 & 	 4.808e-01 	 & 	     2.2613 	 & 	 3.227e-02 	 & 	     2.3304 	 & 	     0.8365 	 \\ 
 36864 	 & 	 8.072e-02 	 & 	     2.0977 	 & 	 6.767e-02 	 & 	     2.3047 	 & 	 1.053e-01 	 & 	     2.1904 	 & 	 5.248e-03 	 & 	     2.6206 	 & 	     0.5980 	 \\ 
147456 	 & 	 3.289e-02 	 & 	     1.2952 	 & 	 2.451e-02 	 & 	     1.4649 	 & 	 4.102e-02 	 & 	     1.3605 	 & 	 1.198e-03 	 & 	     2.1305 	 & 	     0.2800 	 \\ 
\hline  
\end{tabular} 

\end{center}
%
%%%%%%%%%%%%%%%%%%%%%%%%%%%%%%%%%%%%%%%%%%%%%%%%%%%%%%%%
%
%%%%%%%%%%%%%%%%%%%%%%%%%%%%%%%%%%%%%%%%%%%%%%%%%%%%%%%%%
%
%%%%%%%%%%%%%%%%%%%%%%%%%%%%%%%%%%%%%%%%%%%%%%%%%%%%%%%%%
%
\begin{center}
{\bf Table 5.8}. Example 1: uniform refinement with $\omega=15.0$
\begin{tabular}{|c||c|c||c|c||c|c||c|c||c|} 
\hline  
$N$ 	 & 	 $\e(u)$ 	 & 	 $r(u)$ 	 & 	 $\e(\bm{\sigma})$ 	 & 	 $r(\bm{\sigma})$ 	 & 	 $\e$ 	 & 	 $r$ 	 & 	$ \e_0(u)$ 	 & 	 $r_0$ 	 & 	 $\e/\eta$ 	 \\
\hline  
    36 	 & 	 1.088e+00 	 & 	 ----- 	 & 	 1.348e+00 	 & 	 ----- 	 & 	 1.733e+00 	 & 	 ----- 	 & 	 8.429e-02 	 & 	 ----- 	 & 	     0.0802 	 \\ 
   144 	 & 	 1.297e+00 	 & 	    ----- 	 & 	 1.393e+00 	 & 	    ----- 	 & 	 1.903e+00 	 & 	    ----- 	 & 	 1.032e-01 	 & 	    ----- 	 & 	     0.0802 	 \\ 
   576 	 & 	 5.437e-01 	 & 	     1.2539 	 & 	 3.563e-01 	 & 	     1.9672 	 & 	 6.500e-01 	 & 	     1.5497 	 & 	 1.906e-02 	 & 	     2.4368 	 & 	     0.1388 	 \\ 
  2304 	 & 	 2.469e-01 	 & 	     1.1390 	 & 	 1.644e-01 	 & 	     1.1153 	 & 	 2.967e-01 	 & 	     1.1318 	 & 	 3.875e-03 	 & 	     2.2987 	 & 	     0.2014 	 \\ 
  9216 	 & 	 1.227e-01 	 & 	     1.0093 	 & 	 8.409e-02 	 & 	     0.9677 	 & 	 1.487e-01 	 & 	     0.9963 	 & 	 1.036e-03 	 & 	     1.9036 	 & 	     0.1980 	 \\ 
 36864 	 & 	 6.124e-02 	 & 	     1.0021 	 & 	 4.249e-02 	 & 	     0.9846 	 & 	 7.454e-02 	 & 	     0.9965 	 & 	 2.711e-04 	 & 	     1.9338 	 & 	     0.1990 	 \\ 
147456 	 & 	 3.062e-02 	 & 	     1.0000 	 & 	 2.136e-02 	 & 	     0.9923 	 & 	 3.733e-02 	 & 	     0.9975 	 & 	 6.901e-05 	 & 	     1.9739 	 & 	     0.1993 	 \\ 
\hline  
\end{tabular} 

\end{center}
%
%%%%%%%%%%%%%%%%%%%%%%%%%%%%%%%%%%%%%%%%%%%%%%%%%%%%%%%%
%

%%%%%%%%%%%%%%%%%%%%%%%%%%%%%%%%%%%%%%%%%%%%%%%%%%%%%%%%
\begin{center}

{\bf Table 5.9}. Example 2: uniform refinement with $\omega=1$

\begin{tabular}{|c||c|c||c|c||c|c||c|c||c|} 
\hline  
$N$ 	 & 	 $\e(u)$ 	 & 	 $r(u)$ 	 & 	 $\e(\bm{\sigma})$ 	 & 	 $r(\bm{\sigma})$ 	 & 	 $\e$ 	 & 	 $r$ 	 & 	 $\e_0(u)$ 	 & 	 $r_0$ 	 & 	 $\e/\eta$ 	 \\\hline  
    54 	 & 	 3.295e+01 	 & 	 ----- 	 & 	 2.298e+01 	 & 	 ----- 	 & 	 4.017e+01 	 & 	 ----- 	 & 	 5.900e+00 	 & 	 ----- 	 & 	     0.1517 	 \\ 
   216 	 & 	 2.394e+01 	 & 	     0.4607 	 & 	 1.498e+01 	 & 	     0.6174 	 & 	 2.824e+01 	 & 	     0.5084 	 & 	 2.044e+00 	 & 	     1.5296 	 & 	     0.1517 	 \\ 
   864 	 & 	 1.514e+01 	 & 	     0.6607 	 & 	 7.545e+00 	 & 	     0.9893 	 & 	 1.692e+01 	 & 	     0.7391 	 & 	 5.490e-01 	 & 	     1.8961 	 & 	     0.1508 	 \\ 
  3456 	 & 	 8.927e+00 	 & 	     0.7624 	 & 	 2.996e+00 	 & 	     1.3324 	 & 	 9.417e+00 	 & 	     0.8453 	 & 	 1.417e-01 	 & 	     1.9540 	 & 	     0.1536 	 \\ 
 13824 	 & 	 4.908e+00 	 & 	     0.8631 	 & 	 9.685e-01 	 & 	     1.6294 	 & 	 5.003e+00 	 & 	     0.9125 	 & 	 3.821e-02 	 & 	     1.8908 	 & 	     0.1622 	 \\ 
 55296 	 & 	 2.555e+00 	 & 	     0.9421 	 & 	 3.991e-01 	 & 	     1.2791 	 & 	 2.586e+00 	 & 	     0.9522 	 & 	 1.025e-02 	 & 	     1.8980 	 & 	     0.1692 	 \\ 
221184 	 & 	 1.295e+00 	 & 	     0.9800 	 & 	 2.632e-01 	 & 	     0.6005 	 & 	 1.322e+00 	 & 	     0.9682 	 & 	 2.706e-03 	 & 	     1.9219 	 & 	     0.1724 	 \\ 
\hline  
\end{tabular} 

\end{center}

%%%%%%%%%%%%%%%%%%%%%%%%%%%%%%%%%%%%%%%%%%%%%%%%%%%%%%%%
\begin{center}

{\bf Table 5.10}. Example 2: uniform refinement with $\omega=10$

\begin{tabular}{|c||c|c||c|c||c|c||c|c||c|} 
\hline  
$N$ 	 & 	 $\e(u)$ 	 & 	 $r(u)$ 	 & 	 $\e(\bm{\sigma})$ 	 & 	 $r(\bm{\sigma})$ 	 & 	 $\e$ 	 & 	 $r$ 	 & 	 $\e_0(u)$ 	 & 	 $r_0$ 	 & 	 $\e/\eta$ 	 \\\hline  
    54 	 & 	 1.823e+02 	 & 	 ----- 	 & 	 2.847e+02 	 & 	 ----- 	 & 	 3.381e+02 	 & 	 ----- 	 & 	 3.045e+01 	 & 	 ----- 	 & 	     0.0750 	 \\ 
   216 	 & 	 2.731e+01 	 & 	     2.7389 	 & 	 3.183e+01 	 & 	     3.1611 	 & 	 4.194e+01 	 & 	     3.0110 	 & 	 4.045e+00 	 & 	     2.9122 	 & 	     0.0750 	 \\ 
   864 	 & 	 2.440e+01 	 & 	     0.1625 	 & 	 1.787e+01 	 & 	     0.8325 	 & 	 3.025e+01 	 & 	     0.4715 	 & 	 1.627e+00 	 & 	     1.3140 	 & 	     0.1806 	 \\ 
  3456 	 & 	 2.426e+01 	 & 	     0.0084 	 & 	 2.212e+01 	 & 	    ----- 	 & 	 3.283e+01 	 & 	    ----- 	 & 	 2.200e+00 	 & 	    ----- 	 & 	     0.1961 	 \\ 
 13824 	 & 	 4.926e+00 	 & 	     2.3001 	 & 	 1.027e+00 	 & 	     4.4286 	 & 	 5.032e+00 	 & 	     2.7058 	 & 	 4.858e-02 	 & 	     5.5011 	 & 	     0.4282 	 \\ 
 55296 	 & 	 2.559e+00 	 & 	     0.9448 	 & 	 4.165e-01 	 & 	     1.3019 	 & 	 2.593e+00 	 & 	     0.9566 	 & 	 1.548e-02 	 & 	     1.6499 	 & 	     0.1701 	 \\ 
221184 	 & 	 1.296e+00 	 & 	     0.9814 	 & 	 2.664e-01 	 & 	     0.6447 	 & 	 1.323e+00 	 & 	     0.9704 	 & 	 4.894e-03 	 & 	     1.6611 	 & 	     0.1730 	 \\ 
\hline  
\end{tabular} 

\end{center}

%
%%%%%%%%%%%%%%%%%%%%%%%%%%%%%%%%%%%%%%%%%%%%%%%%%%%%%%%%%
%
\begin{center}
{\bf Table 5.11}. Example 2: uniform  refinement with $\omega=15.0$
\begin{tabular}{|c||c|c||c|c||c|c||c|c||c|} 
\hline  
$N$ 	 & 	   $\e(u)$ 	 & 	     $r(u)$ 	 & $\e(\bm{\sigma})$ 	 & $r(\bm{\sigma})$ 	 & 	      $\e$ 	 & 	 $r$ 	 & 	 $\e_0(u)$ 	 & 	 $r_0$ 	 & 	 $\e/\eta$ 	 \\\hline  
    54 	 & 	 1.384e+01 	 & 	      ----- 	 & 	 2.487e+01 	 & 	      ----- 	 & 	 2.847e+01 	 & 	 ----- 	 & 	 1.864e+00 	 & 	 ----- 	 & 	     0.0429 	 \\ 
   216 	 & 	 1.190e+02 	 & 	      ----- 	 & 	 1.661e+02 	 & 	      ----- 	 & 	 2.043e+02 	 & 	    ----- 	 & 	 1.168e+01 	 & 	    ----- 	 & 	     0.0429 	 \\ 
   864 	 & 	 3.577e+01 	 & 	     1.7334 	 & 	 3.391e+01 	 & 	     2.2927 	 & 	 4.929e+01 	 & 	     2.0515 	 & 	 2.264e+00 	 & 	     2.3669 	 & 	     0.0989 	 \\ 
  3456 	 & 	 1.055e+01 	 & 	     1.7613 	 & 	 6.061e+00 	 & 	     2.4840 	 & 	 1.217e+01 	 & 	     2.0181 	 & 	 3.680e-01 	 & 	     2.6212 	 & 	     0.1595 	 \\ 
 13824 	 & 	 4.966e+00 	 & 	     1.0873 	 & 	 1.174e+00 	 & 	     2.3679 	 & 	 5.103e+00 	 & 	     1.2537 	 & 	 5.623e-02 	 & 	     2.7103 	 & 	     0.1881 	 \\ 
 55296 	 & 	 2.564e+00 	 & 	     0.9536 	 & 	 4.419e-01 	 & 	     1.4098 	 & 	 2.602e+00 	 & 	     0.9718 	 & 	 1.604e-02 	 & 	     1.8092 	 & 	     0.1720 	 \\ 
221184 	 & 	 1.296e+00 	 & 	     0.9841 	 & 	 2.679e-01 	 & 	     0.7220 	 & 	 1.324e+00 	 & 	     0.9750 	 & 	 4.245e-03 	 & 	     1.9181 	 & 	     0.1737 	 \\ 
\hline  
\end{tabular} 

\end{center}

%\newpage
%%%%%%%%%%%%%%%%%%%%%%%%%%%%%%%%%%%%%%%%%%%%%%%%%%%%%%%%
\begin{center}

{\bf Table 5.12}. Example 3: uniform refinement with $\omega=1$

\begin{tabular}{|c||c|c||c|c||c|c||c|c||c|} 
\hline  
$N$ 	 & 	 $\e(u)$ 	 & 	 $r(u)$ 	 & 	 $\e(\bm{\sigma})$ 	 & 	 $r(\bm{\sigma})$ 	 & 	 $\e$ 	 & 	 $r$ 	 & 	 $\e_0(u)$ 	 & 	 $r_0$ 	 & 	 $\e/\eta$ 	 \\\hline  
    54 	 & 	 2.918e-01 	 & 	 ----- 	 & 	 2.214e-01 	 & 	 ----- 	 & 	 3.662e-01 	 & 	 ----- 	 & 	 3.455e-02 	 & 	 ----- 	 & 	     0.1911 	 \\ 
   216 	 & 	 2.159e-01 	 & 	     0.4346 	 & 	 1.337e-01 	 & 	     0.7274 	 & 	 2.539e-01 	 & 	     0.5284 	 & 	 1.442e-02 	 & 	     1.2603 	 & 	     0.1911 	 \\ 
   864 	 & 	 1.467e-01 	 & 	     0.5574 	 & 	 8.440e-02 	 & 	     0.6637 	 & 	 1.692e-01 	 & 	     0.5853 	 & 	 5.906e-03 	 & 	     1.2881 	 & 	     0.1874 	 \\ 
  3456 	 & 	 9.596e-02 	 & 	     0.6123 	 & 	 5.398e-02 	 & 	     0.6450 	 & 	 1.101e-01 	 & 	     0.6203 	 & 	 2.111e-03 	 & 	     1.4841 	 & 	     0.1859 	 \\ 
 13824 	 & 	 6.180e-02 	 & 	     0.6347 	 & 	 3.441e-02 	 & 	     0.6497 	 & 	 7.073e-02 	 & 	     0.6383 	 & 	 7.442e-04 	 & 	     1.5042 	 & 	     0.1794 	 \\ 
 55296 	 & 	 3.946e-02 	 & 	     0.6473 	 & 	 2.185e-02 	 & 	     0.6551 	 & 	 4.510e-02 	 & 	     0.6491 	 & 	 2.657e-04 	 & 	     1.4860 	 & 	     0.1702 	 \\ 
221184 	 & 	 2.506e-02 	 & 	     0.6548 	 & 	 1.384e-02 	 & 	     0.6591 	 & 	 2.863e-02 	 & 	     0.6558 	 & 	 9.689e-05 	 & 	     1.4554 	 & 	     0.1595 	 \\ 
\hline  
\end{tabular} 

\end{center}

%%%%%%%%%%%%%%%%%%%%%%%%%%%%%%%%%%%%%%%%%%%%%%%%%%%%%%%%%
%%%%%%%%%%%%%%%%%%%%%%%%%%%%%%%%%%%%%%%%%%%%%%%%%%%%%%%%%
%
\begin{center}
{\bf Table 5.13}. Example 3: uniform refinement with $\omega=10.0$
\begin{tabular}{|c||c|c||c|c||c|c||c|c||c|} 
\hline  
$N$ 	 & 	 $\e(u)$ 	 & 	 $r(u)$ 	 & 	 $\e(\bm{\sigma})$ 	 & 	 $r(\bm{\sigma})$ 	 & 	 $\e$ 	 & 	 $r$ 	 & 	 $\e_0(u)$ 	 & 	 $r_0$ 	 & 	 $\e/\eta$ 	 \\\hline  
    54 	 & 	 5.118e-01 	 & 	 ----- 	 & 	 6.493e-01 	 & 	 ----- 	 & 	 8.268e-01 	 & 	 ----- 	 & 	 7.391e-02 	 & 	 ----- 	 & 	     0.0747 	 \\ 
   216 	 & 	 3.680e-01 	 & 	     0.4760 	 & 	 3.521e-01 	 & 	     0.8830 	 & 	 5.093e-01 	 & 	     0.6990 	 & 	 3.794e-02 	 & 	     0.9621 	 & 	     0.0747 	 \\ 
   864 	 & 	 1.614e-01 	 & 	     1.1891 	 & 	 1.042e-01 	 & 	     1.7561 	 & 	 1.921e-01 	 & 	     1.4065 	 & 	 7.287e-03 	 & 	     2.3803 	 & 	     0.1592 	 \\ 
  3456 	 & 	 9.701e-02 	 & 	     0.7342 	 & 	 5.493e-02 	 & 	     0.9243 	 & 	 1.115e-01 	 & 	     0.7851 	 & 	 1.728e-03 	 & 	     2.0761 	 & 	     0.1997 	 \\ 
 13824 	 & 	 6.192e-02 	 & 	     0.6477 	 & 	 3.446e-02 	 & 	     0.6725 	 & 	 7.087e-02 	 & 	     0.6536 	 & 	 4.796e-04 	 & 	     1.8494 	 & 	     0.1818 	 \\ 
 55296 	 & 	 3.948e-02 	 & 	     0.6494 	 & 	 2.186e-02 	 & 	     0.6568 	 & 	 4.512e-02 	 & 	     0.6512 	 & 	 1.547e-04 	 & 	     1.6318 	 & 	     0.1705 	 \\ 
221184 	 & 	 2.507e-02 	 & 	     0.6553 	 & 	 1.384e-02 	 & 	     0.6594 	 & 	 2.863e-02 	 & 	     0.6562 	 & 	 5.335e-05 	 & 	     1.5362 	 & 	     0.1596 	 \\ 
\hline  
\end{tabular} 

\end{center}

%%%%%%%%%%%%%%%%%%%%%%%%%%%%%%%%%%%%%%%%%%%%%%%%%%%%%%%%
%\newpage
\begin{center}

{\bf Table 5.14}. Example 3: uniform refinement with $\omega=15$

\begin{tabular}{|c||c|c||c|c||c|c||c|c||c|} 
\hline  
$N$ 	 & 	 $\e(u)$ 	 & 	 $r(u)$ 	 & 	 $\e(\bm{\sigma})$ 	 & 	 $r(\bm{\sigma})$ 	 & 	 $\e$ 	 & 	 $r$ 	 & 	 $\e_0(u)$ 	 & 	 $r_0$ 	 & 	 $\e/\eta$ 	 \\\hline  
    54 	 & 	 3.665e-01 	 & 	      ----- 	 & 	 7.394e-01 	 & 	 ----- 	 & 	 8.252e-01 	 & 	 ----- 	 & 	 5.125e-02 	 & 	 ----- 	 & 	     0.0479 	 \\ 
   216 	 & 	 5.089e+00 	 & 	      ----- 	 & 	 7.604e+00 	 & 	      ----- 	 & 	 9.150e+00 	 & 	    ----- 	 & 	 5.319e-01 	 & 	    ----- 	 & 	     0.0479 	 \\ 
   864 	 & 	 2.585e-01 	 & 	     4.2995 	 & 	 1.793e-01 	 & 	     5.4060 	 & 	 3.146e-01 	 & 	     4.8623 	 & 	 1.364e-02 	 & 	     5.2850 	 & 	     0.0992 	 \\ 
  3456 	 & 	 1.564e-01 	 & 	     0.7244 	 & 	 1.270e-01 	 & 	     0.4985 	 & 	 2.015e-01 	 & 	     0.6429 	 & 	 7.878e-03 	 & 	     0.7923 	 & 	     0.2020 	 \\ 
 13824 	 & 	 6.458e-02 	 & 	     1.2765 	 & 	 3.874e-02 	 & 	     1.7122 	 & 	 7.531e-02 	 & 	     1.4197 	 & 	 1.281e-03 	 & 	     2.6203 	 & 	     0.2791 	 \\ 
 55296 	 & 	 3.956e-02 	 & 	     0.7070 	 & 	 2.200e-02 	 & 	     0.8166 	 & 	 4.526e-02 	 & 	     0.7344 	 & 	 2.157e-04 	 & 	     2.5704 	 & 	     0.1810 	 \\ 
221184 	 & 	 2.507e-02 	 & 	     0.6578 	 & 	 1.385e-02 	 & 	     0.6679 	 & 	 2.864e-02 	 & 	     0.6602 	 & 	 5.689e-05 	 & 	     1.9229 	 & 	     0.1600 	 \\ 
\hline  
\end{tabular} 

\end{center}

%%%%%%%%%%%%%%%%%%%%%%%%%%%%%%%%%%%%%%%%%%%%%%%%%%%%%%%%

%%%%%%%%%%%%%%%%%%%%%%%%%%%%%%%%%%%%%%%%%%%%%%%%%%%%%%%%

%%%%%%%%%%%%%%%%%%%%%%%%%%%%%%%%%%%%%%%%%%%%%%%%%%%%%%%%
\newpage

\begin{center}
{\bf Table 5.15}. Example 2: hybrid adaptive refinement with $\omega=1.0$
\begin{tabular}{|c||c|c||c|c||c|c||c|} 
\hline  
$N$ 	 & 	 $\e(u)$ 	 & 	 $r(u)$ 	 & 	 $\e(\bm{\sigma})$ 	 & 	 $r(\bm{\sigma})$ 	 & 	 $\e$ 	 & 	 $r$ 	 &  	 $\e/\eta$ 	 \\
\hline  
    54 	 & 	 3.295e+01 	 & 	      ----- 	 & 	 2.298e+01 	 & 	      ----- 	 & 	 4.017e+01 	 & 	      ----- 	 & 	     0.1517 	 \\ 
   216 	 & 	 2.394e+01 	 & 	     0.4607 	 & 	 1.498e+01 	 & 	     0.6174 	 & 	 2.824e+01 	 & 	     0.5084 	 & 	     0.1517 	 \\ 
   252 	 & 	 2.148e+01 	 & 	     1.4097 	 & 	 9.065e+00 	 & 	     6.5157 	 & 	 2.331e+01 	 & 	     2.4891 	 & 	     0.1508 	 \\ 
   270 	 & 	 1.812e+01 	 & 	     4.9215 	 & 	 7.802e+00 	 & 	     4.3518 	 & 	 1.973e+01 	 & 	     4.8339 	 & 	     0.1495 	 \\ 
   378 	 & 	 1.233e+01 	 & 	     2.2903 	 & 	 4.207e+00 	 & 	     3.6715 	 & 	 1.303e+01 	 & 	     2.4682 	 & 	     0.1719 	 \\ 
   414 	 & 	 1.119e+01 	 & 	     2.1266 	 & 	 4.183e+00 	 & 	     0.1236 	 & 	 1.195e+01 	 & 	     1.8998 	 & 	     0.1566 	 \\ 
   630 	 & 	 6.986e+00 	 & 	     2.2446 	 & 	 2.113e+00 	 & 	     3.2540 	 & 	 7.298e+00 	 & 	     2.3476 	 & 	     0.1859 	 \\ 
   828 	 & 	 5.371e+00 	 & 	     1.9239 	 & 	 1.394e+00 	 & 	     3.0423 	 & 	 5.549e+00 	 & 	     2.0055 	 & 	     0.1679 	 \\ 
  1134 	 & 	 4.041e+00 	 & 	     1.8098 	 & 	 1.248e+00 	 & 	     0.7059 	 & 	 4.229e+00 	 & 	     1.7275 	 & 	     0.1735 	 \\ 
  2070 	 & 	 2.901e+00 	 & 	     1.1010 	 & 	 8.147e-01 	 & 	     1.4161 	 & 	 3.013e+00 	 & 	     1.1262 	 & 	     0.1709 	 \\ 
  3222 	 & 	 2.429e+00 	 & 	     0.8037 	 & 	 9.596e-01 	 & 	      ----- 	 & 	 2.611e+00 	 & 	     0.6474 	 & 	     0.1723 	 \\ 
  6732 	 & 	 1.611e+00 	 & 	     1.1141 	 & 	 5.239e-01 	 & 	     1.6426 	 & 	 1.694e+00 	 & 	     1.1746 	 & 	     0.1919 	 \\ 
 12456 	 & 	 1.197e+00 	 & 	     0.9653 	 & 	 4.504e-01 	 & 	     0.4914 	 & 	 1.279e+00 	 & 	     0.9135 	 & 	     0.1839 	 \\ 
 23364 	 & 	 8.526e-01 	 & 	     1.0789 	 & 	 3.089e-01 	 & 	     1.1988 	 & 	 9.069e-01 	 & 	     1.0933 	 & 	     0.1873 	 \\ 
 46521 	 & 	 6.143e-01 	 & 	     0.9518 	 & 	 2.382e-01 	 & 	     0.7550 	 & 	 6.589e-01 	 & 	     0.9275 	 & 	     0.1855 	 \\ 
 89235 	 & 	 4.229e-01 	 & 	     1.1467 	 & 	 1.447e-01 	 & 	     1.5297 	 & 	 4.470e-01 	 & 	     1.1917 	 & 	     0.1876 	 \\ 
173520 	 & 	 3.117e-01 	 & 	     0.9170 	 & 	 1.201e-01 	 & 	     0.5622 	 & 	 3.341e-01 	 & 	     0.8756 	 & 	     0.1799 	 \\ 
328185 	 & 	 2.212e-01 	 & 	     1.0765 	 & 	 7.958e-02 	 & 	     1.2905 	 & 	 2.351e-01 	 & 	     1.1026 	 & 	     0.1847 	 \\ 
\hline  
\end{tabular}
\end{center} 

%%%%%%%%%%%%%%%%%%%%%%%%%%%%%%%%%%%%%%%%%%%%%%%%%%%%%%%%

%%%%%%%%%%%%%%%%%%%%%%%%%%%%%%%%%%%%%%%%%%%%%%%%%%%%%%%%
\begin{center}
{\bf Table 5.16}. Example 2: hybrid adaptive refinement with $\omega=10.0$
\begin{tabular}{|c||c|c||c|c||c|c||c|} 
\hline  
$N$ 	 & 	 $\e(u)$ 	 & 	 $r(u)$ 	 & 	 $\e(\bm{\sigma})$ & 	 $r(\bm{\sigma})$ & 	      $\e$ 	 & 	        $r$ 	 &  	  $\e/\eta$ 	 \\\hline  
    54 	 & 	 1.823e+02 	 & 	      ----- 	 & 	 2.847e+02 	 & 	      ----- 	 & 	 3.381e+02 	 & 	      ----- 	 &  	     0.0750 	 \\ 
   216 	 & 	 2.731e+01 	 & 	     2.7389 	 & 	 3.183e+01 	 & 	     3.1611 	 & 	 4.194e+01 	 & 	     3.0110 	 &  	     0.0750 	 \\ 
   864 	 & 	 2.440e+01 	 & 	     0.1625 	 & 	 1.787e+01 	 & 	     0.8325 	 & 	 3.025e+01 	 & 	     0.4715 	 &  	     0.1806 	 \\ 
  3456 	 & 	 2.426e+01 	 & 	     0.0084 	 & 	 2.212e+01 	 & 	      ----- 	 & 	 3.283e+01 	 & 	       ---- 	 &  	     0.1961 	 \\ 
 13824 	 & 	 4.926e+00 	 & 	     2.3001 	 & 	 1.027e+00 	 & 	     4.4286 	 & 	 5.032e+00 	 & 	     2.7058 	 &  	     0.4282 	 \\ 
 14112 	 & 	 3.962e+00 	 & 	    21.1095 	 & 	 1.432e+00 	 & 	      ----- 	 & 	 4.213e+00 	 & 	    17.2184 	 &  	     0.1701 	 \\ 
 14544 	 & 	 2.880e+00 	 & 	    21.1550 	 & 	 9.110e-01 	 & 	    30.0031 	 & 	 3.021e+00 	 & 	    22.0647 	 &  	     0.1873 	 \\ 
 15696 	 & 	 2.321e+00 	 & 	     5.6609 	 & 	 9.138e-01 	 & 	      ----- 	 & 	 2.495e+00 	 & 	     5.0216 	 &  	     0.1765 	 \\ 
 18558 	 & 	 1.594e+00 	 & 	     4.4874 	 & 	 5.338e-01 	 & 	     6.4195 	 & 	 1.681e+00 	 & 	     4.7131 	 &  	     0.1935 	 \\ 
 23922 	 & 	 1.185e+00 	 & 	     2.3332 	 & 	 4.509e-01 	 & 	     1.3300 	 & 	 1.268e+00 	 & 	     2.2196 	 &  	     0.1856 	 \\ 
 34290 	 & 	 8.409e-01 	 & 	     1.9078 	 & 	 3.043e-01 	 & 	     2.1842 	 & 	 8.942e-01 	 & 	     1.9412 	 &  	     0.1891 	 \\ 
 56304 	 & 	 6.103e-01 	 & 	     1.2927 	 & 	 2.361e-01 	 & 	     1.0224 	 & 	 6.544e-01 	 & 	     1.2595 	 &  	     0.1858 	 \\ 
 97893 	 & 	 4.209e-01 	 & 	     1.3435 	 & 	 1.439e-01 	 & 	     1.7920 	 & 	 4.448e-01 	 & 	     1.3960 	 &  	     0.1878 	 \\ 
171693 	 & 	 3.176e-01 	 & 	     1.0021 	 & 	 1.217e-01 	 & 	     0.5945 	 & 	 3.401e-01 	 & 	     0.9548 	 &  	     0.1800 	 \\ 
324936 	 & 	 2.243e-01 	 & 	     1.0903 	 & 	 8.097e-02 	 & 	     1.2783 	 & 	 2.385e-01 	 & 	     1.1132 	 &  	     0.1844 	 \\ 
\hline  
\end{tabular}
\end{center} 

%%%%%%%%%%%%%%%%%%%%%%%%%%%%%%%%%%%%%%%%%%%%%%%%%%%%%%%%

%%%%%%%%%%%%%%%%%%%%%%%%%%%%%%%%%%%%%%%%%%%%%%%%%%%%%%%%
%%%%%%%%%%%%%%%%%%%%%%%%%%%%%%%%%%%%%%%%%%%%%%%%%%%%%%%%%
%%%%%%%%%%%%%%%%%%%%%%%%%%%%%%%%%%%%%%%%%%%%%%%%%%%%%%%%%
%
\begin{center}
{\bf Table 5.17}. Example 2: hybrid adaptive refinement with $\omega=15.0$
\begin{tabular}{|c||c|c||c|c||c|c||c|c||c|} 
\hline  
$N$ 	 & 	   $\e(u)$ 	 & 	     $r(u)$ 	 & $\e(\bm{\sigma})$ 	 & 	 $r(\bm{\sigma})$ & 	      $\e$ 	 & 	        $r$ 	 &  	 $\e/\eta$ 	 \\\hline  
    54 	 & 	 1.384e+01 	 & 	      ----- 	 & 	 2.487e+01 	 & 	      ----- 	 & 	 2.847e+01 	 & 	      ----- 	 &  	     0.0429 	 \\ 
   216 	 & 	 1.190e+02 	 & 	      ----- 	 & 	 1.661e+02 	 & 	      ----- 	 & 	 2.043e+02 	 & 	      ----- 	 &  	     0.0429 	 \\ 
   864 	 & 	 3.577e+01 	 & 	     1.7334 	 & 	 3.391e+01 	 & 	     2.2927 	 & 	 4.929e+01 	 & 	     2.0515 	 &  	     0.0989 	 \\ 
  3456 	 & 	 1.055e+01 	 & 	     1.7613 	 & 	 6.061e+00 	 & 	     2.4840 	 & 	 1.217e+01 	 & 	     2.0181 	 &  	     0.1595 	 \\ 
 13824 	 & 	 4.966e+00 	 & 	     1.0873 	 & 	 1.174e+00 	 & 	     2.3679 	 & 	 5.103e+00 	 & 	     1.2537 	 &  	     0.1881 	 \\ 
 55296 	 & 	 2.564e+00 	 & 	     0.9536 	 & 	 4.419e-01 	 & 	     1.4098 	 & 	 2.602e+00 	 & 	     0.9718 	 &  	     0.1720 	 \\ 
 55872 	 & 	 2.157e+00 	 & 	    33.3993 	 & 	 7.913e-01 	 & 	      ----- 	 & 	 2.297e+00 	 & 	    24.0356 	 &  	     0.1737 	 \\ 
 57888 	 & 	 1.553e+00 	 & 	    18.5436 	 & 	 5.102e-01 	 & 	    24.7643 	 & 	 1.634e+00 	 & 	    19.2139 	 &  	     0.1891 	 \\ 
 62514 	 & 	 1.162e+00 	 & 	     7.5468 	 & 	 4.379e-01 	 & 	     3.9739 	 & 	 1.241e+00 	 & 	     7.1521 	 &  	     0.1857 	 \\ 
 72018 	 & 	 8.300e-01 	 & 	     4.7507 	 & 	 3.008e-01 	 & 	     5.3071 	 & 	 8.828e-01 	 & 	     4.8176 	 &  	     0.1895 	 \\ 
 93294 	 & 	 6.016e-01 	 & 	     2.4863 	 & 	 2.337e-01 	 & 	     1.9524 	 & 	 6.454e-01 	 & 	     2.4204 	 &  	     0.1865 	 \\ 
133092 	 & 	 4.181e-01 	 & 	     2.0485 	 & 	 1.441e-01 	 & 	     2.7202 	 & 	 4.423e-01 	 & 	     2.1279 	 &  	     0.1880 	 \\ 
214128 	 & 	 3.087e-01 	 & 	     1.2754 	 & 	 1.188e-01 	 & 	     0.8123 	 & 	 3.308e-01 	 & 	     1.2211 	 &  	     0.1806 	 \\ 
364752 	 & 	 2.196e-01 	 & 	     1.2801 	 & 	 7.895e-02 	 & 	     1.5345 	 & 	 2.333e-01 	 & 	     1.3110	 &  	     0.1848 	 \\ 
\hline  
\end{tabular} 
\end{center}

%%%%%%%%%%%%%%%%%%%%%%%%%%%%%%%%%%%%%%%%%%%%%%%%%%%%%%%%
\newpage

%%%%%%%%%%%%%%%%%%%%%%%%%%%%%%%%%%%%%%%%%%%%%%%%%%%%%%%%%
%
\begin{center}
% GNUPLOT: LaTeX picture
\setlength{\unitlength}{0.240900pt}
\ifx\plotpoint\undefined\newsavebox{\plotpoint}\fi
\sbox{\plotpoint}{\rule[-0.200pt]{0.400pt}{0.400pt}}%
\begin{picture}(1500,900)(0,0)
\sbox{\plotpoint}{\rule[-0.200pt]{0.400pt}{0.400pt}}%
\put(201.0,123.0){\rule[-0.200pt]{4.818pt}{0.400pt}}
\put(181,123){\makebox(0,0)[r]{$0.1$}}
\put(1419.0,123.0){\rule[-0.200pt]{4.818pt}{0.400pt}}
\put(201.0,197.0){\rule[-0.200pt]{2.409pt}{0.400pt}}
\put(1429.0,197.0){\rule[-0.200pt]{2.409pt}{0.400pt}}
\put(201.0,240.0){\rule[-0.200pt]{2.409pt}{0.400pt}}
\put(1429.0,240.0){\rule[-0.200pt]{2.409pt}{0.400pt}}
\put(201.0,271.0){\rule[-0.200pt]{2.409pt}{0.400pt}}
\put(1429.0,271.0){\rule[-0.200pt]{2.409pt}{0.400pt}}
\put(201.0,295.0){\rule[-0.200pt]{2.409pt}{0.400pt}}
\put(1429.0,295.0){\rule[-0.200pt]{2.409pt}{0.400pt}}
\put(201.0,314.0){\rule[-0.200pt]{2.409pt}{0.400pt}}
\put(1429.0,314.0){\rule[-0.200pt]{2.409pt}{0.400pt}}
\put(201.0,331.0){\rule[-0.200pt]{2.409pt}{0.400pt}}
\put(1429.0,331.0){\rule[-0.200pt]{2.409pt}{0.400pt}}
\put(201.0,345.0){\rule[-0.200pt]{2.409pt}{0.400pt}}
\put(1429.0,345.0){\rule[-0.200pt]{2.409pt}{0.400pt}}
\put(201.0,357.0){\rule[-0.200pt]{2.409pt}{0.400pt}}
\put(1429.0,357.0){\rule[-0.200pt]{2.409pt}{0.400pt}}
\put(201.0,369.0){\rule[-0.200pt]{4.818pt}{0.400pt}}
\put(181,369){\makebox(0,0)[r]{$1$}}
\put(1419.0,369.0){\rule[-0.200pt]{4.818pt}{0.400pt}}
\put(201.0,443.0){\rule[-0.200pt]{2.409pt}{0.400pt}}
\put(1429.0,443.0){\rule[-0.200pt]{2.409pt}{0.400pt}}
\put(201.0,486.0){\rule[-0.200pt]{2.409pt}{0.400pt}}
\put(1429.0,486.0){\rule[-0.200pt]{2.409pt}{0.400pt}}
\put(201.0,517.0){\rule[-0.200pt]{2.409pt}{0.400pt}}
\put(1429.0,517.0){\rule[-0.200pt]{2.409pt}{0.400pt}}
\put(201.0,540.0){\rule[-0.200pt]{2.409pt}{0.400pt}}
\put(1429.0,540.0){\rule[-0.200pt]{2.409pt}{0.400pt}}
\put(201.0,560.0){\rule[-0.200pt]{2.409pt}{0.400pt}}
\put(1429.0,560.0){\rule[-0.200pt]{2.409pt}{0.400pt}}
\put(201.0,576.0){\rule[-0.200pt]{2.409pt}{0.400pt}}
\put(1429.0,576.0){\rule[-0.200pt]{2.409pt}{0.400pt}}
\put(201.0,591.0){\rule[-0.200pt]{2.409pt}{0.400pt}}
\put(1429.0,591.0){\rule[-0.200pt]{2.409pt}{0.400pt}}
\put(201.0,603.0){\rule[-0.200pt]{2.409pt}{0.400pt}}
\put(1429.0,603.0){\rule[-0.200pt]{2.409pt}{0.400pt}}
\put(201.0,614.0){\rule[-0.200pt]{4.818pt}{0.400pt}}
\put(181,614){\makebox(0,0)[r]{$10$}}
\put(1419.0,614.0){\rule[-0.200pt]{4.818pt}{0.400pt}}
\put(201.0,688.0){\rule[-0.200pt]{2.409pt}{0.400pt}}
\put(1429.0,688.0){\rule[-0.200pt]{2.409pt}{0.400pt}}
\put(201.0,732.0){\rule[-0.200pt]{2.409pt}{0.400pt}}
\put(1429.0,732.0){\rule[-0.200pt]{2.409pt}{0.400pt}}
\put(201.0,762.0){\rule[-0.200pt]{2.409pt}{0.400pt}}
\put(1429.0,762.0){\rule[-0.200pt]{2.409pt}{0.400pt}}
\put(201.0,786.0){\rule[-0.200pt]{2.409pt}{0.400pt}}
\put(1429.0,786.0){\rule[-0.200pt]{2.409pt}{0.400pt}}
\put(201.0,805.0){\rule[-0.200pt]{2.409pt}{0.400pt}}
\put(1429.0,805.0){\rule[-0.200pt]{2.409pt}{0.400pt}}
\put(201.0,822.0){\rule[-0.200pt]{2.409pt}{0.400pt}}
\put(1429.0,822.0){\rule[-0.200pt]{2.409pt}{0.400pt}}
\put(201.0,836.0){\rule[-0.200pt]{2.409pt}{0.400pt}}
\put(1429.0,836.0){\rule[-0.200pt]{2.409pt}{0.400pt}}
\put(201.0,849.0){\rule[-0.200pt]{2.409pt}{0.400pt}}
\put(1429.0,849.0){\rule[-0.200pt]{2.409pt}{0.400pt}}
\put(201.0,860.0){\rule[-0.200pt]{4.818pt}{0.400pt}}
\put(181,860){\makebox(0,0)[r]{$100$}}
\put(1419.0,860.0){\rule[-0.200pt]{4.818pt}{0.400pt}}
\put(201.0,123.0){\rule[-0.200pt]{0.400pt}{4.818pt}}
\put(201,82){\makebox(0,0){$10$}}
\put(201.0,840.0){\rule[-0.200pt]{0.400pt}{4.818pt}}
\put(279.0,123.0){\rule[-0.200pt]{0.400pt}{2.409pt}}
\put(279.0,850.0){\rule[-0.200pt]{0.400pt}{2.409pt}}
\put(382.0,123.0){\rule[-0.200pt]{0.400pt}{2.409pt}}
\put(382.0,850.0){\rule[-0.200pt]{0.400pt}{2.409pt}}
\put(435.0,123.0){\rule[-0.200pt]{0.400pt}{2.409pt}}
\put(435.0,850.0){\rule[-0.200pt]{0.400pt}{2.409pt}}
\put(460.0,123.0){\rule[-0.200pt]{0.400pt}{4.818pt}}
\put(460,82){\makebox(0,0){$100$}}
\put(460.0,840.0){\rule[-0.200pt]{0.400pt}{4.818pt}}
\put(538.0,123.0){\rule[-0.200pt]{0.400pt}{2.409pt}}
\put(538.0,850.0){\rule[-0.200pt]{0.400pt}{2.409pt}}
\put(641.0,123.0){\rule[-0.200pt]{0.400pt}{2.409pt}}
\put(641.0,850.0){\rule[-0.200pt]{0.400pt}{2.409pt}}
\put(694.0,123.0){\rule[-0.200pt]{0.400pt}{2.409pt}}
\put(694.0,850.0){\rule[-0.200pt]{0.400pt}{2.409pt}}
\put(719.0,123.0){\rule[-0.200pt]{0.400pt}{4.818pt}}
\put(719,82){\makebox(0,0){$1000$}}
\put(719.0,840.0){\rule[-0.200pt]{0.400pt}{4.818pt}}
\put(797.0,123.0){\rule[-0.200pt]{0.400pt}{2.409pt}}
\put(797.0,850.0){\rule[-0.200pt]{0.400pt}{2.409pt}}
\put(900.0,123.0){\rule[-0.200pt]{0.400pt}{2.409pt}}
\put(900.0,850.0){\rule[-0.200pt]{0.400pt}{2.409pt}}
\put(953.0,123.0){\rule[-0.200pt]{0.400pt}{2.409pt}}
\put(953.0,850.0){\rule[-0.200pt]{0.400pt}{2.409pt}}
\put(978.0,123.0){\rule[-0.200pt]{0.400pt}{4.818pt}}
\put(978,82){\makebox(0,0){$10000$}}
\put(978.0,840.0){\rule[-0.200pt]{0.400pt}{4.818pt}}
\put(1056.0,123.0){\rule[-0.200pt]{0.400pt}{2.409pt}}
\put(1056.0,850.0){\rule[-0.200pt]{0.400pt}{2.409pt}}
\put(1159.0,123.0){\rule[-0.200pt]{0.400pt}{2.409pt}}
\put(1159.0,850.0){\rule[-0.200pt]{0.400pt}{2.409pt}}
\put(1212.0,123.0){\rule[-0.200pt]{0.400pt}{2.409pt}}
\put(1212.0,850.0){\rule[-0.200pt]{0.400pt}{2.409pt}}
\put(1237.0,123.0){\rule[-0.200pt]{0.400pt}{4.818pt}}
\put(1237,82){\makebox(0,0){$100000$}}
\put(1237.0,840.0){\rule[-0.200pt]{0.400pt}{4.818pt}}
\put(1315.0,123.0){\rule[-0.200pt]{0.400pt}{2.409pt}}
\put(1315.0,850.0){\rule[-0.200pt]{0.400pt}{2.409pt}}
\put(1418.0,123.0){\rule[-0.200pt]{0.400pt}{2.409pt}}
\put(1418.0,850.0){\rule[-0.200pt]{0.400pt}{2.409pt}}
\put(201.0,123.0){\rule[-0.200pt]{0.400pt}{177.543pt}}
\put(201.0,123.0){\rule[-0.200pt]{298.234pt}{0.400pt}}
\put(1439.0,123.0){\rule[-0.200pt]{0.400pt}{177.543pt}}
\put(201.0,860.0){\rule[-0.200pt]{298.234pt}{0.400pt}}
\put(40,491){\makebox(0,0){$\e$}}
\put(820,21){\makebox(0,0){Degrees of freedom ${N}$}}
\put(1272,816){\makebox(0,0)[r]{Uniform refinement}}
\put(1292.0,816.0){\rule[-0.200pt]{24.090pt}{0.400pt}}
\put(391,763){\usebox{\plotpoint}}
\multiput(391.00,761.92)(2.066,-0.498){73}{\rule{1.742pt}{0.120pt}}
\multiput(391.00,762.17)(152.384,-38.000){2}{\rule{0.871pt}{0.400pt}}
\multiput(547.00,723.92)(1.423,-0.499){107}{\rule{1.235pt}{0.120pt}}
\multiput(547.00,724.17)(153.438,-55.000){2}{\rule{0.617pt}{0.400pt}}
\multiput(703.00,668.92)(1.262,-0.499){121}{\rule{1.106pt}{0.120pt}}
\multiput(703.00,669.17)(153.704,-62.000){2}{\rule{0.553pt}{0.400pt}}
\multiput(859.00,606.92)(1.150,-0.499){133}{\rule{1.018pt}{0.120pt}}
\multiput(859.00,607.17)(153.888,-68.000){2}{\rule{0.509pt}{0.400pt}}
\multiput(1015.00,538.92)(1.117,-0.499){137}{\rule{0.991pt}{0.120pt}}
\multiput(1015.00,539.17)(153.942,-70.000){2}{\rule{0.496pt}{0.400pt}}
\multiput(1171.00,468.92)(1.086,-0.499){141}{\rule{0.967pt}{0.120pt}}
\multiput(1171.00,469.17)(153.994,-72.000){2}{\rule{0.483pt}{0.400pt}}
\put(391,763){\raisebox{-.8pt}{\makebox(0,0){$\Diamond$}}}
\put(547,725){\raisebox{-.8pt}{\makebox(0,0){$\Diamond$}}}
\put(703,670){\raisebox{-.8pt}{\makebox(0,0){$\Diamond$}}}
\put(859,608){\raisebox{-.8pt}{\makebox(0,0){$\Diamond$}}}
\put(1015,540){\raisebox{-.8pt}{\makebox(0,0){$\Diamond$}}}
\put(1171,470){\raisebox{-.8pt}{\makebox(0,0){$\Diamond$}}}
\put(1327,398){\raisebox{-.8pt}{\makebox(0,0){$\Diamond$}}}
\put(1342,816){\raisebox{-.8pt}{\makebox(0,0){$\Diamond$}}}
\put(1272,775){\makebox(0,0)[r]{Adaptive refinement based on $\eta$}}
\multiput(1292,775)(20.756,0.000){5}{\usebox{\plotpoint}}
\put(1392,775){\usebox{\plotpoint}}
\put(391,763){\usebox{\plotpoint}}
\multiput(391,763)(20.166,-4.912){8}{\usebox{\plotpoint}}
\multiput(547,725)(13.442,-15.814){2}{\usebox{\plotpoint}}
\multiput(572,687)(13.566,-15.708){3}{\usebox{\plotpoint}}
\put(613.66,639.34){\usebox{\plotpoint}}
\multiput(620,633)(13.917,-15.398){3}{\usebox{\plotpoint}}
\multiput(667,581)(14.915,-14.434){2}{\usebox{\plotpoint}}
\multiput(698,551)(16.207,-12.966){3}{\usebox{\plotpoint}}
\multiput(733,523)(18.231,-9.920){3}{\usebox{\plotpoint}}
\multiput(801,486)(19.880,-5.964){3}{\usebox{\plotpoint}}
\multiput(851,471)(18.154,-10.061){4}{\usebox{\plotpoint}}
\multiput(934,425)(19.034,-8.276){4}{\usebox{\plotpoint}}
\multiput(1003,395)(18.406,-9.592){4}{\usebox{\plotpoint}}
\multiput(1074,358)(18.987,-8.384){4}{\usebox{\plotpoint}}
\multiput(1151,324)(18.155,-10.059){4}{\usebox{\plotpoint}}
\multiput(1225,283)(19.144,-8.020){4}{\usebox{\plotpoint}}
\multiput(1299,252)(18.356,-9.688){4}{\usebox{\plotpoint}}
\put(1371,214){\usebox{\plotpoint}}
\put(391,763){\makebox(0,0){$+$}}
\put(547,725){\makebox(0,0){$+$}}
\put(564,705){\makebox(0,0){$+$}}
\put(572,687){\makebox(0,0){$+$}}
\put(610,643){\makebox(0,0){$+$}}
\put(620,633){\makebox(0,0){$+$}}
\put(667,581){\makebox(0,0){$+$}}
\put(698,551){\makebox(0,0){$+$}}
\put(733,523){\makebox(0,0){$+$}}
\put(801,486){\makebox(0,0){$+$}}
\put(851,471){\makebox(0,0){$+$}}
\put(934,425){\makebox(0,0){$+$}}
\put(1003,395){\makebox(0,0){$+$}}
\put(1074,358){\makebox(0,0){$+$}}
\put(1151,324){\makebox(0,0){$+$}}
\put(1225,283){\makebox(0,0){$+$}}
\put(1299,252){\makebox(0,0){$+$}}
\put(1371,214){\makebox(0,0){$+$}}
\put(1342,775){\makebox(0,0){$+$}}
\put(201.0,123.0){\rule[-0.200pt]{0.400pt}{177.543pt}}
\put(201.0,123.0){\rule[-0.200pt]{298.234pt}{0.400pt}}
\put(1439.0,123.0){\rule[-0.200pt]{0.400pt}{177.543pt}}
\put(201.0,860.0){\rule[-0.200pt]{298.234pt}{0.400pt}}
\end{picture}

\medskip
{\bf Figure 5.1}. Example 2: Global error for the uniform and adaptive refinements, with $\omega=1.0$ %adaptive refinement with $\omega=4.0$
\end{center}
%
%%%%%%%%%%%%%%%%%%%%%%%%%%%%%%%%%%%%%%%%%%%%%%%%%%%%%%%%%
%
%%%%%%%%%%%%%%%%%%%%%%%%%%%%%%%%%%%%%%%%%%%%%%%%%%%%%%%%%
%
\begin{center}
\input{bbd-II-w10}

\medskip
{\bf Figure 5.2}. Example 2: Global error for the uniform and adaptive refinements, with $\omega=10.0$
\end{center}
%
%%%%%%%%%%%%%%%%%%%%%%%%%%%%%%%%%%%%%%%%%%%%%%%%%%%%%%%%%
%
%
%%%%%%%%%%%%%%%%%%%%%%%%%%%%%%%%%%%%%%%%%%%%%%%%%%%%%%%%%
%
\begin{center}
\input{bbd-II-w15}

\medskip
{\bf Figure 5.3}. Example 2: Global error for the uniform and adaptive refinements, with $\omega=15.0$
\end{center}

%
%%%%%%%%%%%%%%%%%%%%%%%%%%%%%%%%%%%%%%%%%%%%%%%%%%%%%%%%%
%
\begin{center}
% GNUPLOT: LaTeX picture
\setlength{\unitlength}{0.240900pt}
\ifx\plotpoint\undefined\newsavebox{\plotpoint}\fi
\sbox{\plotpoint}{\rule[-0.200pt]{0.400pt}{0.400pt}}%
\begin{picture}(1500,900)(0,0)
\sbox{\plotpoint}{\rule[-0.200pt]{0.400pt}{0.400pt}}%
\put(221.0,123.0){\rule[-0.200pt]{4.818pt}{0.400pt}}
\put(201,123){\makebox(0,0)[r]{$0.01$}}
\put(1419.0,123.0){\rule[-0.200pt]{4.818pt}{0.400pt}}
\put(221.0,234.0){\rule[-0.200pt]{2.409pt}{0.400pt}}
\put(1429.0,234.0){\rule[-0.200pt]{2.409pt}{0.400pt}}
\put(221.0,299.0){\rule[-0.200pt]{2.409pt}{0.400pt}}
\put(1429.0,299.0){\rule[-0.200pt]{2.409pt}{0.400pt}}
\put(221.0,345.0){\rule[-0.200pt]{2.409pt}{0.400pt}}
\put(1429.0,345.0){\rule[-0.200pt]{2.409pt}{0.400pt}}
\put(221.0,381.0){\rule[-0.200pt]{2.409pt}{0.400pt}}
\put(1429.0,381.0){\rule[-0.200pt]{2.409pt}{0.400pt}}
\put(221.0,410.0){\rule[-0.200pt]{2.409pt}{0.400pt}}
\put(1429.0,410.0){\rule[-0.200pt]{2.409pt}{0.400pt}}
\put(221.0,434.0){\rule[-0.200pt]{2.409pt}{0.400pt}}
\put(1429.0,434.0){\rule[-0.200pt]{2.409pt}{0.400pt}}
\put(221.0,456.0){\rule[-0.200pt]{2.409pt}{0.400pt}}
\put(1429.0,456.0){\rule[-0.200pt]{2.409pt}{0.400pt}}
\put(221.0,475.0){\rule[-0.200pt]{2.409pt}{0.400pt}}
\put(1429.0,475.0){\rule[-0.200pt]{2.409pt}{0.400pt}}
\put(221.0,491.0){\rule[-0.200pt]{4.818pt}{0.400pt}}
\put(201,491){\makebox(0,0)[r]{$0.1$}}
\put(1419.0,491.0){\rule[-0.200pt]{4.818pt}{0.400pt}}
\put(221.0,602.0){\rule[-0.200pt]{2.409pt}{0.400pt}}
\put(1429.0,602.0){\rule[-0.200pt]{2.409pt}{0.400pt}}
\put(221.0,667.0){\rule[-0.200pt]{2.409pt}{0.400pt}}
\put(1429.0,667.0){\rule[-0.200pt]{2.409pt}{0.400pt}}
\put(221.0,713.0){\rule[-0.200pt]{2.409pt}{0.400pt}}
\put(1429.0,713.0){\rule[-0.200pt]{2.409pt}{0.400pt}}
\put(221.0,749.0){\rule[-0.200pt]{2.409pt}{0.400pt}}
\put(1429.0,749.0){\rule[-0.200pt]{2.409pt}{0.400pt}}
\put(221.0,778.0){\rule[-0.200pt]{2.409pt}{0.400pt}}
\put(1429.0,778.0){\rule[-0.200pt]{2.409pt}{0.400pt}}
\put(221.0,803.0){\rule[-0.200pt]{2.409pt}{0.400pt}}
\put(1429.0,803.0){\rule[-0.200pt]{2.409pt}{0.400pt}}
\put(221.0,824.0){\rule[-0.200pt]{2.409pt}{0.400pt}}
\put(1429.0,824.0){\rule[-0.200pt]{2.409pt}{0.400pt}}
\put(221.0,843.0){\rule[-0.200pt]{2.409pt}{0.400pt}}
\put(1429.0,843.0){\rule[-0.200pt]{2.409pt}{0.400pt}}
\put(221.0,860.0){\rule[-0.200pt]{4.818pt}{0.400pt}}
\put(201,860){\makebox(0,0)[r]{$1$}}
\put(1419.0,860.0){\rule[-0.200pt]{4.818pt}{0.400pt}}
\put(221.0,123.0){\rule[-0.200pt]{0.400pt}{4.818pt}}
\put(221,82){\makebox(0,0){$10$}}
\put(221.0,840.0){\rule[-0.200pt]{0.400pt}{4.818pt}}
\put(298.0,123.0){\rule[-0.200pt]{0.400pt}{2.409pt}}
\put(298.0,850.0){\rule[-0.200pt]{0.400pt}{2.409pt}}
\put(399.0,123.0){\rule[-0.200pt]{0.400pt}{2.409pt}}
\put(399.0,850.0){\rule[-0.200pt]{0.400pt}{2.409pt}}
\put(451.0,123.0){\rule[-0.200pt]{0.400pt}{2.409pt}}
\put(451.0,850.0){\rule[-0.200pt]{0.400pt}{2.409pt}}
\put(476.0,123.0){\rule[-0.200pt]{0.400pt}{4.818pt}}
\put(476,82){\makebox(0,0){$100$}}
\put(476.0,840.0){\rule[-0.200pt]{0.400pt}{4.818pt}}
\put(553.0,123.0){\rule[-0.200pt]{0.400pt}{2.409pt}}
\put(553.0,850.0){\rule[-0.200pt]{0.400pt}{2.409pt}}
\put(654.0,123.0){\rule[-0.200pt]{0.400pt}{2.409pt}}
\put(654.0,850.0){\rule[-0.200pt]{0.400pt}{2.409pt}}
\put(706.0,123.0){\rule[-0.200pt]{0.400pt}{2.409pt}}
\put(706.0,850.0){\rule[-0.200pt]{0.400pt}{2.409pt}}
\put(731.0,123.0){\rule[-0.200pt]{0.400pt}{4.818pt}}
\put(731,82){\makebox(0,0){$1000$}}
\put(731.0,840.0){\rule[-0.200pt]{0.400pt}{4.818pt}}
\put(808.0,123.0){\rule[-0.200pt]{0.400pt}{2.409pt}}
\put(808.0,850.0){\rule[-0.200pt]{0.400pt}{2.409pt}}
\put(909.0,123.0){\rule[-0.200pt]{0.400pt}{2.409pt}}
\put(909.0,850.0){\rule[-0.200pt]{0.400pt}{2.409pt}}
\put(961.0,123.0){\rule[-0.200pt]{0.400pt}{2.409pt}}
\put(961.0,850.0){\rule[-0.200pt]{0.400pt}{2.409pt}}
\put(986.0,123.0){\rule[-0.200pt]{0.400pt}{4.818pt}}
\put(986,82){\makebox(0,0){$10000$}}
\put(986.0,840.0){\rule[-0.200pt]{0.400pt}{4.818pt}}
\put(1062.0,123.0){\rule[-0.200pt]{0.400pt}{2.409pt}}
\put(1062.0,850.0){\rule[-0.200pt]{0.400pt}{2.409pt}}
\put(1164.0,123.0){\rule[-0.200pt]{0.400pt}{2.409pt}}
\put(1164.0,850.0){\rule[-0.200pt]{0.400pt}{2.409pt}}
\put(1216.0,123.0){\rule[-0.200pt]{0.400pt}{2.409pt}}
\put(1216.0,850.0){\rule[-0.200pt]{0.400pt}{2.409pt}}
\put(1241.0,123.0){\rule[-0.200pt]{0.400pt}{4.818pt}}
\put(1241,82){\makebox(0,0){$100000$}}
\put(1241.0,840.0){\rule[-0.200pt]{0.400pt}{4.818pt}}
\put(1317.0,123.0){\rule[-0.200pt]{0.400pt}{2.409pt}}
\put(1317.0,850.0){\rule[-0.200pt]{0.400pt}{2.409pt}}
\put(1419.0,123.0){\rule[-0.200pt]{0.400pt}{2.409pt}}
\put(1419.0,850.0){\rule[-0.200pt]{0.400pt}{2.409pt}}
\put(221.0,123.0){\rule[-0.200pt]{0.400pt}{177.543pt}}
\put(221.0,123.0){\rule[-0.200pt]{293.416pt}{0.400pt}}
\put(1439.0,123.0){\rule[-0.200pt]{0.400pt}{177.543pt}}
\put(221.0,860.0){\rule[-0.200pt]{293.416pt}{0.400pt}}
\put(40,491){\makebox(0,0){$\e$}}
\put(830,21){\makebox(0,0){Degrees of freedom ${N}$}}
\put(1272,804){\makebox(0,0)[r]{Uniform refinement}}
\put(1292.0,804.0){\rule[-0.200pt]{24.090pt}{0.400pt}}
\put(408,699){\usebox{\plotpoint}}
\multiput(408.00,697.92)(1.323,-0.499){113}{\rule{1.155pt}{0.120pt}}
\multiput(408.00,698.17)(150.602,-58.000){2}{\rule{0.578pt}{0.400pt}}
\multiput(561.00,639.92)(1.188,-0.499){127}{\rule{1.048pt}{0.120pt}}
\multiput(561.00,640.17)(151.825,-65.000){2}{\rule{0.524pt}{0.400pt}}
\multiput(715.00,574.92)(1.111,-0.499){135}{\rule{0.987pt}{0.120pt}}
\multiput(715.00,575.17)(150.952,-69.000){2}{\rule{0.493pt}{0.400pt}}
\multiput(868.00,505.92)(1.087,-0.499){139}{\rule{0.968pt}{0.120pt}}
\multiput(868.00,506.17)(151.992,-71.000){2}{\rule{0.484pt}{0.400pt}}
\multiput(1022.00,434.92)(1.065,-0.499){141}{\rule{0.950pt}{0.120pt}}
\multiput(1022.00,435.17)(151.028,-72.000){2}{\rule{0.475pt}{0.400pt}}
\multiput(1175.00,362.92)(1.057,-0.499){143}{\rule{0.944pt}{0.120pt}}
\multiput(1175.00,363.17)(152.041,-73.000){2}{\rule{0.472pt}{0.400pt}}
\put(408,699){\raisebox{-.8pt}{\makebox(0,0){$\Diamond$}}}
\put(561,641){\raisebox{-.8pt}{\makebox(0,0){$\Diamond$}}}
\put(715,576){\raisebox{-.8pt}{\makebox(0,0){$\Diamond$}}}
\put(868,507){\raisebox{-.8pt}{\makebox(0,0){$\Diamond$}}}
\put(1022,436){\raisebox{-.8pt}{\makebox(0,0){$\Diamond$}}}
\put(1175,364){\raisebox{-.8pt}{\makebox(0,0){$\Diamond$}}}
\put(1329,291){\raisebox{-.8pt}{\makebox(0,0){$\Diamond$}}}
\put(1342,804){\raisebox{-.8pt}{\makebox(0,0){$\Diamond$}}}
\put(1272,763){\makebox(0,0)[r]{Adaptive refinement based on $\eta$}}
\multiput(1292,763)(20.756,0.000){5}{\usebox{\plotpoint}}
\put(1392,763){\usebox{\plotpoint}}
\put(408,699){\usebox{\plotpoint}}
\multiput(408,699)(19.408,-7.357){8}{\usebox{\plotpoint}}
\multiput(561,641)(19.150,-8.003){4}{\usebox{\plotpoint}}
\multiput(628,613)(16.926,-12.012){2}{\usebox{\plotpoint}}
\multiput(659,591)(19.935,-5.778){3}{\usebox{\plotpoint}}
\put(733.66,567.32){\usebox{\plotpoint}}
\multiput(748,558)(18.296,-9.801){3}{\usebox{\plotpoint}}
\multiput(804,528)(18.426,-9.554){2}{\usebox{\plotpoint}}
\put(839.28,503.35){\usebox{\plotpoint}}
\multiput(845,496)(18.768,-8.863){2}{\usebox{\plotpoint}}
\multiput(881,479)(19.387,-7.413){2}{\usebox{\plotpoint}}
\put(926.95,452.21){\usebox{\plotpoint}}
\multiput(928,451)(18.739,-8.923){2}{\usebox{\plotpoint}}
\put(981.91,423.37){\usebox{\plotpoint}}
\put(998.26,410.92){\usebox{\plotpoint}}
\multiput(1011,395)(18.021,-10.298){2}{\usebox{\plotpoint}}
\multiput(1039,379)(18.664,-9.080){2}{\usebox{\plotpoint}}
\put(1082.83,353.71){\usebox{\plotpoint}}
\multiput(1091,345)(18.659,-9.090){2}{\usebox{\plotpoint}}
\put(1135.81,322.21){\usebox{\plotpoint}}
\multiput(1153,311)(13.442,-15.814){2}{\usebox{\plotpoint}}
\put(1183.43,283.25){\usebox{\plotpoint}}
\multiput(1196,276)(18.259,-9.870){2}{\usebox{\plotpoint}}
\put(1236.77,251.73){\usebox{\plotpoint}}
\multiput(1248,239)(18.367,-9.667){2}{\usebox{\plotpoint}}
\multiput(1286,219)(17.601,-11.000){2}{\usebox{\plotpoint}}
\put(1319.79,191.91){\usebox{\plotpoint}}
\put(1334.97,178.22){\usebox{\plotpoint}}
\multiput(1352,168)(17.840,-10.608){3}{\usebox{\plotpoint}}
\put(1389,146){\usebox{\plotpoint}}
\put(408,699){\makebox(0,0){$+$}}
\put(561,641){\makebox(0,0){$+$}}
\put(628,613){\makebox(0,0){$+$}}
\put(659,591){\makebox(0,0){$+$}}
\put(728,571){\makebox(0,0){$+$}}
\put(748,558){\makebox(0,0){$+$}}
\put(804,528){\makebox(0,0){$+$}}
\put(831,514){\makebox(0,0){$+$}}
\put(845,496){\makebox(0,0){$+$}}
\put(881,479){\makebox(0,0){$+$}}
\put(915,466){\makebox(0,0){$+$}}
\put(928,451){\makebox(0,0){$+$}}
\put(970,431){\makebox(0,0){$+$}}
\put(995,415){\makebox(0,0){$+$}}
\put(1011,395){\makebox(0,0){$+$}}
\put(1039,379){\makebox(0,0){$+$}}
\put(1076,361){\makebox(0,0){$+$}}
\put(1091,345){\makebox(0,0){$+$}}
\put(1130,326){\makebox(0,0){$+$}}
\put(1153,311){\makebox(0,0){$+$}}
\put(1170,291){\makebox(0,0){$+$}}
\put(1196,276){\makebox(0,0){$+$}}
\put(1233,256){\makebox(0,0){$+$}}
\put(1248,239){\makebox(0,0){$+$}}
\put(1286,219){\makebox(0,0){$+$}}
\put(1310,204){\makebox(0,0){$+$}}
\put(1327,183){\makebox(0,0){$+$}}
\put(1352,168){\makebox(0,0){$+$}}
\put(1389,146){\makebox(0,0){$+$}}
\put(1342,763){\makebox(0,0){$+$}}
\put(221.0,123.0){\rule[-0.200pt]{0.400pt}{177.543pt}}
\put(221.0,123.0){\rule[-0.200pt]{293.416pt}{0.400pt}}
\put(1439.0,123.0){\rule[-0.200pt]{0.400pt}{177.543pt}}
\put(221.0,860.0){\rule[-0.200pt]{293.416pt}{0.400pt}}
\end{picture}

\medskip
{\bf Figure 5.4}. Example 3: Global error for the uniform and adaptive refinements, with $\omega=1.0$
\end{center}
%
%%%%%%%%%%%%%%%%%%%%%%%%%%%%%%%%%%%%%%%%%%%%%%%%%%%%%%%%%

\newpage
\begin{center}
{\bf Table 5.18}. Example 3: hybrid adaptive refinement with $\omega=1.0$
\begin{tabular}{|c||c|c||c|c||c|c||c|c||c|} 
\hline  
$N$ 	 & 	 $\e(u)$ 	 & 	     $r(u)$ 	 &$\e(\bm{\sigma})$ 	 & 	$r(\bm{\sigma})$ & 	 $\e$      	 & 	        $r$ 	 &  	  $\e/\eta$ 	 \\\hline  
    54 	 & 	 2.918e-01 	 & 	      ----- 	 & 	 2.214e-01 	 & 	       -----	 & 	 3.662e-01 	 & 	      ----- 	 &  	     0.1911 	 \\ 
   216 	 & 	 2.159e-01 	 & 	     0.4346 	 & 	 1.337e-01 	 & 	     0.7274 	 & 	 2.539e-01 	 & 	     0.5284 	 &  	     0.1911 	 \\ 
   396 	 & 	 1.852e-01 	 & 	     0.5055 	 & 	 1.053e-01 	 & 	     0.7868 	 & 	 2.131e-01 	 & 	     0.5788 	 &  	     0.1874 	 \\ 
   522 	 & 	 1.616e-01 	 & 	     0.9852 	 & 	 9.330e-02 	 & 	     0.8786 	 & 	 1.866e-01 	 & 	     0.9588 	 &  	     0.1961 	 \\ 
   972 	 & 	 1.455e-01 	 & 	     0.3389 	 & 	 7.630e-02 	 & 	     0.6472 	 & 	 1.643e-01 	 & 	     0.4106 	 &  	     0.2035 	 \\ 
  1170 	 & 	 1.346e-01 	 & 	     0.8402 	 & 	 6.947e-02 	 & 	     1.0118	 & 	 1.514e-01 	 & 	     0.8768 	 &  	     0.2102 	 \\ 
  1944 	 & 	 1.119e-01 	 & 	     0.7271 	 & 	 5.717e-02 	 & 	     0.7677 	 & 	 1.257e-01 	 & 	     0.7355 	 &  	     0.2206 	 \\ 
  2466 	 & 	 1.014e-01 	 & 	     0.8291 	 & 	 5.385e-02 	 & 	     0.5030 	 & 	 1.148e-01 	 & 	     0.7595 	 &  	     0.2147 	 \\ 
  2808 	 & 	 9.128e-02 	 & 	     1.6173 	 & 	 4.757e-02 	 & 	     1.9085 	 & 	 1.029e-01 	 & 	     1.6804 	 &  	     0.2214 	 \\ 
  3879 	 & 	 8.223e-02 	 & 	     0.6467 	 & 	 4.229e-02 	 & 	     0.7288 	 & 	 9.246e-02 	 & 	     0.6641 	 &  	     0.2175 	 \\ 
  5274 	 & 	 7.613e-02 	 & 	     0.5016 	 & 	 3.860e-02 	 & 	     0.5933 	 & 	 8.536e-02 	 & 	     0.5206 	 &  	     0.2166 	 \\ 
  5958 	 & 	 6.952e-02 	 & 	     1.4879 	 & 	 3.478e-02 	 & 	     1.7080 	 & 	 7.774e-02 	 & 	     1.5324 	 &  	     0.2276 	 \\ 
  8676 	 & 	 6.167e-02 	 & 	     0.6376 	 & 	 3.015e-02 	 & 	     0.7605 	 & 	 6.865e-02 	 & 	     0.6617 	 &  	     0.2288 	 \\ 
 10854 	 & 	 5.568e-02 	 & 	     0.9127 	 & 	 2.698e-02 	 & 	     0.9938 	 & 	 6.187e-02 	 & 	     0.9282 	 &  	     0.2266 	 \\ 
 12564 	 & 	 4.928e-02 	 & 	     1.6706 	 & 	 2.348e-02 	 & 	     1.8985 	 & 	 5.458e-02 	 & 	     1.7133 	 &  	     0.2317 	 \\ 
 16227 	 & 	 4.478e-02 	 & 	     0.7481 	 & 	 2.147e-02 	 & 	     0.6976 	 & 	 4.966e-02 	 & 	     0.7387 	 &  	     0.2248 	 \\ 
 22662 	 & 	 3.965e-02 	 & 	     0.7290 	 & 	 1.962e-02 	 & 	     0.5405 	 & 	 4.424e-02 	 & 	     0.6928 	 &  	     0.2248 	 \\ 
 25794 	 & 	 3.602e-02 	 & 	     1.4803 	 & 	 1.761e-02 	 & 	     1.6728 	 & 	 4.010e-02 	 & 	     1.5178 	 &  	     0.2300 	 \\ 
 36684 	 & 	 3.199e-02 	 & 	     0.6750 	 & 	 1.556e-02 	 & 	     0.7024 	 & 	 3.557e-02 	 & 	     0.6802 	 &  	     0.2313 	 \\ 
 45513 	 & 	 2.922e-02 	 & 	     0.8390 	 & 	 1.408e-02 	 & 	     0.9262 	 & 	 3.244e-02 	 & 	     0.8556 	 &  	     0.2313 	 \\ 
 52974 	 & 	 2.565e-02 	 & 	     1.7156 	 & 	 1.266e-02 	 & 	     1.4056 	 & 	 2.860e-02 	 & 	     1.6561 	 &  	     0.2385 	 \\ 
 67023 	 & 	 2.330e-02 	 & 	     0.8173 	 & 	 1.140e-02 	 & 	     0.8861 	 & 	 2.594e-02 	 & 	     0.8307 	 &  	     0.2320 	 \\ 
 93411 	 & 	 2.049e-02 	 & 	     0.7751 	 & 	 1.021e-02 	 & 	     0.6641 	 & 	 2.289e-02 	 & 	     0.7534 	 &  	     0.2306 	 \\ 
107127 	 & 	 1.853e-02 	 & 	     1.4647 	 & 	 9.185e-03 	 & 	     1.5486 	 & 	 2.068e-02 	 & 	     1.4813 	 &  	     0.2343 	 \\ 
151002 	 & 	 1.632e-02 	 & 	     0.7391 	 & 	 8.116e-03 	 & 	     0.7209 	 & 	 1.823e-02 	 & 	     0.7355 	 &  	     0.2359 	 \\ 
186462 	 & 	 1.494e-02 	 & 	     0.8394 	 & 	 7.313e-03 	 & 	     0.9875 	 & 	 1.664e-02 	 & 	     0.8684 	 &  	     0.2349 	 \\ 
217863 	 & 	 1.301e-02 	 & 	     1.7801 	 & 	 6.421e-03 	 & 	     1.6707 	 & 	 1.451e-02 	 & 	     1.7588 	 &  	     0.2417 	 \\ 
272781 	 & 	 1.190e-02 	 & 	     0.7941 	 & 	 5.880e-03 	 & 	     0.7830 	 & 	 1.327e-02 	 & 	     0.7919 	 &  	     0.2342 	 \\ 
381510 	 & 	 1.032e-02 	 & 	     0.8472 	 & 	 5.207e-03 	 & 	     0.7253 	 & 	 1.156e-02 	 & 	     0.8229 	 &  	     0.2338 	 \\ 
\hline                                                                             
\end{tabular}                                                                      
\end{center}                                                                       
%%%%%%%%%%%%%%%%%%%%%%%%%%%%%%%%%%%%%%%%%%%%%%%%%%%%%%%%%                          
%       
%
%%%%%%%%%%%%%%%%%%%%%%%%%%%%%%%%%%%%%%%%%%%%%%%%%%%%%%%%%
%
\begin{center}
% GNUPLOT: LaTeX picture
\setlength{\unitlength}{0.240900pt}
\ifx\plotpoint\undefined\newsavebox{\plotpoint}\fi
\sbox{\plotpoint}{\rule[-0.200pt]{0.400pt}{0.400pt}}%
\begin{picture}(1500,900)(0,0)
\sbox{\plotpoint}{\rule[-0.200pt]{0.400pt}{0.400pt}}%
\put(221.0,123.0){\rule[-0.200pt]{4.818pt}{0.400pt}}
\put(201,123){\makebox(0,0)[r]{$0.01$}}
\put(1419.0,123.0){\rule[-0.200pt]{4.818pt}{0.400pt}}
\put(221.0,208.0){\rule[-0.200pt]{2.409pt}{0.400pt}}
\put(1429.0,208.0){\rule[-0.200pt]{2.409pt}{0.400pt}}
\put(221.0,258.0){\rule[-0.200pt]{2.409pt}{0.400pt}}
\put(1429.0,258.0){\rule[-0.200pt]{2.409pt}{0.400pt}}
\put(221.0,294.0){\rule[-0.200pt]{2.409pt}{0.400pt}}
\put(1429.0,294.0){\rule[-0.200pt]{2.409pt}{0.400pt}}
\put(221.0,321.0){\rule[-0.200pt]{2.409pt}{0.400pt}}
\put(1429.0,321.0){\rule[-0.200pt]{2.409pt}{0.400pt}}
\put(221.0,343.0){\rule[-0.200pt]{2.409pt}{0.400pt}}
\put(1429.0,343.0){\rule[-0.200pt]{2.409pt}{0.400pt}}
\put(221.0,362.0){\rule[-0.200pt]{2.409pt}{0.400pt}}
\put(1429.0,362.0){\rule[-0.200pt]{2.409pt}{0.400pt}}
\put(221.0,379.0){\rule[-0.200pt]{2.409pt}{0.400pt}}
\put(1429.0,379.0){\rule[-0.200pt]{2.409pt}{0.400pt}}
\put(221.0,393.0){\rule[-0.200pt]{2.409pt}{0.400pt}}
\put(1429.0,393.0){\rule[-0.200pt]{2.409pt}{0.400pt}}
\put(221.0,406.0){\rule[-0.200pt]{4.818pt}{0.400pt}}
\put(201,406){\makebox(0,0)[r]{$0.1$}}
\put(1419.0,406.0){\rule[-0.200pt]{4.818pt}{0.400pt}}
\put(221.0,491.0){\rule[-0.200pt]{2.409pt}{0.400pt}}
\put(1429.0,491.0){\rule[-0.200pt]{2.409pt}{0.400pt}}
\put(221.0,541.0){\rule[-0.200pt]{2.409pt}{0.400pt}}
\put(1429.0,541.0){\rule[-0.200pt]{2.409pt}{0.400pt}}
\put(221.0,577.0){\rule[-0.200pt]{2.409pt}{0.400pt}}
\put(1429.0,577.0){\rule[-0.200pt]{2.409pt}{0.400pt}}
\put(221.0,604.0){\rule[-0.200pt]{2.409pt}{0.400pt}}
\put(1429.0,604.0){\rule[-0.200pt]{2.409pt}{0.400pt}}
\put(221.0,627.0){\rule[-0.200pt]{2.409pt}{0.400pt}}
\put(1429.0,627.0){\rule[-0.200pt]{2.409pt}{0.400pt}}
\put(221.0,646.0){\rule[-0.200pt]{2.409pt}{0.400pt}}
\put(1429.0,646.0){\rule[-0.200pt]{2.409pt}{0.400pt}}
\put(221.0,662.0){\rule[-0.200pt]{2.409pt}{0.400pt}}
\put(1429.0,662.0){\rule[-0.200pt]{2.409pt}{0.400pt}}
\put(221.0,677.0){\rule[-0.200pt]{2.409pt}{0.400pt}}
\put(1429.0,677.0){\rule[-0.200pt]{2.409pt}{0.400pt}}
\put(221.0,689.0){\rule[-0.200pt]{4.818pt}{0.400pt}}
\put(201,689){\makebox(0,0)[r]{$1$}}
\put(1419.0,689.0){\rule[-0.200pt]{4.818pt}{0.400pt}}
\put(221.0,775.0){\rule[-0.200pt]{2.409pt}{0.400pt}}
\put(1429.0,775.0){\rule[-0.200pt]{2.409pt}{0.400pt}}
\put(221.0,825.0){\rule[-0.200pt]{2.409pt}{0.400pt}}
\put(1429.0,825.0){\rule[-0.200pt]{2.409pt}{0.400pt}}
\put(221.0,860.0){\rule[-0.200pt]{2.409pt}{0.400pt}}
\put(1429.0,860.0){\rule[-0.200pt]{2.409pt}{0.400pt}}
\put(221.0,123.0){\rule[-0.200pt]{0.400pt}{4.818pt}}
\put(221,82){\makebox(0,0){$10$}}
\put(221.0,840.0){\rule[-0.200pt]{0.400pt}{4.818pt}}
\put(298.0,123.0){\rule[-0.200pt]{0.400pt}{2.409pt}}
\put(298.0,850.0){\rule[-0.200pt]{0.400pt}{2.409pt}}
\put(399.0,123.0){\rule[-0.200pt]{0.400pt}{2.409pt}}
\put(399.0,850.0){\rule[-0.200pt]{0.400pt}{2.409pt}}
\put(451.0,123.0){\rule[-0.200pt]{0.400pt}{2.409pt}}
\put(451.0,850.0){\rule[-0.200pt]{0.400pt}{2.409pt}}
\put(476.0,123.0){\rule[-0.200pt]{0.400pt}{4.818pt}}
\put(476,82){\makebox(0,0){$100$}}
\put(476.0,840.0){\rule[-0.200pt]{0.400pt}{4.818pt}}
\put(553.0,123.0){\rule[-0.200pt]{0.400pt}{2.409pt}}
\put(553.0,850.0){\rule[-0.200pt]{0.400pt}{2.409pt}}
\put(654.0,123.0){\rule[-0.200pt]{0.400pt}{2.409pt}}
\put(654.0,850.0){\rule[-0.200pt]{0.400pt}{2.409pt}}
\put(706.0,123.0){\rule[-0.200pt]{0.400pt}{2.409pt}}
\put(706.0,850.0){\rule[-0.200pt]{0.400pt}{2.409pt}}
\put(731.0,123.0){\rule[-0.200pt]{0.400pt}{4.818pt}}
\put(731,82){\makebox(0,0){$1000$}}
\put(731.0,840.0){\rule[-0.200pt]{0.400pt}{4.818pt}}
\put(808.0,123.0){\rule[-0.200pt]{0.400pt}{2.409pt}}
\put(808.0,850.0){\rule[-0.200pt]{0.400pt}{2.409pt}}
\put(909.0,123.0){\rule[-0.200pt]{0.400pt}{2.409pt}}
\put(909.0,850.0){\rule[-0.200pt]{0.400pt}{2.409pt}}
\put(961.0,123.0){\rule[-0.200pt]{0.400pt}{2.409pt}}
\put(961.0,850.0){\rule[-0.200pt]{0.400pt}{2.409pt}}
\put(986.0,123.0){\rule[-0.200pt]{0.400pt}{4.818pt}}
\put(986,82){\makebox(0,0){$10000$}}
\put(986.0,840.0){\rule[-0.200pt]{0.400pt}{4.818pt}}
\put(1062.0,123.0){\rule[-0.200pt]{0.400pt}{2.409pt}}
\put(1062.0,850.0){\rule[-0.200pt]{0.400pt}{2.409pt}}
\put(1164.0,123.0){\rule[-0.200pt]{0.400pt}{2.409pt}}
\put(1164.0,850.0){\rule[-0.200pt]{0.400pt}{2.409pt}}
\put(1216.0,123.0){\rule[-0.200pt]{0.400pt}{2.409pt}}
\put(1216.0,850.0){\rule[-0.200pt]{0.400pt}{2.409pt}}
\put(1241.0,123.0){\rule[-0.200pt]{0.400pt}{4.818pt}}
\put(1241,82){\makebox(0,0){$100000$}}
\put(1241.0,840.0){\rule[-0.200pt]{0.400pt}{4.818pt}}
\put(1317.0,123.0){\rule[-0.200pt]{0.400pt}{2.409pt}}
\put(1317.0,850.0){\rule[-0.200pt]{0.400pt}{2.409pt}}
\put(1419.0,123.0){\rule[-0.200pt]{0.400pt}{2.409pt}}
\put(1419.0,850.0){\rule[-0.200pt]{0.400pt}{2.409pt}}
\put(221.0,123.0){\rule[-0.200pt]{0.400pt}{177.543pt}}
\put(221.0,123.0){\rule[-0.200pt]{293.416pt}{0.400pt}}
\put(1439.0,123.0){\rule[-0.200pt]{0.400pt}{177.543pt}}
\put(221.0,860.0){\rule[-0.200pt]{293.416pt}{0.400pt}}
\put(40,491){\makebox(0,0){$\e$}}
\put(830,21){\makebox(0,0){Degrees of freedom ${N}$}}
\put(1272,755){\makebox(0,0)[r]{Uniform refinement}}
\put(1292.0,755.0){\rule[-0.200pt]{24.090pt}{0.400pt}}
\put(408,666){\usebox{\plotpoint}}
\multiput(408.00,664.92)(1.279,-0.499){117}{\rule{1.120pt}{0.120pt}}
\multiput(408.00,665.17)(150.675,-60.000){2}{\rule{0.560pt}{0.400pt}}
\multiput(561.00,604.92)(0.647,-0.499){235}{\rule{0.618pt}{0.120pt}}
\multiput(561.00,605.17)(152.718,-119.000){2}{\rule{0.309pt}{0.400pt}}
\multiput(715.00,485.92)(1.145,-0.499){131}{\rule{1.013pt}{0.120pt}}
\multiput(715.00,486.17)(150.897,-67.000){2}{\rule{0.507pt}{0.400pt}}
\multiput(868.00,418.92)(1.380,-0.499){109}{\rule{1.200pt}{0.120pt}}
\multiput(868.00,419.17)(151.509,-56.000){2}{\rule{0.600pt}{0.400pt}}
\multiput(1022.00,362.92)(1.371,-0.499){109}{\rule{1.193pt}{0.120pt}}
\multiput(1022.00,363.17)(150.524,-56.000){2}{\rule{0.596pt}{0.400pt}}
\multiput(1175.00,306.92)(1.380,-0.499){109}{\rule{1.200pt}{0.120pt}}
\multiput(1175.00,307.17)(151.509,-56.000){2}{\rule{0.600pt}{0.400pt}}
\put(408,666){\raisebox{-.8pt}{\makebox(0,0){$\Diamond$}}}
\put(561,606){\raisebox{-.8pt}{\makebox(0,0){$\Diamond$}}}
\put(715,487){\raisebox{-.8pt}{\makebox(0,0){$\Diamond$}}}
\put(868,420){\raisebox{-.8pt}{\makebox(0,0){$\Diamond$}}}
\put(1022,364){\raisebox{-.8pt}{\makebox(0,0){$\Diamond$}}}
\put(1175,308){\raisebox{-.8pt}{\makebox(0,0){$\Diamond$}}}
\put(1329,252){\raisebox{-.8pt}{\makebox(0,0){$\Diamond$}}}
\put(1342,755){\raisebox{-.8pt}{\makebox(0,0){$\Diamond$}}}
\put(1272,714){\makebox(0,0)[r]{Adaptive refinement based on $\eta$}}
\multiput(1292,714)(20.756,0.000){5}{\usebox{\plotpoint}}
\put(1392,714){\usebox{\plotpoint}}
\put(408,666){\usebox{\plotpoint}}
\multiput(408,666)(19.323,-7.578){8}{\usebox{\plotpoint}}
\multiput(561,606)(16.424,-12.691){10}{\usebox{\plotpoint}}
\multiput(715,487)(19.012,-8.326){8}{\usebox{\plotpoint}}
\multiput(868,420)(19.506,-7.093){8}{\usebox{\plotpoint}}
\put(1022.81,348.29){\usebox{\plotpoint}}
\put(1033.80,331.30){\usebox{\plotpoint}}
\put(1048.25,316.52){\usebox{\plotpoint}}
\put(1064.20,304.20){\usebox{\plotpoint}}
\put(1080.86,292.61){\usebox{\plotpoint}}
\put(1099.69,283.89){\usebox{\plotpoint}}
\multiput(1115,274)(19.271,-7.708){2}{\usebox{\plotpoint}}
\put(1155.90,257.99){\usebox{\plotpoint}}
\put(1172.91,246.78){\usebox{\plotpoint}}
\put(1190.92,236.95){\usebox{\plotpoint}}
\multiput(1200,233)(19.634,-6.732){2}{\usebox{\plotpoint}}
\put(1247.51,212.07){\usebox{\plotpoint}}
\multiput(1249,211)(19.026,-8.294){2}{\usebox{\plotpoint}}
\put(1304.14,186.28){\usebox{\plotpoint}}
\put(1320.59,174.01){\usebox{\plotpoint}}
\put(1337.96,163.18){\usebox{\plotpoint}}
\multiput(1352,157)(18.860,-8.665){2}{\usebox{\plotpoint}}
\put(1389,140){\usebox{\plotpoint}}
\put(408,666){\makebox(0,0){$+$}}
\put(561,606){\makebox(0,0){$+$}}
\put(715,487){\makebox(0,0){$+$}}
\put(868,420){\makebox(0,0){$+$}}
\put(1022,364){\makebox(0,0){$+$}}
\put(1022,351){\makebox(0,0){$+$}}
\put(1025,341){\makebox(0,0){$+$}}
\put(1032,334){\makebox(0,0){$+$}}
\put(1036,328){\makebox(0,0){$+$}}
\put(1052,313){\makebox(0,0){$+$}}
\put(1061,309){\makebox(0,0){$+$}}
\put(1067,300){\makebox(0,0){$+$}}
\put(1082,292){\makebox(0,0){$+$}}
\put(1106,281){\makebox(0,0){$+$}}
\put(1115,274){\makebox(0,0){$+$}}
\put(1145,262){\makebox(0,0){$+$}}
\put(1164,255){\makebox(0,0){$+$}}
\put(1177,243){\makebox(0,0){$+$}}
\put(1200,233){\makebox(0,0){$+$}}
\put(1235,221){\makebox(0,0){$+$}}
\put(1249,211){\makebox(0,0){$+$}}
\put(1288,194){\makebox(0,0){$+$}}
\put(1311,183){\makebox(0,0){$+$}}
\put(1327,168){\makebox(0,0){$+$}}
\put(1352,157){\makebox(0,0){$+$}}
\put(1389,140){\makebox(0,0){$+$}}
\put(1342,714){\makebox(0,0){$+$}}
\put(221.0,123.0){\rule[-0.200pt]{0.400pt}{177.543pt}}
\put(221.0,123.0){\rule[-0.200pt]{293.416pt}{0.400pt}}
\put(1439.0,123.0){\rule[-0.200pt]{0.400pt}{177.543pt}}
\put(221.0,860.0){\rule[-0.200pt]{293.416pt}{0.400pt}}
\end{picture}

\medskip
{\bf Figure 5.5}. Example 3: Global error for the uniform and adaptive refinements, with $\omega=10.0$
\end{center}
%

%
%%%%%%%%%%%%%%%%%%%%%%%%%%%%%%%%%%%%%%%%%%%%%%%%%%%%%%%%%
%
                                                                           
\newpage                                                                           
\begin{center}                                                                     
%                                         $\e_0(u)$       & 	 $r_0$     	 &        
{\bf Table 5.19}. Example 3: hybrid adaptive refinement with $\omega=10.0$         
\begin{tabular}{|c||c|c||c|c||c|c||c|}                                             
\hline                                                                             
$N$ 	 & 	 $\e(u)$ 	 & 	 $r(u)$ 	 &     $\e(\bm{\sigma})$ &      $r(\bm{\sigma})$ & 	   $\e$ 	 & 	        $r$ 	 & 	 $\e/\eta$ 	 \\\hline  
    54 	 & 	 5.118e-01 	 & 	      ----- 	 & 	 6.493e-01 	 & 	       -----	 & 	 8.268e-01 	 & 	      ----- 	 & 		     0.0747 	 \\ 
   216 	 & 	 3.680e-01 	 & 	     0.4760 	 & 	 3.521e-01 	 & 	     0.8830 	 & 	 5.093e-01 	 & 	     0.6990 	 & 		     0.0747 	 \\ 
   864 	 & 	 1.614e-01 	 & 	     1.1891 	 & 	 1.042e-01 	 & 	     1.7561 	 & 	 1.921e-01 	 & 	     1.4065 	 & 		     0.1592 	 \\ 
  3456 	 & 	 9.701e-02 	 & 	     0.7342 	 & 	 5.493e-02 	 & 	     0.9243 	 & 	 1.115e-01 	 & 	     0.7851 	 & 		     0.1997 	 \\ 
 13824 	 & 	 6.192e-02 	 & 	     0.6477 	 & 	 3.446e-02 	 & 	     0.6725 	 & 	 7.087e-02 	 & 	     0.6536 	 & 		     0.1818 	 \\ 
 13878 	 & 	 5.678e-02 	 & 	    44.5216 	 & 	 2.910e-02 	 & 	    86.7418 	 & 	 6.380e-02 	 & 	    53.8964 	 & 		     0.1705 	 \\ 
 14238 	 & 	 5.240e-02 	 & 	     6.2688 	 & 	 2.698e-02 	 & 	     5.9135 	 & 	 5.893e-02 	 & 	     6.1946 	 & 		     0.1616 	 \\ 
 15219 	 & 	 5.031e-02 	 & 	     1.2211 	 & 	 2.402e-02 	 & 	     3.4895 	 & 	 5.575e-02 	 & 	     1.6689 	 & 		     0.1585 	 \\ 
 15732 	 & 	 4.789e-02 	 & 	     2.9773 	 & 	 2.288e-02 	 & 	     2.9285 	 & 	 5.307e-02 	 & 	     2.9682 	 & 		     0.1714 	 \\ 
 18162 	 & 	 4.209e-02 	 & 	     1.7974 	 & 	 2.068e-02 	 & 	     1.4101 	 & 	 4.689e-02 	 & 	     1.7238 	 & 		     0.1802 	 \\ 
 19800 	 & 	 4.077e-02 	 & 	     0.7344 	 & 	 2.023e-02 	 & 	     0.5030 	 & 	 4.552e-02 	 & 	     0.6890 	 & 		     0.1806 	 \\ 
 20898 	 & 	 3.808e-02 	 & 	     2.5292 	 & 	 1.843e-02 	 & 	     3.4467 	 & 	 4.231e-02 	 & 	     2.7069 	 & 		     0.1943 	 \\ 
 23814 	 & 	 3.576e-02 	 & 	     0.9616 	 & 	 1.684e-02 	 & 	     1.3861 	 & 	 3.953e-02 	 & 	     1.0404 	 & 		     0.1935 	 \\ 
 29538 	 & 	 3.258e-02 	 & 	     0.8666 	 & 	 1.539e-02 	 & 	     0.8359 	 & 	 3.603e-02 	 & 	     0.8610 	 & 		     0.1944 	 \\ 
 32004 	 & 	 3.099e-02 	 & 	     1.2446 	 & 	 1.433e-02 	 & 	     1.7798 	 & 	 3.414e-02 	 & 	     1.3405 	 & 		     0.2024 	 \\ 
 42012 	 & 	 2.805e-02 	 & 	     0.7329 	 & 	 1.286e-02 	 & 	     0.7975 	 & 	 3.086e-02 	 & 	     0.7442 	 & 		     0.2085 	 \\ 
 49941 	 & 	 2.649e-02 	 & 	     0.6629 	 & 	 1.223e-02 	 & 	     0.5777 	 & 	 2.918e-02 	 & 	     0.6480 	 & 		     0.2102 	 \\ 
 56385 	 & 	 2.396e-02 	 & 	     1.6524 	 & 	 1.118e-02 	 & 	     1.4819 	 & 	 2.644e-02 	 & 	     1.6222 	 & 		     0.2209 	 \\ 
 69525 	 & 	 2.225e-02 	 & 	     0.7070 	 & 	 1.033e-02 	 & 	     0.7525 	 & 	 2.453e-02 	 & 	     0.7151 	 & 		     0.2186 	 \\ 
 95067 	 & 	 2.002e-02 	 & 	     0.6762 	 & 	 9.381e-03 	 & 	     0.6163 	 & 	 2.211e-02 	 & 	     0.6655 	 & 		     0.2198 	 \\ 
107829 	 & 	 1.850e-02 	 & 	     1.2523 	 & 	 8.711e-03 	 & 	     1.1776 	 & 	 2.045e-02 	 & 	     1.2388 	 & 		     0.2268 	 \\ 
152721 	 & 	 1.605e-02 	 & 	     0.8146 	 & 	 7.615e-03 	 & 	     0.7725 	 & 	 1.777e-02 	 & 	     0.8069 	 & 		     0.2315 	 \\ 
188208 	 & 	 1.471e-02 	 & 	     0.8365 	 & 	 7.045e-03 	 & 	     0.7451 	 & 	 1.631e-02 	 & 	     0.8196 	 & 		     0.2294 	 \\ 
218754 	 & 	 1.293e-02 	 & 	     1.7198 	 & 	 6.331e-03 	 & 	     1.4208 	 & 	 1.439e-02 	 & 	     1.6630 	 & 		     0.2369 	 \\ 
273780 	 & 	 1.184e-02 	 & 	     0.7842 	 & 	 5.865e-03 	 & 	     0.6810 	 & 	 1.321e-02 	 & 	     0.7641 	 & 		     0.2318 	 \\ 
382905 	 & 	 1.026e-02 	 & 	     0.8524 	 & 	 5.190e-03 	 & 	     0.7291 	 & 	 1.150e-02 	 & 	     0.8277 	 & 		     0.2323 	 \\ 
\hline  
\end{tabular} 
\end{center}

\begin{center}
% GNUPLOT: LaTeX picture
\setlength{\unitlength}{0.240900pt}
\ifx\plotpoint\undefined\newsavebox{\plotpoint}\fi
\sbox{\plotpoint}{\rule[-0.200pt]{0.400pt}{0.400pt}}%
\begin{picture}(1500,900)(0,0)
\sbox{\plotpoint}{\rule[-0.200pt]{0.400pt}{0.400pt}}%
\put(221.0,123.0){\rule[-0.200pt]{4.818pt}{0.400pt}}
\put(201,123){\makebox(0,0)[r]{$0.01$}}
\put(1419.0,123.0){\rule[-0.200pt]{4.818pt}{0.400pt}}
\put(221.0,193.0){\rule[-0.200pt]{2.409pt}{0.400pt}}
\put(1429.0,193.0){\rule[-0.200pt]{2.409pt}{0.400pt}}
\put(221.0,234.0){\rule[-0.200pt]{2.409pt}{0.400pt}}
\put(1429.0,234.0){\rule[-0.200pt]{2.409pt}{0.400pt}}
\put(221.0,263.0){\rule[-0.200pt]{2.409pt}{0.400pt}}
\put(1429.0,263.0){\rule[-0.200pt]{2.409pt}{0.400pt}}
\put(221.0,285.0){\rule[-0.200pt]{2.409pt}{0.400pt}}
\put(1429.0,285.0){\rule[-0.200pt]{2.409pt}{0.400pt}}
\put(221.0,304.0){\rule[-0.200pt]{2.409pt}{0.400pt}}
\put(1429.0,304.0){\rule[-0.200pt]{2.409pt}{0.400pt}}
\put(221.0,319.0){\rule[-0.200pt]{2.409pt}{0.400pt}}
\put(1429.0,319.0){\rule[-0.200pt]{2.409pt}{0.400pt}}
\put(221.0,333.0){\rule[-0.200pt]{2.409pt}{0.400pt}}
\put(1429.0,333.0){\rule[-0.200pt]{2.409pt}{0.400pt}}
\put(221.0,344.0){\rule[-0.200pt]{2.409pt}{0.400pt}}
\put(1429.0,344.0){\rule[-0.200pt]{2.409pt}{0.400pt}}
\put(221.0,355.0){\rule[-0.200pt]{4.818pt}{0.400pt}}
\put(201,355){\makebox(0,0)[r]{$0.1$}}
\put(1419.0,355.0){\rule[-0.200pt]{4.818pt}{0.400pt}}
\put(221.0,425.0){\rule[-0.200pt]{2.409pt}{0.400pt}}
\put(1429.0,425.0){\rule[-0.200pt]{2.409pt}{0.400pt}}
\put(221.0,466.0){\rule[-0.200pt]{2.409pt}{0.400pt}}
\put(1429.0,466.0){\rule[-0.200pt]{2.409pt}{0.400pt}}
\put(221.0,495.0){\rule[-0.200pt]{2.409pt}{0.400pt}}
\put(1429.0,495.0){\rule[-0.200pt]{2.409pt}{0.400pt}}
\put(221.0,517.0){\rule[-0.200pt]{2.409pt}{0.400pt}}
\put(1429.0,517.0){\rule[-0.200pt]{2.409pt}{0.400pt}}
\put(221.0,536.0){\rule[-0.200pt]{2.409pt}{0.400pt}}
\put(1429.0,536.0){\rule[-0.200pt]{2.409pt}{0.400pt}}
\put(221.0,551.0){\rule[-0.200pt]{2.409pt}{0.400pt}}
\put(1429.0,551.0){\rule[-0.200pt]{2.409pt}{0.400pt}}
\put(221.0,565.0){\rule[-0.200pt]{2.409pt}{0.400pt}}
\put(1429.0,565.0){\rule[-0.200pt]{2.409pt}{0.400pt}}
\put(221.0,576.0){\rule[-0.200pt]{2.409pt}{0.400pt}}
\put(1429.0,576.0){\rule[-0.200pt]{2.409pt}{0.400pt}}
\put(221.0,587.0){\rule[-0.200pt]{4.818pt}{0.400pt}}
\put(201,587){\makebox(0,0)[r]{$1$}}
\put(1419.0,587.0){\rule[-0.200pt]{4.818pt}{0.400pt}}
\put(221.0,657.0){\rule[-0.200pt]{2.409pt}{0.400pt}}
\put(1429.0,657.0){\rule[-0.200pt]{2.409pt}{0.400pt}}
\put(221.0,698.0){\rule[-0.200pt]{2.409pt}{0.400pt}}
\put(1429.0,698.0){\rule[-0.200pt]{2.409pt}{0.400pt}}
\put(221.0,727.0){\rule[-0.200pt]{2.409pt}{0.400pt}}
\put(1429.0,727.0){\rule[-0.200pt]{2.409pt}{0.400pt}}
\put(221.0,749.0){\rule[-0.200pt]{2.409pt}{0.400pt}}
\put(1429.0,749.0){\rule[-0.200pt]{2.409pt}{0.400pt}}
\put(221.0,768.0){\rule[-0.200pt]{2.409pt}{0.400pt}}
\put(1429.0,768.0){\rule[-0.200pt]{2.409pt}{0.400pt}}
\put(221.0,783.0){\rule[-0.200pt]{2.409pt}{0.400pt}}
\put(1429.0,783.0){\rule[-0.200pt]{2.409pt}{0.400pt}}
\put(221.0,797.0){\rule[-0.200pt]{2.409pt}{0.400pt}}
\put(1429.0,797.0){\rule[-0.200pt]{2.409pt}{0.400pt}}
\put(221.0,809.0){\rule[-0.200pt]{2.409pt}{0.400pt}}
\put(1429.0,809.0){\rule[-0.200pt]{2.409pt}{0.400pt}}
\put(221.0,819.0){\rule[-0.200pt]{4.818pt}{0.400pt}}
\put(201,819){\makebox(0,0)[r]{$10$}}
\put(1419.0,819.0){\rule[-0.200pt]{4.818pt}{0.400pt}}
\put(221.0,123.0){\rule[-0.200pt]{0.400pt}{4.818pt}}
\put(221,82){\makebox(0,0){$10$}}
\put(221.0,840.0){\rule[-0.200pt]{0.400pt}{4.818pt}}
\put(298.0,123.0){\rule[-0.200pt]{0.400pt}{2.409pt}}
\put(298.0,850.0){\rule[-0.200pt]{0.400pt}{2.409pt}}
\put(399.0,123.0){\rule[-0.200pt]{0.400pt}{2.409pt}}
\put(399.0,850.0){\rule[-0.200pt]{0.400pt}{2.409pt}}
\put(451.0,123.0){\rule[-0.200pt]{0.400pt}{2.409pt}}
\put(451.0,850.0){\rule[-0.200pt]{0.400pt}{2.409pt}}
\put(476.0,123.0){\rule[-0.200pt]{0.400pt}{4.818pt}}
\put(476,82){\makebox(0,0){$100$}}
\put(476.0,840.0){\rule[-0.200pt]{0.400pt}{4.818pt}}
\put(553.0,123.0){\rule[-0.200pt]{0.400pt}{2.409pt}}
\put(553.0,850.0){\rule[-0.200pt]{0.400pt}{2.409pt}}
\put(654.0,123.0){\rule[-0.200pt]{0.400pt}{2.409pt}}
\put(654.0,850.0){\rule[-0.200pt]{0.400pt}{2.409pt}}
\put(706.0,123.0){\rule[-0.200pt]{0.400pt}{2.409pt}}
\put(706.0,850.0){\rule[-0.200pt]{0.400pt}{2.409pt}}
\put(731.0,123.0){\rule[-0.200pt]{0.400pt}{4.818pt}}
\put(731,82){\makebox(0,0){$1000$}}
\put(731.0,840.0){\rule[-0.200pt]{0.400pt}{4.818pt}}
\put(808.0,123.0){\rule[-0.200pt]{0.400pt}{2.409pt}}
\put(808.0,850.0){\rule[-0.200pt]{0.400pt}{2.409pt}}
\put(909.0,123.0){\rule[-0.200pt]{0.400pt}{2.409pt}}
\put(909.0,850.0){\rule[-0.200pt]{0.400pt}{2.409pt}}
\put(961.0,123.0){\rule[-0.200pt]{0.400pt}{2.409pt}}
\put(961.0,850.0){\rule[-0.200pt]{0.400pt}{2.409pt}}
\put(986.0,123.0){\rule[-0.200pt]{0.400pt}{4.818pt}}
\put(986,82){\makebox(0,0){$10000$}}
\put(986.0,840.0){\rule[-0.200pt]{0.400pt}{4.818pt}}
\put(1062.0,123.0){\rule[-0.200pt]{0.400pt}{2.409pt}}
\put(1062.0,850.0){\rule[-0.200pt]{0.400pt}{2.409pt}}
\put(1164.0,123.0){\rule[-0.200pt]{0.400pt}{2.409pt}}
\put(1164.0,850.0){\rule[-0.200pt]{0.400pt}{2.409pt}}
\put(1216.0,123.0){\rule[-0.200pt]{0.400pt}{2.409pt}}
\put(1216.0,850.0){\rule[-0.200pt]{0.400pt}{2.409pt}}
\put(1241.0,123.0){\rule[-0.200pt]{0.400pt}{4.818pt}}
\put(1241,82){\makebox(0,0){$100000$}}
\put(1241.0,840.0){\rule[-0.200pt]{0.400pt}{4.818pt}}
\put(1317.0,123.0){\rule[-0.200pt]{0.400pt}{2.409pt}}
\put(1317.0,850.0){\rule[-0.200pt]{0.400pt}{2.409pt}}
\put(1419.0,123.0){\rule[-0.200pt]{0.400pt}{2.409pt}}
\put(1419.0,850.0){\rule[-0.200pt]{0.400pt}{2.409pt}}
\put(221.0,123.0){\rule[-0.200pt]{0.400pt}{177.543pt}}
\put(221.0,123.0){\rule[-0.200pt]{293.416pt}{0.400pt}}
\put(1439.0,123.0){\rule[-0.200pt]{0.400pt}{177.543pt}}
\put(221.0,860.0){\rule[-0.200pt]{293.416pt}{0.400pt}}
\put(40,491){\makebox(0,0){$\e$}}
\put(830,21){\makebox(0,0){Degrees of freedom ${N}$}}
\put(1272,818){\makebox(0,0)[r]{Uniform refinement}}
\put(1292.0,818.0){\rule[-0.200pt]{24.090pt}{0.400pt}}
\put(408,568){\usebox{\plotpoint}}
\multiput(408.58,568.00)(0.499,0.791){303}{\rule{0.120pt}{0.733pt}}
\multiput(407.17,568.00)(153.000,240.479){2}{\rule{0.400pt}{0.366pt}}
\multiput(561.58,805.93)(0.499,-1.102){305}{\rule{0.120pt}{0.981pt}}
\multiput(560.17,807.96)(154.000,-336.965){2}{\rule{0.400pt}{0.490pt}}
\multiput(715.00,469.92)(1.709,-0.498){87}{\rule{1.460pt}{0.120pt}}
\multiput(715.00,470.17)(149.970,-45.000){2}{\rule{0.730pt}{0.400pt}}
\multiput(868.00,424.92)(0.771,-0.499){197}{\rule{0.716pt}{0.120pt}}
\multiput(868.00,425.17)(152.514,-100.000){2}{\rule{0.358pt}{0.400pt}}
\multiput(1022.00,324.92)(1.506,-0.498){99}{\rule{1.300pt}{0.120pt}}
\multiput(1022.00,325.17)(150.302,-51.000){2}{\rule{0.650pt}{0.400pt}}
\multiput(1175.00,273.92)(1.682,-0.498){89}{\rule{1.439pt}{0.120pt}}
\multiput(1175.00,274.17)(151.013,-46.000){2}{\rule{0.720pt}{0.400pt}}
\put(408,568){\raisebox{-.8pt}{\makebox(0,0){$\Diamond$}}}
\put(561,810){\raisebox{-.8pt}{\makebox(0,0){$\Diamond$}}}
\put(715,471){\raisebox{-.8pt}{\makebox(0,0){$\Diamond$}}}
\put(868,426){\raisebox{-.8pt}{\makebox(0,0){$\Diamond$}}}
\put(1022,326){\raisebox{-.8pt}{\makebox(0,0){$\Diamond$}}}
\put(1175,275){\raisebox{-.8pt}{\makebox(0,0){$\Diamond$}}}
\put(1329,229){\raisebox{-.8pt}{\makebox(0,0){$\Diamond$}}}
\put(1342,818){\raisebox{-.8pt}{\makebox(0,0){$\Diamond$}}}
\put(1272,777){\makebox(0,0)[r]{Adaptive refinement based on $\eta$}}
\multiput(1292,777)(20.756,0.000){5}{\usebox{\plotpoint}}
\put(1392,777){\usebox{\plotpoint}}
\put(408,568){\usebox{\plotpoint}}
\multiput(408,568)(11.091,17.543){14}{\usebox{\plotpoint}}
\multiput(561,810)(8.584,-18.897){18}{\usebox{\plotpoint}}
\multiput(715,471)(19.912,-5.857){8}{\usebox{\plotpoint}}
\multiput(868,426)(17.407,-11.304){9}{\usebox{\plotpoint}}
\multiput(1022,326)(19.690,-6.563){8}{\usebox{\plotpoint}}
\put(1178.02,255.97){\usebox{\plotpoint}}
\put(1187.68,239.40){\usebox{\plotpoint}}
\put(1201.24,224.05){\usebox{\plotpoint}}
\put(1218.29,212.36){\usebox{\plotpoint}}
\put(1230.88,197.13){\usebox{\plotpoint}}
\put(1249.92,188.90){\usebox{\plotpoint}}
\put(1268.92,180.65){\usebox{\plotpoint}}
\put(1288.25,173.69){\usebox{\plotpoint}}
\put(1308.35,168.51){\usebox{\plotpoint}}
\put(1327.42,160.87){\usebox{\plotpoint}}
\put(1345.79,151.38){\usebox{\plotpoint}}
\multiput(1356,147)(19.457,-7.227){2}{\usebox{\plotpoint}}
\put(1391,134){\usebox{\plotpoint}}
\put(408,568){\makebox(0,0){$+$}}
\put(561,810){\makebox(0,0){$+$}}
\put(715,471){\makebox(0,0){$+$}}
\put(868,426){\makebox(0,0){$+$}}
\put(1022,326){\makebox(0,0){$+$}}
\put(1175,275){\makebox(0,0){$+$}}
\put(1175,266){\makebox(0,0){$+$}}
\put(1175,260){\makebox(0,0){$+$}}
\put(1178,256){\makebox(0,0){$+$}}
\put(1181,252){\makebox(0,0){$+$}}
\put(1182,243){\makebox(0,0){$+$}}
\put(1187,240){\makebox(0,0){$+$}}
\put(1196,232){\makebox(0,0){$+$}}
\put(1200,225){\makebox(0,0){$+$}}
\put(1213,215){\makebox(0,0){$+$}}
\put(1223,210){\makebox(0,0){$+$}}
\put(1229,198){\makebox(0,0){$+$}}
\put(1242,192){\makebox(0,0){$+$}}
\put(1265,183){\makebox(0,0){$+$}}
\put(1275,177){\makebox(0,0){$+$}}
\put(1303,170){\makebox(0,0){$+$}}
\put(1321,165){\makebox(0,0){$+$}}
\put(1335,156){\makebox(0,0){$+$}}
\put(1356,147){\makebox(0,0){$+$}}
\put(1391,134){\makebox(0,0){$+$}}
\put(1342,777){\makebox(0,0){$+$}}
\put(221.0,123.0){\rule[-0.200pt]{0.400pt}{177.543pt}}
\put(221.0,123.0){\rule[-0.200pt]{293.416pt}{0.400pt}}
\put(1439.0,123.0){\rule[-0.200pt]{0.400pt}{177.543pt}}
\put(221.0,860.0){\rule[-0.200pt]{293.416pt}{0.400pt}}
\end{picture}

\medskip
{\bf Figure 5.6}. Example 3: Global error for the uniform and adaptive refinements, with $\omega=15.0$
\end{center}
%
%%%%%%%%%%%%%%%%%%%%%%%%%%%%%%%%%%%%%%%%%%%%%%%%%%%%%%%%%

\newpage
\begin{center}
{\bf Table 5.20}. Example 3: hybrid adaptive refinement with $\omega=15.0$
\begin{tabular}{|c||c|c||c|c||c|c||c|} 
\hline  
$N$ 	 & 	 $\e(u)$ 	 & 	 $r(u)$ 	 & $\e(\bm{\sigma})$ 	 & 	 $r(\bm{\sigma})$ & 	 $\e$      	 & 	        $r$ 	 & 	 	  $\e/\eta$ 	 \\\hline  
    54 	 & 	 3.665e-01 	 & 	      ----- 	 & 	 7.394e-01 	 & 	      ----- 	 & 	 8.252e-01 	 & 	      ----- 	 & 	 	     0.0479 	 \\ 
   216 	 & 	 5.089e+00 	 & 	      ----- 	 & 	 7.604e+00 	 & 	      ----- 	 & 	 9.150e+00 	 & 	      ----- 	 & 	 	     0.0479 	 \\ 
   864 	 & 	 2.585e-01 	 & 	     4.2995 	 & 	 1.793e-01 	 & 	     5.4060 	 & 	 3.146e-01 	 & 	     4.8623 	 & 	 	     0.0992 	 \\ 
  3456 	 & 	 1.564e-01 	 & 	     0.7244 	 & 	 1.270e-01 	 & 	     0.4985 	 & 	 2.015e-01 	 & 	     0.6429 	 & 	 	     0.2020 	 \\ 
 13824 	 & 	 6.458e-02 	 & 	     1.2765 	 & 	 3.874e-02 	 & 	     1.7122 	 & 	 7.531e-02 	 & 	     1.4197 	 & 	 	     0.2791 	 \\ 
 55296 	 & 	 3.956e-02 	 & 	     0.7070 	 & 	 2.200e-02 	 & 	     0.8166 	 & 	 4.526e-02 	 & 	     0.7344 	 & 	 	     0.1810 	 \\ 
 55350 	 & 	 3.667e-02 	 & 	   155.5654 	 & 	 1.924e-02 	 & 	   274.7426 	 & 	 4.141e-02 	 & 	   182.4908 	 & 	 	     0.1600 	 \\ 
 55404 	 & 	 3.437e-02 	 & 	   132.4926 	 & 	 1.866e-02 	 & 	    62.6713 	 & 	 3.911e-02 	 & 	   117.0130 	 & 	 	     0.1529 	 \\ 
 56808 	 & 	 3.331e-02 	 & 	     2.5103 	 & 	 1.731e-02 	 & 	     6.0037 	 & 	 3.754e-02 	 & 	     3.2790 	 & 	 	     0.1474 	 \\ 
 58185 	 & 	 3.162e-02 	 & 	     4.3581 	 & 	 1.679e-02 	 & 	     2.5333 	 & 	 3.580e-02 	 & 	     3.9634 	 & 	 	     0.1558 	 \\ 
 59085 	 & 	 2.885e-02 	 & 	    11.9475 	 & 	 1.567e-02 	 & 	     8.9721 	 & 	 3.283e-02 	 & 	    11.2811 	 & 	 	     0.1683 	 \\ 
 61578 	 & 	 2.787e-02 	 & 	     1.6716 	 & 	 1.561e-02 	 & 	     0.2008 	 & 	 3.194e-02 	 & 	     1.3284 	 & 	 	     0.1670 	 \\ 
 66816 	 & 	 2.600e-02 	 & 	     1.6995 	 & 	 1.415e-02 	 & 	     2.4048 	 & 	 2.960e-02 	 & 	     1.8643 	 & 	 	     0.1736 	 \\ 
 68994 	 & 	 2.442e-02 	 & 	     3.9088 	 & 	 1.289e-02 	 & 	     5.8062 	 & 	 2.761e-02 	 & 	     4.3324 	 & 	 	     0.1837 	 \\ 
 77976 	 & 	 2.224e-02 	 & 	     1.5276 	 & 	 1.144e-02 	 & 	     1.9558 	 & 	 2.501e-02 	 & 	     1.6190 	 & 	 	     0.1882 	 \\ 
 85212 	 & 	 2.110e-02 	 & 	     1.1850 	 & 	 1.058e-02 	 & 	     1.7558 	 & 	 2.360e-02 	 & 	     1.3021 	 & 	 	     0.1892 	 \\ 
 90324 	 & 	 1.899e-02 	 & 	     3.6171 	 & 	 8.936e-03 	 & 	     5.7985 	 & 	 2.099e-02 	 & 	     4.0337 	 & 	 	     0.1986 	 \\ 
101547 	 & 	 1.802e-02 	 & 	     0.8974 	 & 	 8.402e-03 	 & 	     1.0516 	 & 	 1.988e-02 	 & 	     0.9251 	 & 	 	     0.1915 	 \\ 
125163 	 & 	 1.649e-02 	 & 	     0.8449 	 & 	 7.658e-03 	 & 	     0.8872 	 & 	 1.819e-02 	 & 	     0.8525 	 & 	 	     0.1943 	 \\ 
135999 	 & 	 1.553e-02 	 & 	     1.4469 	 & 	 7.239e-03 	 & 	     1.3554 	 & 	 1.714e-02 	 & 	     1.4307 	 & 	 	     0.2025 	 \\ 
175401 	 & 	 1.435e-02 	 & 	     0.6235 	 & 	 6.893e-03 	 & 	     0.3850 	 & 	 1.592e-02 	 & 	     0.5798 	 & 	 	     0.2095 	 \\ 
207495 	 & 	 1.364e-02 	 & 	     0.5987 	 & 	 6.761e-03 	 & 	     0.2306 	 & 	 1.523e-02 	 & 	     0.5279 	 & 	 	     0.2157 	 \\ 
233838 	 & 	 1.237e-02 	 & 	     1.6463 	 & 	 6.280e-03 	 & 	     1.2333 	 & 	 1.387e-02 	 & 	     1.5633 	 & 	 	     0.2292 	 \\ 
284310 	 & 	 1.139e-02 	 & 	     0.8454 	 & 	 5.567e-03 	 & 	     1.2338 	 & 	 1.267e-02 	 & 	     0.9227 	 & 	 	     0.2284 	 \\ 
388269 	 & 	 1.008e-02 	 & 	     0.7805 	 & 	 4.770e-03 	 & 	     0.9917 	 & 	 1.115e-02 	 & 	     0.8202 	 & 	 	     0.2255 	 \\ 
\hline  
\end{tabular} 
\end{center}

%%%%%%%%%%%%%%%%%%%%%%%%%%%%%%%%%%%%%%%%%%%%%%%%%%%%%%%%%
\begin{center}
\includegraphics[height=10cm,width=13cm]{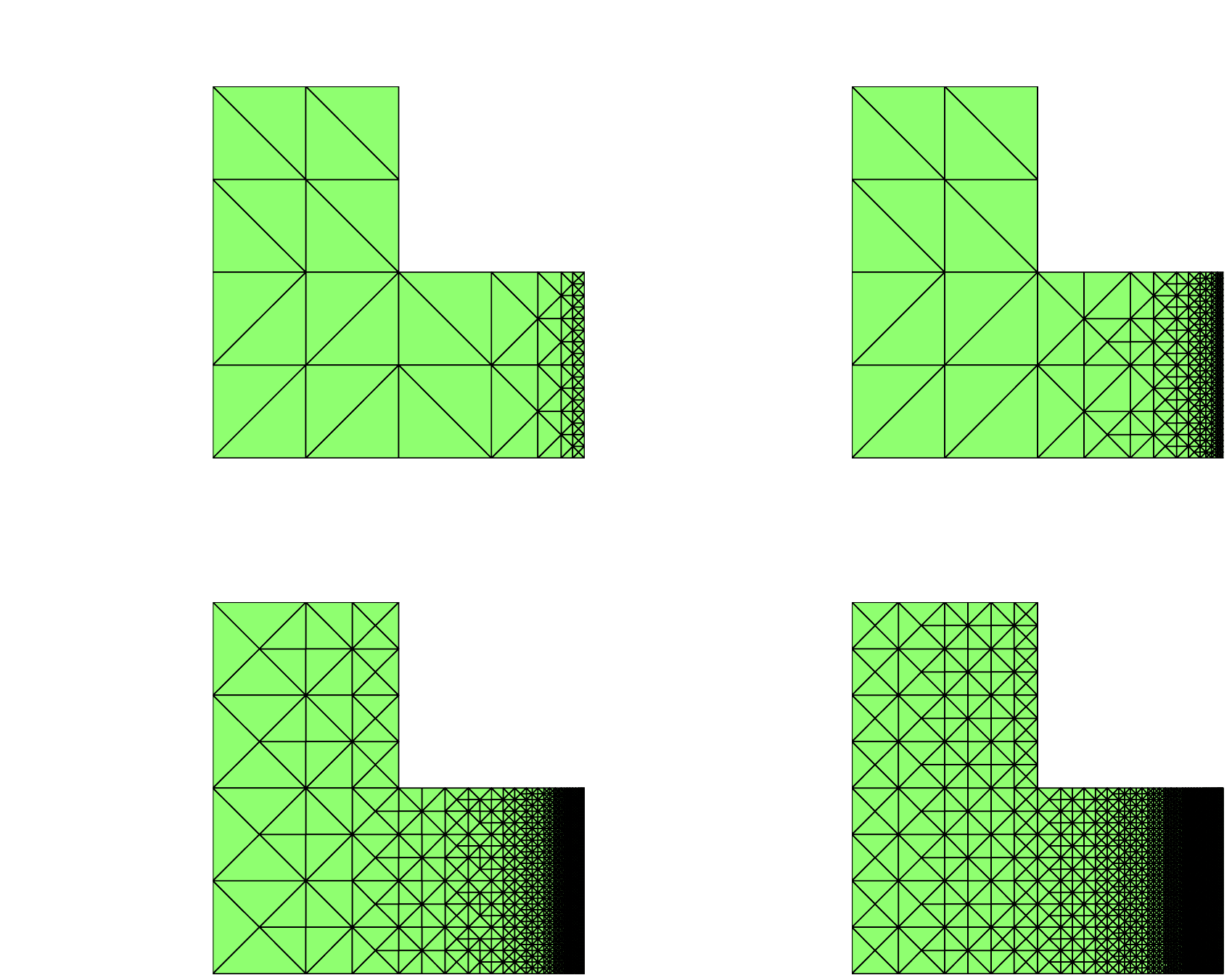}
\medskip

{\bf Figure 5.7:} {\rm Adapted intermediate meshes with 1134, 6732, 46521 and 173520 dof (Example 2, using $\omega=1.0$).}
\end{center}

\newpage

%%%%%%%%%%%%%%%%%%%%%%%%%%%%%%%%%%%%%%%%%%%%%%%%%%%%%%%%%
\begin{center}
\includegraphics[height=10cm,width=13cm]{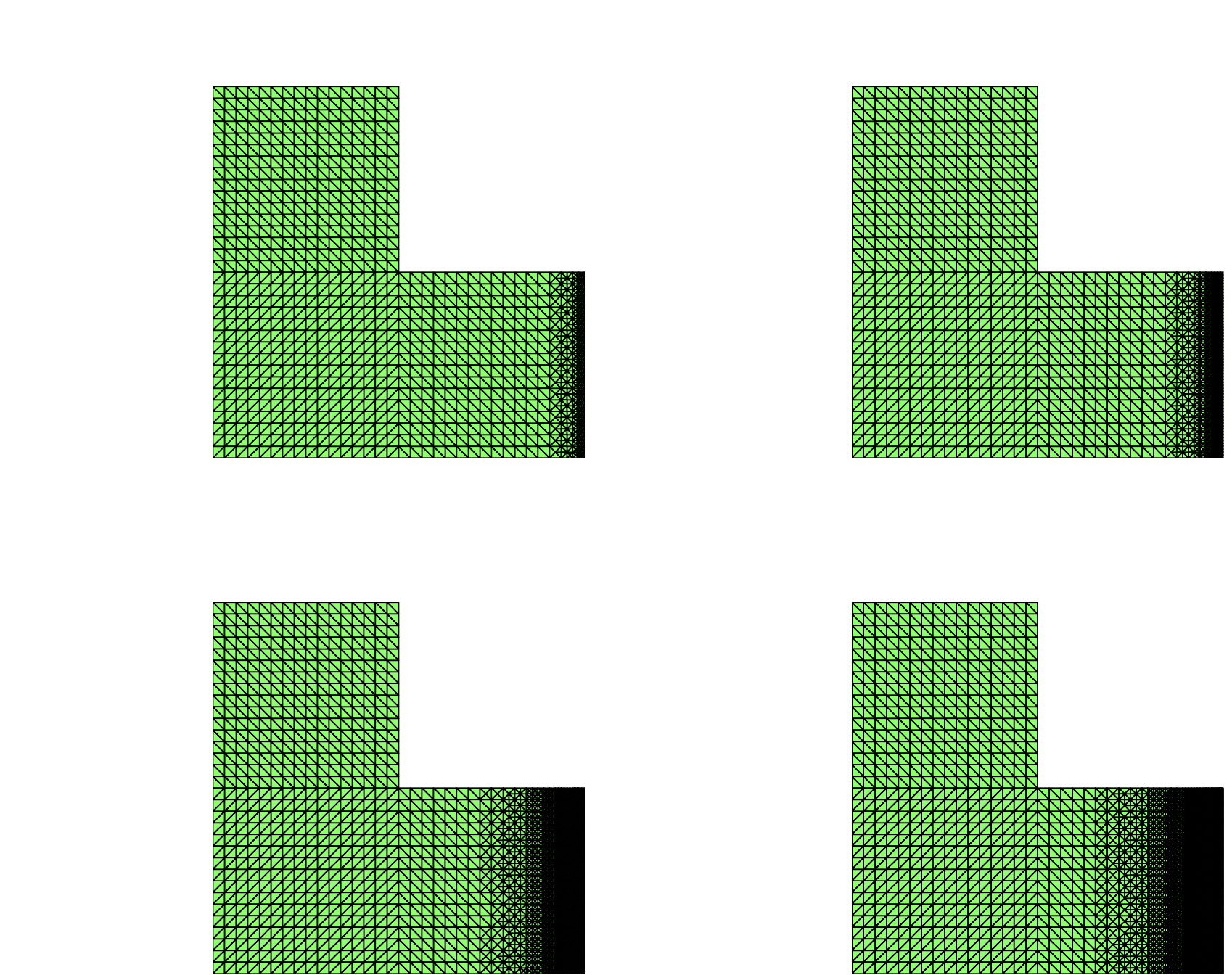}
\medskip

{\bf Figure 5.8:} {\rm Adapted intermediate meshes with 18558, 34290, 97893 and 171693 dof (Example 2, using $\omega=10.0$).}
\end{center}
%
%%%%%%%%%%%%%%%%%%%%%%%%%%%%%%%%%%%%%%%%%%%%%%%%%%%%%%%%
%
\begin{center}
\includegraphics[height=10cm,width=13cm]{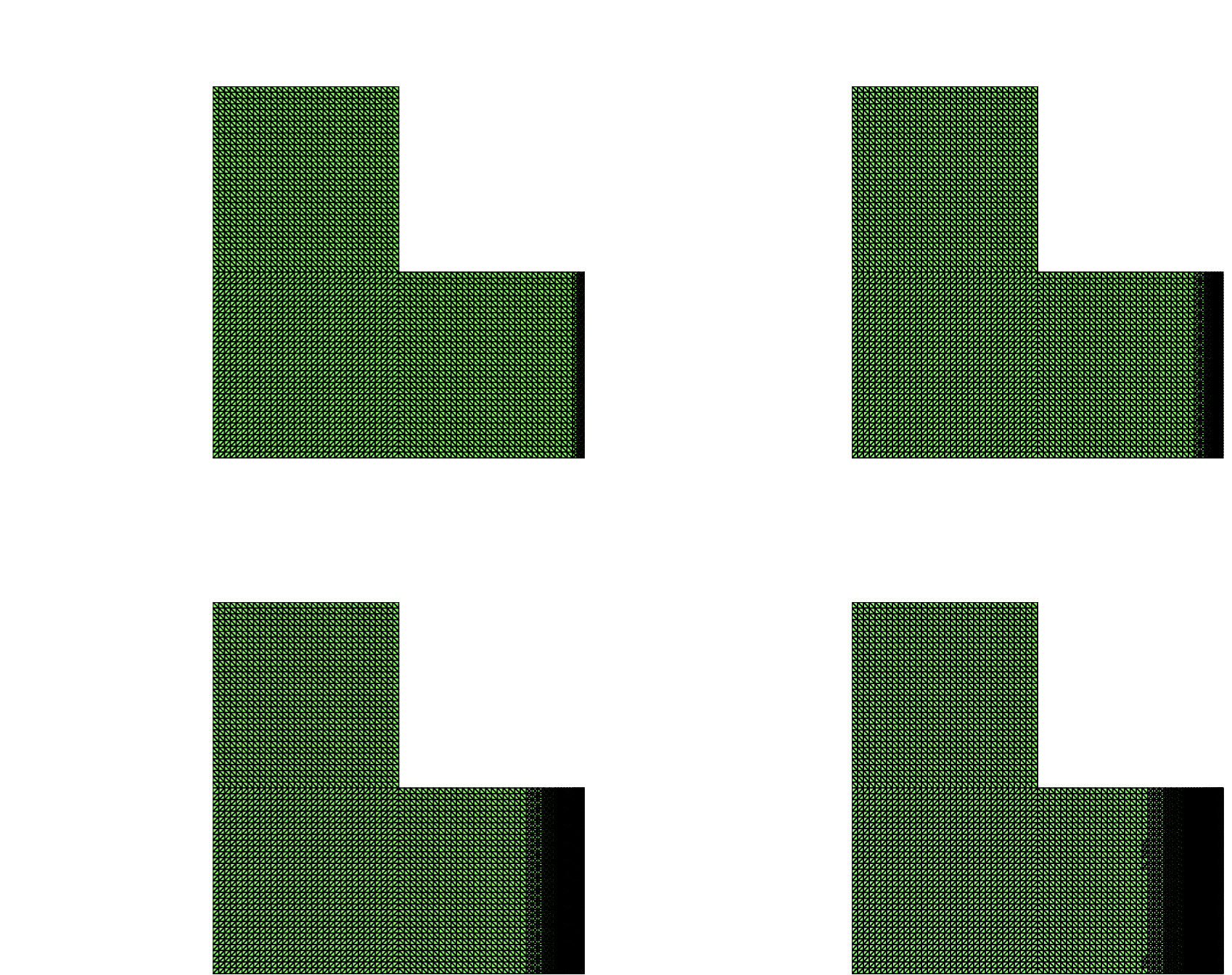}
\medskip

{\bf Figure 5.9:} {\rm Adapted intermediate meshes with 57888, 72018, 133092 and 214128 dof (Example 2, using $\omega=15.0$).}
\end{center}
%
%%%%%%%%%%%%%%%%%%%%%%%%%%%%%%%%%%%%%%%%%%%%%%%%%%%%%%%%%
%
\begin{center}
\includegraphics[height=10cm,width=13cm]{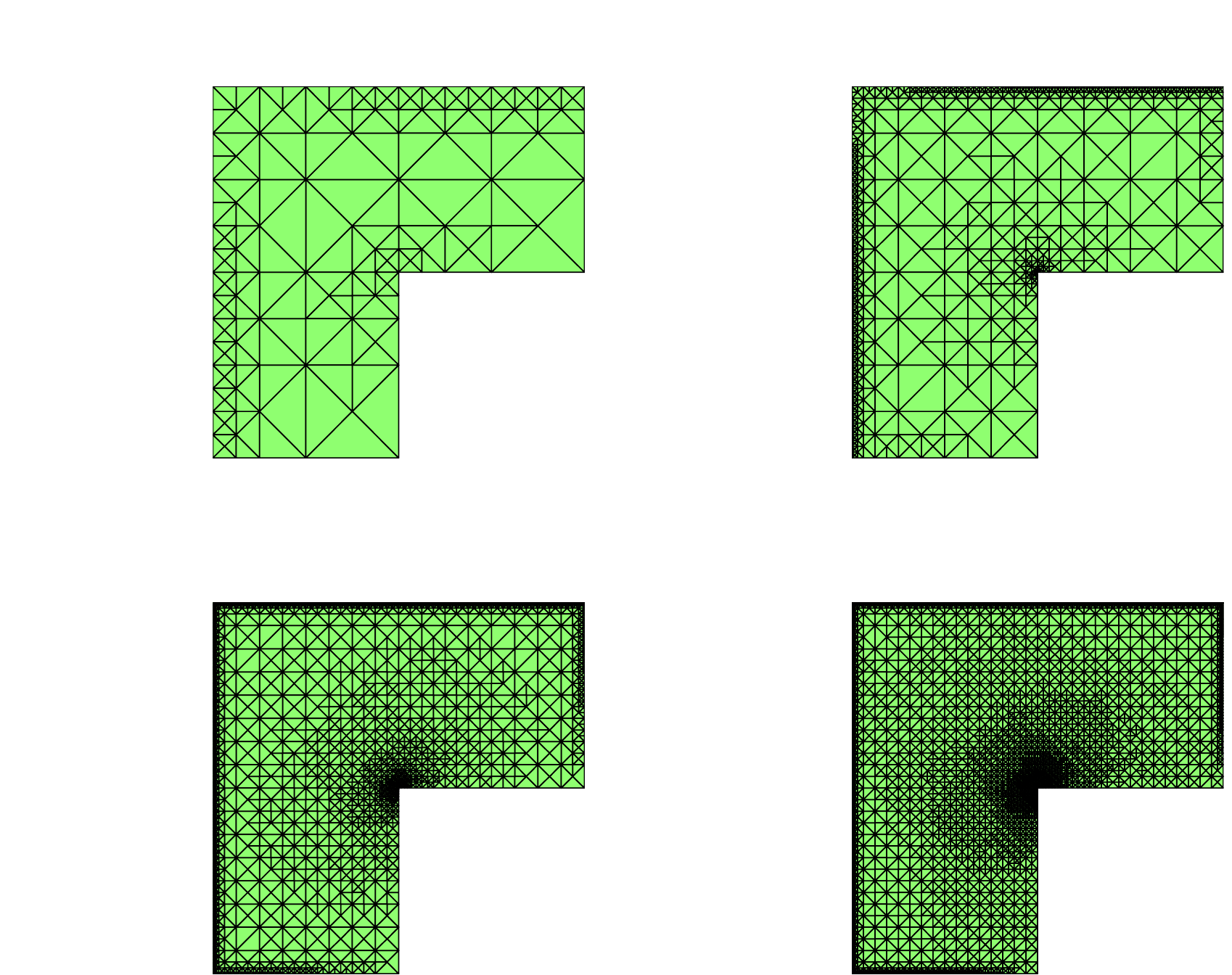}
\medskip

{\bf Figure 5.10:} {\rm Adapted intermediate meshes with 1944, 10854, 52974 and 151002 dof (Example 3, using $\omega=1.0$).}
\end{center}

%
%%%%%%%%%%%%%%%%%%%%%%%%%%%%%%%%%%%%%%%%%%%%%%%%%%%%%%%%%
%
%%%%%%%%%%%%%%%%%%%%%%%%%%%%%%%%%%%%%%%%%%%%%%%%%%%%%%%%%
\begin{center}
\includegraphics[height=10cm,width=13cm]{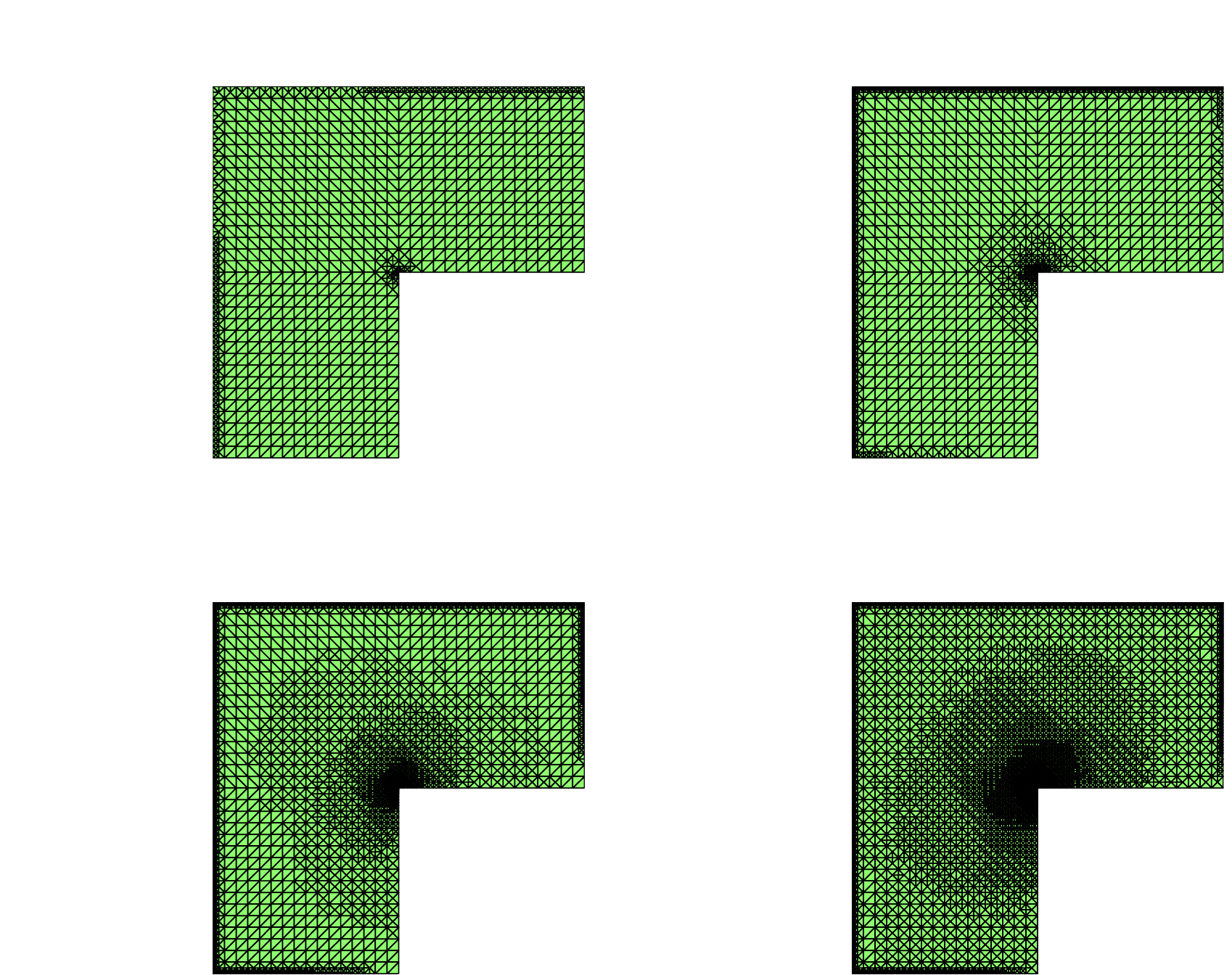}
\medskip

{\bf Figure 5.11:} {\rm Adapted intermediate meshes with 18162, 42012, 107829 and 218754 dof (Example 3, using $\omega=10.0$).}
\end{center}
%
%%%%%%%%%%%%%%%%%%%%%%%%%%%%%%%%%%%%%%%%%%%%%%%%%%%%%%%%
%
%\begin{center}
%\includegraphics[height=10cm,width=13cm]{Meshes-Eje3_w_15.eps}
%\medskip
%
%{\bf Figure 5.12:} {\rm Adapted intermediate meshes with 59085, 101547, 175401 and 388269 dof (Example 3, using $\omega=15.0$).}
%\end{center}
%%%%%%%%%%%%%%%%%%%%%%%%%%%%%%%%%%%%%%%%%%%%%%%%%%%%%%%%%%%%%%%%%%%%%%%%%%%%%%%%
%%%%%%%%%%%%
\subsection*{Acknowledgements}
This work was done during short visits of  R. Bustinza and
T.P. Barrios to the Departamento de
Ingenier\'{\i}a Matem\'atica e Inform\'atica, Universidad P\'ublica de Navarra,
 Campus Tudela, Spain. They  wish to thank professors V. Dom\'\i nguez
and
 R. Ortega, both from this University, for the kind
hospitality.
%The authors thanks professor , from Universidad
%Cat\'olica de la Sant\'\i sima Concepci\'on, Chile, for his valuable
%discussions
%on the development of the current work.

%\noindent
%The second autor\ednote{VD} is partially supported by Project  MTM2010-21037.
%%%%%%%%%%%%%%%%%%%%%%%%%%%%%%%%%%%%%%%%%%%%%%%%%%%%%%%%%%%%%%%%%%%%%%%%%%%%%%%%%%%%%%%%%%%%

\end{document}